%% file: christoffel-l1.tex
\documentclass[10pt,reqno]{amsart}
\pdfoutput=1

\usepackage{amsmath,amssymb,graphicx,bbm}
\usepackage{amsthm,verbatim,xcolor}
\usepackage{mathrsfs}

\usepackage[norelsize,lined,boxed,linesnumbered,commentsnumbered]{algorithm2e}

\usepackage[footnotesize,bf]{caption}
\usepackage[left=1.25in,right=1.25in,top=1in]{geometry}

\usepackage{acncommands}
\usepackage{subcaption,enumitem} 

\usepackage{array,multirow,booktabs} 
 \graphicspath{{./figures/}}

\newcommand{\vertiii}[1]{{\left\vert\kern-0.25ex\left\vert\kern-0.25ex\left\vert #1 
    \right\vert\kern-0.25ex\right\vert\kern-0.25ex\right\vert}}
\newcommand{\brvsamp}[1]{\brv^{(#1)}}

\newcommand{\naive}{na\"{i}ve}

\def\rv{Z}  
\def\rvd{z}  
\def\brv{Z} 
\def\brvset{\V{Z}} 
\def\dom{{I_{\brv}}} 
\newcommand{\inprod}[3]{\left( #1, #2 \right)_{#3}}
\def\bset{\Lambda} 
\newcommand{\bs}[1]{\boldsymbol{#1}}
\def\bsalpha{\bs{\alpha}}  
\def\bsbeta{\bs{\beta}} 
\def\bseta{\bs{\eta}} 
\def\bsi{\bs{i}} 
\def\bsPsi{\bs{\Psi}} 
\newcommand{\norm}[2]{\left\lVert #1 \right\rVert_{#2}} 
\def\boldf{\mathbf{f}} 
\def\bsf{\bs{f}} 
\def\vand{\boldsymbol{\Phi}} 
\def\coef{\boldsymbol{\alpha}} 
\def\stoptol{\varepsilon} 
\def\bW{\mathbf{W}} 
\def\bR{\mathbf{R}} 
\def\bS{\mathbf{S}} 
\def\bsc{\bs{c}} 
\def\by{y} 
\def\bz{z} 
\def\pbwt{w} 
\def\eqmeas{\mu} 
\def\wtmat{\sqrt{\V{W}}}

\usepackage{chngcntr}

\newcounter{parentnumber}

\makeatletter
\newenvironment{subtheorem}[1]{%
  \counterwithin*{#1}{parentnumber}
  \def\subtheoremcounter{#1}%
  \refstepcounter{#1}%
  \protected@edef\theparentnumber{\csname the#1\endcsname}%
  \setcounter{parentnumber}{\value{#1}}%
  \setcounter{#1}{0}%
  \expandafter\def\csname the#1\endcsname{\theparentnumber.\Alph{#1}}%
  \ignorespaces
  \def\foobar{#1} 
}{%
  \setcounter{\subtheoremcounter}{\value{parentnumber}}%
  \counterwithout*{\foobar}{parentnumber} 
  \ignorespacesafterend
}
\makeatother
\usepackage{hyperref}
\usepackage{cleveref}
\crefname{theorem}{Theorem}{Theorems}

\title[A generalized sampling and preconditioning scheme for sparse PCE approximation]{A generalized sampling and preconditioning scheme for sparse approximation of polynomial chaos expansions}
 \author{John D. Jakeman}
 \thanks{John D.Jakeman. Computer Science Research Institute, Sandia National Laboratories, 1450 Innovation Parkway, SE, Albuquerque, NM 87123}
 \author{Akil Narayan}
 \thanks{Akil Narayan. Scientific Computing and Imaging (SCI) Institute and Mathematics Department, University of Utah, 72 S Central Campus Drive, Salt Lake City, UT 84112. A.~Narayan was partially supported by AFOSR FA9550-15-1-0467 and DARPA N660011524053.}
 \author{Tao Zhou}
 \thanks{Tao Zhou. Institute of Computational Mathematics and the Chinese Academy of Sciences, Beijing, China. T.~Zhou work was supported the National Natural Science Foundation of China (Award Nos. 91130003 and 11571351).}

\begin{document}

\begin{abstract}
In this paper we propose an algorithm for recovering sparse 
orthogonal polynomials using stochastic collocation. Our approach is motivated by the desire to use
generalized polynomial chaos expansions (PCE) to quantify uncertainty in models subject to uncertain 
input parameters. The standard sampling approach for recovering sparse polynomials is to use Monte Carlo (MC)
sampling of the density of orthogonality. However MC methods result in poor function recovery when the
polynomial degree is high.
Here we propose a general algorithm that can be applied to any admissible weight function on a bounded domain 
and a wide class of exponential weight functions defined on unbounded domains. 
Our proposed algorithm samples with respect to the weighted equilibrium measure of the parametric domain, 
and subsequently solves a preconditioned $\ell^1$-minimization problem, where the weights of the diagonal 
preconditioning matrix are given by evaluations of the Christoffel function. We present theoretical analysis to motivate the algorithm, 
and numerical results that show our method is superior to standard Monte Carlo methods 
in many situations of interest. Numerical examples are also provided that demonstrate that 
our proposed Christoffel Sparse Approximation algorithm leads to comparable or improved 
accuracy even when compared with Legendre and Hermite specific algorithms.
\end{abstract}

\maketitle

\input{introduction}
\input{background}
\input{csa}
\input{results}

\input{conclusions}

\section*{Acknowledgments}
We gratefully acknowledge helpful discussions with Dr. Alireza Doostan and Dr. Jerrad Hampton.

This work was supported by the U.S. Department of Energy, Office of Science, Office 
of Advanced Scientific Computing Research, Applied Mathematics program. Sandia National 
Laboratories is a multi-program laboratory managed and operated by Sandia Corporation, 
a wholly owned subsidiary of Lockheed Martin Corporation, for the U.S. Department of Energy 
National Nuclear Security Administration under contract DE-AC04-94AL85000.

\input{proofs}

\bibliographystyle{abbrv}
\bibliography{references}

\end{document}

%% file: introduction.tex
\section{Introduction}\label{sec:introduction}
Quantifying the effect of uncertain model parameters on model output
is essential to building the confidence of stakeholders in the predictions of that model.
When a simulation model is computationally expensive to run,  
building an approximation of the response of the model output to variations in the model input,
is often an efficient means of quantifying parametric uncertainty. 
In this paper we
consider the polynomial approximation of a function (model)  $f(\brv): \R^d \rightarrow \R$
where $\brv=(\rv_1,\ldots,\rv_d)$ is a finite-dimensional random
variable with associated probability density function $\pbwt(\rvd)$.
Specifically we will approximate 
$f$ with a generalized Polynomial Chaos Expansion (PCE) which consists of 
an polynomial basis whose elements are orthogonal under the weight
$w$~\cite{Ghanem_book_1991,Xiu_K_SISC_2002}. 

The stochastic Galerkin~\cite{Ghanem_book_1991,Xiu_K_SISC_2002} 
and stochastic collocation~\cite{Babuska_NT_SNA_2007,Xiu_H_SISC_2005}
methods are the two main approaches for approximating PCE coefficients. 
In this paper we focus on stochastic collocation because it allows the
computational model to be treated as a black box. Stochastic collocation proceeds in two steps: (i)
running the computational model with a set of realizations of the random parameters $\{\brvsamp{i}\}_i$;
and (ii) constructing an approximation of corresponding model output $\hat{f}\approx f$.
Pseudo-spectral projection~\cite{Conrad_M_SISC_2013,Constantine_EP_CMAME_2012}, 
sparse grid interpolation~\cite{Barthelmann_NR_ACM_2000,Buzzard_RESS_2012,Jakeman_W_JCP_2015}, 
orthogonal least interpolation~\cite{Narayan_X_SISC_2012}, 
least squares~\cite{Migliorati_NSTSISC_2013,Tang_Z_SISC_2014}
are stochastic collocation methods which have effectively used polynomial 
approximations in many situations. 

Recently compressed sensing via $\ell^1$-minimization 
techniques~\cite{Candes_T_ITIEEE_2006,Candes_RT_CPAM_2006,Donoho_ITIEEE_2006,Donoho_ET_ITIEEE_2006} 
have been shown to be an effective means of approximating PCE coefficients from small number of 
function samples~\cite{Blatman_S_JCP_2011,Doostan_O_JCP_2011,Jakeman_ES_JCP_2014,Peng_HD_JCP_2014,Yang_K_JCP_2013}. 
These methods are most effective when the number of non-zero terms
in the PCE approximation of the model output is small (i.e. sparsity) or the magnitude of the 
PCE coefficients decays rapidly (i.e. compressibility). 
The efficacy of sparse polynomial approximation, however, is heavily reliant on the sampling strategy
used to generate the random variable ensemble $\{\brvsamp{i}\}_i$. The most common approach is to draw 
samples from the probability measure $\pbwt$ of the random variables, however the accuracy 
of the function approximation decreases as the polynomial degree increases.
To improve sparse recovery for high-degree polynomials the authors of~\cite{Rauhut_W_JAT_2012} proposed sampling 
from the Chebyshev distribution and applying an appropriate preconditioning to the polynomial 
Vandermonde-type matrix. This method, which we will here-after refer to as the asymptotic sampling method for bounded variables, however can only be applied to polynomials which are orthogonal 
to bounded random variables and moreover, theoretical and 
numerical results~\cite{Yan_GX_IJUQ_2012} have shown that the accuracy of the procedure proposed 
degrades with increasing parameter dimension. 

Recently the authors of~\cite{Hampton_D_JCP_2015} developed a sampling strategy that attempts
to overcome the limitations of the probabilistic sampling and Chebyshev sampling methods. Their so called coherence optimal sampling
strategy, can be applied to a large class of orthogonal polynomials and performs well for low and 
high-degree polynomials and low-and high-dimensional parameter dimensions. The scheme uses
Markov Chain Monte Carlo (MCMC) sampling to draw samples from a measure that attempts to minimize
the coherence parameter of the orthogonal polynomial system

In this paper we present another general sampling strategy for accurate sparse recovery 
of orthogonal polynomials, which we coin that we
call Christoffel Sparse Approximation (CSA). 
The CSA algorithm is applicable for bounded and
unbounded domains, with tensor-product or more general non-tensor-product
weights and domains. The algorithm is based on a similar algorithm for discrete least-squares approximation that we introduced in \cite{Narayan_JJ_MC}.

The essential idea of the CSA algorithm is to sample from a certain equilibrium measure that is associated with the random parameter density and state space. 
We use the Christoffel function to generate preconditioning weights which are then used to solve a standard weighted $\ell^1$-minimization problem. In Section~\ref{sec:christoffel-sparse-approximation} we present this algorithm, the Christoffel Sparse Approximation (CSA) method and give various formula for the sampling density.

In section~\ref{sec:1d-analysis} we prove that the CSA-based algorithm can successfully recover univariate sparse and compressible solutions. We show that for the bounded unvariate domains, we can accomplish this with the optimal $M \gtrsim s$ samples. For unbounded univariate domains, we pay an additional penalty, requiring $M \gtrsim s n^{2/3}$, where $n$ is the maximum polynomial degree in the basis. 
Although we only present theoretical results for univariate $s$-term approximation using CSA, the numerical results presented in Section~\ref{sec:results} demonstrate the CSA algorithm performs comparably or better than existing $\ell^1$-minimization algorithms in a number of practical settings, including high-dimensional multivariate settings.

Finally, we note that our results are useful outside of the context of this paper: For example, the 3 parts of Theorem \ref{thm:bound-csa} are used in quantifying sampling count criterion for $\ell^1$ recovery when randomly sub-sampling from a tensor-product Gauss quadrature grid. This idea, using our results, is explored in \cite{guo_subsampling_2015}, and builds on the idea presented in \cite{tang_subsampled_2014}.

%% file: background.tex
\section{Background}\label{sec:background}
\subsection{Polynomial chaos expansions}
Polynomial Chaos methods represent both the model inputs $\Theta=(\theta_1,\ldots,\theta_{\tilde{d}})$ 
and model output $f(\Theta)$ as an expansion of orthonormal polynomials of random variables 
$\brv=(\rv_1,\ldots,\rv_d)$. Specifically we represent the random inputs 
 as
\begin{equation}
\label{eq:pce-rv-integer-index}
\theta_n\approx\sum_{i=1}^{N_{\theta_n}}\beta_{i}\phi_{i}(\brv),\quad n\in[\tilde{d}]
\end{equation}
and the model output as
\begin{equation}
\label{eq:pce-integer-index}
f(\Theta(\brv))\approx\hat{f}=\sum_{i=1}^N\alpha_{i}\phi_{i}(\brv).
\end{equation}
where for $N\in\mathbb{N}$ we use the notation $[N]=\{1,\ldots,N\}$.
We refer to~\eqref{eq:pce-rv-integer-index} and~\eqref{eq:pce-integer-index} as polynomial chaos expansions (PCE). 
The PCE basis functions $\{\phi_{i}(\brv)\}$ are tensor products of orthonormal polynomials 
which are chosen to be 
orthonormal with respect to the distribution $\pbwt(\brv)$ of the random vector $\brv$.
That is 
\[
\inprod{\phi_{i}(\brv)}{\phi_{j}(\brv)}{\pbwt} = \int_{\dom} \phi_{i}(\brv)\phi_{j}(\brv)\pbwt(\brv) d\brv= \delta_{ij}
\]
where $\dom$ is the range of the random variables.

The random variable (germ) $\brv$ of the PCE is typically related to 
the distribution of the input variables. For example, if the 
one-dimensional input variable $\theta$ is uniform on $[a,b]$ then 
$\rv$ is also chosen to be uniform on [-1,1] and $\phi$ are chosen 
to be Legendre polynomials such that $\theta=\beta_1 + \beta_2\rv=(b+a)/2+\rv (b-a)/2$. 
For simplicity and without loss of
generality, we will assume that $\brv$ has the same distribution as
 $\Theta$ and thus we can use the two variables interchangeably (up to a linear transformation which we will ignore).

 Any function $f\in L^2(\pbwt)$ can be represented by an infinite PCE, and this expansion converges in the $L^2$ sense under relatively mild conditions on the distribution of $\brv$ \cite{ernst_convergence_2012}. However in practice an infinite PCE must be truncated to a form like \eqref{eq:pce-integer-index}. The most common approach is to set a degree $n$ and retain only the multivariate polynomials of degree at most $n$. 
Rewriting~\eqref{eq:pce-integer-index} using the typical multi-dimensional index notation
\begin{equation}
\label{eq:pce-multi-index}
f\approx\hat{f}=\sum_{\bsi\in\bset}\alpha_{\bsi}\phi_{\bsi}(\brv)
\end{equation}
the total degree basis of degree $n$ is given by
\begin{equation}\label{eq:hyperbolic-index-set} 
 \bset = \bset_{n} = \{\phi_{\bsi} : \norm{\bsi}{1} \le n\},\quad \bsi = (i_1,\ldots,i_d)
\end{equation}
Here we denote the space of polynomials of degree at most $n$ by $P_n$.
The number of terms in this total degree basis is
\[
 \text{card}\; \bset_{n} \equiv N = { d+n \choose d }
\]
The rate of convergence is dependent on the regularity of the response surface. If $f$ 
is analytical with respect to the random variables then~\eqref{eq:pce-integer-index} 
converges exponentially with the polynomial degree in $L^2(\pbwt)$-sense~\cite{Bieri_S_CMAME_2009}.

\subsection{$\ell^1$-minimization}\label{sec:compressed-sensing}
The coefficients of a polynomial chaos expansion can be approximated 
effectively using $\ell^1$-minimization. Specifically,
given a small set of $M$ unstructured realizations  $\brvset=\{\brvsamp{1},\ldots,\brvsamp{M}\}$, with 
corresponding model outputs $\boldf=(f(\brvsamp{1}),\ldots,f(\brvsamp{M}))^T$, we would like to find
a solution that satisfies
\[
 \vand\coef \approx \boldf
\]
where $\coef=(\alpha_{\bsi_1},\ldots,\alpha_{\bsi_N})^T$ denotes the vector 
of PCE coefficients and $\vand\in\mathbb{R}^{M\times N}$ denotes the Vandermonde matrix with entries
$\vand_{ij} = \phi_j(\brvsamp{i}),\quad i\in[M],\; j\in[N]$. 

When the model $f$ is high-dimensional and computationally expensive, the number of model simulations 
that can be generated is typically much smaller than the number of unknown PCE coefficients, i.e $M\ll N$. 
Under these conditions, finding the PCE coefficients is ill-posed and we must 
impose some form of regularization to obtain a unique solution. One way to regularize is to enforce
sparsity in a PCE. 

A polynomial chaos expansion is defined as $s$-sparse when $\norm{\coef}{0} \le s$, 
i.e. the number of non-zero coefficients does not exceed $s$.
Under certain conditions, $\ell^1$-minimization provides a means of identifying {\it sparse} coefficient 
vectors from a limited amount of simulation data\footnote{In practice, not many simulation models will be truly sparse, 
but PCE are often {\it compressible}, that is the magnitude of the 
coefficients decays rapidly or alternatively most of the PCE variance is concentrated in a few terms. 
Compressible vectors are well-approximated by sparse vectors
and thus the coefficients of compressible PCE can also be recovered accurately using $\ell^1$-minimization.}.
Specifically $\ell^1$-minimization attempts to find the dominant PCE coefficients
 by solving the optimization problem
\begin{equation}
\label{eq:pre-condition-bpdn}
\coef^\star  = \argmin_{\coef}\; \|\coef\|_1\quad \text{such that}\quad \|\wtmat\vand\coef - \wtmat\boldf\|_2 \le \stoptol
\end{equation}
where $\wtmat\in{R}^{M\times M}$ is a diagonal matrix with entries chosen to enhance 
the recovery properties of $\ell^1$-minimization, and $\stoptol$ is a noise/tolerance that allows the data to slightly deviate from the PCE. This $\ell^1$-minimization problem is often referred to as Basis Pursuit Denoising (BPDN). The problem obtained by setting $\stoptol=0$, to enforce interpolation, is 
termed Basis Pursuit. 

\subsection{Near-optimal sparse recovery}\label{sec:background-cs-recovery}
For simplicity we first consider the situation when the function $f$ is assumed to be \textit{exactly} $s$-sparse in the PCE basis $\phi_{\bsi}$. In this case we take the noise tolerance in \eqref{eq:pre-condition-bpdn} to be zero, $\stoptol=0$ so that we are enforcing exact interpolation. The intent of the computational $\ell^1$ optimization problem \eqref{eq:pre-condition-bpdn} is to recover a sparse representation. Such a representation could be recovered by solving the more difficult analytical problem
\begin{align}\label{eq:pre-condition-l0}
  \coef^\dagger  = \argmin_{\coef}\; \|\coef\|_0\quad \text{such that}\quad \vand\coef = \boldf.
\end{align}
This $N P$-hard problem is not computationally feasible, so the above optimization is usually not performed. However, if the matrix $\vand$ satisfies the \textit{restricted isometry property} (RIP) \cite{Candes_T_ITIEEE_2006,candes_decoding_2005}, then the desired sparse solution $\coef^\dagger$ coincides with the computable minimal $\ell^1$-norm solution from \eqref{eq:pre-condition-bpdn}. 
One is then led to wonder how systems satisfying the RIP can be constructed. The RIP essentially states that $\vand$ has a limited ability to elongate or contract the $\ell^2$ norm of $s$-sparse vectors. This can be roughly interpreted as requiring (i) that the rows of $\vand$ are approximately orthonormal, and (ii) the entries of $\vand$ are not too concentrated in any subset of entries and instead that the total mass is roughly equidistributed. One expects that property (i) can be satisfied if one constructs $\vand$ by choosing the $\phi_{\bsi}$ as an orthonormal basis and sampling parameter values (rows) according to the measure of orthogonality. A quantifiable way to enforce (ii) is via the concept of \textit{coherence}. It was shown in \cite{Rahut_TFNMSR_2010} that if the $\phi_{\bsi}$ are an orthonormal basis, then if the number of rows $M$ of the system satisfies 
\begin{align}\label{eq:bounded-orthonormal-system-recovery-condition}
  M &\gtrsim (s \log^3 s) L, & L &= \max_{\bsi \in \bset} \left\| \phi_{\bsi}\right\|^2_{\infty},
\end{align}
then the system satisfies the RIP, and thus the computable $\ell^1$ solution from \eqref{eq:pre-condition-bpdn} (with $\stoptol = 0$) coincides with the $s$-sparse solution.
Above, the infinity norm $\|\cdot\|_\infty$ is taken over the support of the sampling measure. The parameter $L$ is called the \textit{mutual coherence}, and is a bound on the entry-wise maximum of the matrix $\vand$. 
Up to logarithmic factors this is the best one can hope for, and so constructing matrices that attain the RIP is an important task. 

Unfortunately, for almost all PCE expansions, the mutual coherence $L$ is not $\mathcal{O}(1)$, and usually grows as the maximum polynomial degree grows.\footnote{The notable exception is when the PCE basis corresponds to the Chebyshev polynomials, and in this case the basis elements have $L^\infty$ norms that are uniformly bounded with respect to the polynomial degree.} However, it was noted in \cite{Rauhut_W_JAT_2012} that for a particular PCE basis, the Legendre polynomial basis, the functions $\phi^2_{\bsi}$ admit an envelope function $1/(c \lambda(Z))$ such that $1/\lambda \in L^1$, where $c$ is a normalization constant. The idea presented in \cite{Rauhut_W_JAT_2012} is then to multiply the basis elements $\phi_{\bsi}$ by $\sqrt{\lambda}$ so that they become uniformly bounded with $\mathcal{O}(1)$ coherence parameter $L$. In other words, solve the problem
\begin{align}\label{eq:pre-condition-l1-interpolation}
  \coef^\star  = \argmin_{\coef}\; \|\coef\|_1\quad \text{such that}\quad \wtmat\vand\coef = \wtmat\boldf,
\end{align}
where $\wtmat$ is a diagonal matrix with the entries of given by the (inverse) envelope function, $(W)_{m,m} = c \lambda(Z^{(m)})$. This motivates the matrix $\wtmat$ introduced in \eqref{eq:pre-condition-bpdn}. In order to retain orthonormality of the rows of $\wtmat \vand$, a Monte Carlo sampling procedure cannot sample from the original orthogonality density $w$, but must instead sample from the biased density $w/\lambda$.

With $N = |\bset|$ the size of the dictionary, we exploit the idea above in this paper using the (inverse) envelope function
\begin{align}\label{eq:christoffel-function}
  (W)_{m,m} &= N \lambda_{\bset}(Z^{(m)}), & \lambda_{\bset} (Z) &= \frac{1}{\sum_{\bsi \in \bset} \phi_{\bsi}^2(Z)}
\end{align}
When the basis elements $\phi_{\bsi}$ correspond to a PCE, the function $\lambda(z)$ is known as the ($L^2$) Christoffel function from the theory of orthogonal polynomials. Note that in the context of the problem \eqref{eq:pre-condition-l1-interpolation}, the weight $\lambda$ defined by \eqref{eq:christoffel-function} means that the preconditioner $\wtmat$ is simply formed using inverse row-norms of $\vand$.

With the weight matrix $\wtmat$ specified, we need only to determine the biased density $w/\lambda$. If we use the index set $\Lambda = \Lambda_n$, then under mild conditions on $w$ and $I_Z$, there is a unique probability measure $\mu$ such that $w/\lambda \approx N \mathrm{d} \mu$, where $\approx$ becomes equality interpreted in the appropriate sense as $n\rightarrow \infty$, and $N = |\bset_n|$. This measure $\mu$ is the weighted pluripotential equilibrium measure \cite{berman_bergman_2009}.

This essentially completes a high-level description of the CSA algorithm: sample from an equilibrium measure, and precondition/weight with the Chrsitoffel function. The detailed algorithm is given in Section~\ref{sec:christoffel-sparse-approximation} where more precise formulas regarding sampling densities are shown. Section \ref{sec:1d-analysis} provides convergence results for the univariate case -- the multivariate case requires bounds on values of Christoffel-function-weighted polynomials, which we are unable to currently provide. Nevertheless, our numerical results in Section \ref{sec:results} show that the CSA algorithm is competitive in various situations both with a standard unweighted Monte Carlo approach (sampling from the density $w$), and with other approaches in the literature that have been specially tailored for certain distributions.

Finally, much our discussion above concerns noiseless recovery of exactly sparse functions, the efficacy of all the procedures above essentially generalizes to ``approximately sparse" (or \textit{compressible}) signals in the presence of noisy data. Indeed the Theorem~\ref{thm:csa-convergence} presented in Section~\ref{sec:1d-analysis} provides bounds on the error of the $s$-term approximation recovered by the CSA algorithm in the presence of noisy data. 

%% file: csa.tex
\section{Christoffel sparse
approximation}\label{sec:christoffel-sparse-approximation}
The Christoffel sparse approximation (CSA) algorithm solves the preconditioned BPDN problem~\eqref{eq:pre-condition-bpdn} to accurately recover sparse orthogonal polynomials from limited data.  Given $Z$ and its distribution along with an index set $\Lambda$, the CSA method consists of five steps which are outlined here and summarized in Algorithm~\ref{alg:csa-algorithm}: 
\begin{enumerate}
\item sample iid with respect to the probability density $v$ of an equilibrium measure, which depends on the orthogonality weight $w$. When $\rv$ is a random variable with unbounded state space $I_Z$, then $v$ depends on $n$, the maximum polynomial degree of the index set $\Lambda$ defining the dictionary. In this case we write $v = v_n$.
\item evaluate the function $f(\brv^{(m)})$ at the selected samples $\{\brv_m\}_{m=1}^M$
  \item form $M \times N(\Lambda)$ Vandermonde-like matrix $\vand$ with entries $\Phi_{m,i} = \phi_{i}(\brv^{(m)})$\;
  \item form the diagonal preconditioning matrix $\wtmat$ using the values $\sqrt{N \lambda(Z_m)}$, which are evaluations of the (scaled) Christoffel function from the $\pbwt$-orthogonal polynomial family $\phi$. (See \eqref{eq:christoffel-function}).
  \item solve the preconditioned $\ell^1$-minimization problem \eqref{eq:pre-condition-bpdn} to approximate the polynomial coefficients $\alpha_{\bsi}$
\end{enumerate}

\input{tikz/csa-algorithm}

The weight function $v$ is a density that depends on the density $w$ of $Z$ and $I_Z$. Owing to the discussion at the end of Section \ref{sec:background-cs-recovery}, we require $v \approx w/N \lambda_N$. As $N \rightarrow\infty$, one has equality (interpreted appropriately) with a unique limit given by the weighted pluripotential equilibrium measure. Said in another way, the asymptotic behavior of the Christoffel function with respect to the orthogonality density $w$ turns out to be governed by the weighted pluripotential equilibrium measure \cite{berman_bergman_2009}. 

When $I_Z$ is bounded, our sampling density $v$ is the density associated to the unweighted pluripotential equilibrium measure on the domain $I_Z$, regardless of $w$. When $I_Z$ is unbounded and $w$ is an exponential weight, then $v$ is a scaled version of the density associated to the $\sqrt{w}$-weighted pluripotential equilibrium measure on the domain $I_Z$; the scaling factor depends on $w$. We detail these sampling techniques for various canonical weights and domains in the coming discussion. The main purpose of this section is to describe the sampling density $v(z)$ for the CSA algorithm; once this is completed, the remainder of the algorithm follows a standard preconditioned $\ell^1$ recovery procedure in compressive sampling.

In this paper we use the equilibrium measure as a tool for performing simulations, and so omit details concerning its definition or properties. Standard references for pluripotential theory are \cite{klimeck_pluripotential_1991} and Appendix B of \cite{saff_logarithmic_1997}. Much of the material presented here defining the sampling density $v(z)$ is also in \cite{Narayan_JJ_MC}.

Given domain $I_Z$, let $Q(z)$ be a certain function defining a weight function $\exp(-Q)$. This new weight function in general can be distinct from the probability density $w$, but in the cases we describe below they are related. The weight function $\exp(-Q)$ is associated with the $\exp(-Q)$-weighted pluripotential equilibrium measure on $I_Z$, denoted by $\mu_{I_Z, Q}$,which is a unique probability measure. When $Q = 0$, we use the notation $\mu_{I_Z} = \mu_{I_Z, 0}$. This measure, loosely speaking, carries information about the extremal behavior of weighted polynomials on $I_Z$. Particular examples of this measure are given in the following sections. On bounded domains, we will identify $Q = 0$, and define our sampling density $v$ to correspond directly to $\mu_{I_Z}$. On unbounded domains with exponential weights, we will identify $Q = -\frac{1}{2} \log w$, and define our sampling density $v$ to be a dilated version of $\mu_{I_Z,Q}$.

As noted above, the particular sampling strategy (via the weighted equilibrium measure) differs on bounded versus unbounded domains, so we discuss them individually here. The main difference between the bounded and unbounded cases is that for the bounded case, $v$ does not depend on $\Lambda$, whereas for the unbounded case it does. The state of knowledge for the univariate case is essentially complete, whereas much is still unknown about the general weighted multivariate case. 

The purpose of the following sections is to describe the sampling density $v$ used in step (a) above. The remaining steps (b) and (c) have straightforward and identical implementation in all cases below: perform the optimization \eqref{eq:pre-condition-bpdn} with the weight matrix $\wtmat$ entries given by \eqref{eq:christoffel-function}.

\subsection{Univariate sampling}\label{sec:univariate-sampling}
In the univariate case, we have an essentially complete description of the various sampling measures. This is a direct result of historically successful analysis and characterization of the weighted potential equilibrium measure. For the bounded domain case, we sample from the (unweighted) potential equilibrium measure, and for the unbounded case we sample from expanded versions of the weighted potential equilibrium measure.

\subsubsection{Bounded intervals}
When $\rv$ is scalar and takes values on the finite interval $I_Z = [-1,1]$, then we sample from the unweighted potential equilibrium measure, which is the arcsine measure having the ``Chebyshev" density
\begin{align}\label{eq:chebyshev-density}
  v(z) = \frac{1}{\pi \sqrt{1 - z^2}}.
\end{align}
Sampling from this density for the purpose of compressive sampling is not new \cite{Rauhut_W_JAT_2012}. This density corresponds to a standard symmetric Beta distribution, and is readily sampled from using available computational packages. Note that the sampling density here is independent of the weight function $w$ defining the PCE basis.

\subsubsection{Exponential densities on $\R$}\label{sec:univariate-sampling-whole}
Now we consider the case that $I_Z$ is unbounded with $I_Z = (-\infty, \infty) = \R$. For any scalar $\alpha > 1$, we assume that $w(z)$ is an exponential probability density weight of the form 
\begin{align*}
  w(z) &= \exp(-|z|^\alpha), & z &\in (-\infty, \infty).
\end{align*}
Note that we need a normalization constant to make $w$ a probability density, but the constant does not affect any of the discussion below so we omit it. Here, we have $Q = - \frac{1}{2} \log w = \frac{1}{2} |z|^\alpha$. Our assumption $\alpha > 1$ is a standard assumption in the theory of orthogonal polynomials, and relates to conditions ensuring $L^2_w$-completeness of polynomials. We use $n = \max_{i \in \bset} i$ to denote the maximum polynomial degree of the index set $\bset$.

Associated to the parameter $\alpha$ (more pedantically, associated to $\sqrt{w}$), we define the following constants:
\begin{align}\label{eq:MRS-whole}
  a^W = a^W\left(\sqrt{w}\right) &= \left(\frac{ \sqrt{\pi} \Gamma\left(\frac{\alpha}{2}\right) }{\Gamma \left(\frac{ \alpha}{2} + \frac{1}{2}\right) } \right)^{1/\alpha}, & a^W_{n} \equiv n^{1/\alpha} a^W.
\end{align}
These are the Mhaskar-Rahkmanov-Saff numbers associated with the weight $\sqrt{w}$ \cite{mhaskar_where_1985,rakhmanov_asymptotic_1984}. We add the superscript $W$ to indicate that these numbers are associated with weights having support on the \textit{Whole} real line $(-\infty, \infty)$. These numbers delinate a compact interval on which a weighted polynomial achieves its $I_Z$-supremum. We define the following intervals:
\begin{align}\label{eq:whole-sn}
  S^W_{n} &= \left[ - a^W_{n}, a^W_{n} \right], & S^W \triangleq S^W_{1}.
\end{align}
The probability density function of the $\sqrt{w}$-weighted equilibrium measure is given by
\begin{align}\label{eq:equilibrium-measure-whole-line}
  \dfdx{\mu_{I_Z,Q}}{z}\left(\rvd\right) &= \frac{\alpha \sqrt{\left(a^W\right)^2 - \rvd^2}}{\pi^2} \int_{-a^W}^{a^W} \frac{u^{\alpha-1} - \rvd^{\alpha-1}}{(u-\rvd) \sqrt{\left(a^W\right)^2 - u^2}} \dx{u}
  & z &\in S^W,
\end{align}
where $\dx{z}$ denotes Lebesgue measure on $I_Z$. See, e.g., \cite{saff_logarithmic_1997}; a summary table of explicit weights for various $\alpha$ is also given in \cite{narayan_adaptive_2014}. 

Given this density, and the index set $\Lambda$ which forms our compressive sampling dictionary, the CSA sampling density $v_n(z)$ is formed by linearly mapping the density of $\mu_{I_Z,Q}$ to $S^W_n$:
\begin{align}\label{eq:whole-line-sampling-density}
  v_n(z) &= n^{-1/\alpha} \dfdx{\mu^W_{I_Z,Q}}{  z }\left( n^{-1/\alpha} \rv\right), & z &\in S^W_{n},
\end{align}
We note that in the particularly important case of $\alpha = 2$, corresponding to a PCE basis of Hermite polynomials, we have $a^W = \sqrt{2}$, and
\begin{align*}
  v_n(z) &= \frac{1}{\pi \sqrt{n}} \sqrt{2 n - z^2}, & z &\in [-\sqrt{2 n}, \sqrt{2 n}],
\end{align*}
which is another standard, symmetric Beta distribution, and so easily sampled.

While \eqref{eq:equilibrium-measure-whole-line} seems relatively complicated, in fact this density behaves essentially like the $\alpha=2$ semicircle density above. I.e., for any $\alpha > 1$, there are positive constants $c_1, c_2$ depending only on $\alpha$ such that
\begin{align*}
  c_1 \sqrt{\left(a^W\right)^2 -z^2} \leq \dfdx{\mu_{I_Z,Q}}{z}\left( z \right) &\leq c_2 \sqrt{\left(a^W\right)^2 -z^2}, & z \in \left(-a^W, a^W\right)
\end{align*}
See Theorem 1.10 of \cite{levin_orthogonal_2001}.

\subsubsection{Exponential densities on $[0, \infty)$}\label{sec:univariate-sampling-half}
In this section we consider the case of $I_Z = [0, \infty)$. We assume that $w(z)$ is a one-sided exponential weight of the form 
\begin{align*}
  w(z) &= \exp(-|z|^\alpha), & z &\in [0, \infty).
\end{align*}
Again we have $Q = - \frac{1}{2} \log w = \frac{1}{2} |z|^\alpha$. We assume that $\alpha > \frac{1}{2}$, and again take $n = \max_{i \in \bset} i$.

Associated to the parameter $\alpha$, we define the following constants:
\begin{align}\label{eq:MRS-half}
  a^H =a^H (\sqrt{w}) &=  \left(\frac{ 2 \sqrt{\pi} \Gamma\left(\alpha\right) }{\Gamma \left(\alpha + \frac{1}{2}\right) } \right)^{1/\alpha}, & a^H_{n} \equiv n^{1/\alpha} a^H.
\end{align}
These are the Mhaskar-Rakhmanov-Saff numbers associated with the weight $\sqrt{w}$ on $[0, \infty)$. 
  
The superscript $H$ indicates that these numbers are associated with weights having support on the \textit{Half} real line. We define the following intervals:
\begin{align}\label{eq:half-sn}
  S^H_{n} &= \left[ 0, a^H_{n} \right], & S^H \triangleq S^H_{1}.
\end{align}
Consider the probability density function of the weighted equilibrium measure, given by
\begin{align}\label{eq:equilibrium-measure-half-line}
  \dfdx{\mu^H_{I_Z,Q}}{z}\left(\rv\right) &= \frac{\alpha}{\pi^2} \sqrt{\frac{a^H_{\alpha} - \rv}{\rv}} \int_0^{a^H_{\alpha}} \frac{u^\alpha - \rv^\alpha}{(u - \rv)\sqrt{u \left(a^H_{\alpha} - u\right)}} \dx{u},
  & z &\in S^H_{\alpha}.
\end{align}
Given this density, and the index set $\Lambda$ which forms our compressive sampling dictionary, we form the CSA sampling density $v_n$ by mapping the density of $\mu^H_{I_Z,Q}$ to $S^H_n$:
\begin{align}\label{eq:half-line-sampling-density}
  v_n(\rvd) &= n^{-1/\alpha} \dfdx{\mu_{I_Z,Q}}{  \rvd }\left( n^{-1/\alpha} \rvd\right), & \rvd &\in S^H_{n}.
\end{align}
We note in particular for case of $\alpha = 1$ corresponding to a PCE basis of Laguerre polynomials we have $a^H_1 = 4$, and: 
\begin{align}\label{eq:laguerre-scaled-equilibrium-density}
  v(\rvd) &= \frac{1}{2 n \pi} \sqrt{\frac{4n - \rvd}{\rvd}}, & \rvd &\in [0, 4 n].
\end{align}
This is just an asymmetric Beta distribution on $[0, 4 n]$.

The formula \eqref{eq:equilibrium-measure-half-line} looks cumbersome, but this measure essentially behaves like the whole-line measure \eqref{eq:equilibrium-measure-whole-line}. In particular, we have
\begin{align*}
  \dfdx{\mu^H_{I_Z,Q}}{z}(z) &= \frac{1}{2 \sqrt{z}} \dfdx{\mu_{\R, Q^2}}{z} \left( 2^{1/2\alpha} \sqrt{z} \right), & z &\in S^H
\end{align*}
See \cite{levin_orthogonal_2005}. Thus, we can actually sample from the measure $\mu^H_{I_Z, Q}$ by instead transforming samples from $\mu^W_{\R, Q^2}$. Given $\alpha$, let $Y$ be a random variable whose distribution is $\mu^W_{\R, Q^2}$. Then $Z = 2^{1/\alpha} Y^2$ has distribution $\mu^H_{I_Z, Q}$. Therefore, sampling from the half-line equilibrium measures can be reduced to the problem of sampling from the whole-line measure \eqref{eq:equilibrium-measure-whole-line}.

\subsection{Multivariate sampling}\label{sec:multivariate-sampling}

In general very little is known (other than existence/uniqueness) about the (un)weighted pluripotential equilibrium measure on sets in $\R^d$ with $d > 1$. We detail some special cases below. We note in particular that essentially nothing is known about the weighted case, and so our descriptions of the sampling density $v$ on unbounded domains with exponential weights are conjectures.

\subsubsection{The bounded domain $[-1,1]^d$}
Let $\rv$ take values on the domain $I_Z = [-1,1]^d$ for any weight function $w$. In this case, the CSA sampling density $v$ corresponds to the unweighted equilibrium measure on $I_Z$; the density of this measure is given by
\begin{align*}
  v(\rv) = \frac{1}{\pi^d} \frac{1}{\prod_{i=1}^d \sqrt{1 - \rv_i^2}},
\end{align*}
which is a tensor-product Chebyshev density on $[-1,1]^d$.

\subsubsection{Convex, symmetric bounded domains}
If $\rv$ takes values on a compact set $I_Z \subset \R^d$, our CSA sampling density $v$ corresponds to the unweighted equilibrium measure on $I_Z$. There is little that can be said about this measure in general, but the following special cases are known:
\begin{itemize}
  \item \cite{bedford_complex_1986}: if $I_Z$ is convex, then the equilibrium measure Lebesgue density $\dfdx{\mu_{I_Z}}{z}(z)$ exists and is bounded above and below relative to $\left[\mathrm{dist}\left(z, \partial I_Z\right)\right]^{-1/2}$, and is thus ``Chebyshev-like". (We use $\mathrm{dist}$ to denote Euclidean distance between sets, with $\partial I_Z$ the boundary of $I_Z$.)
  \item \cite{bedford_complex_1986}: if $I_Z$ is the unit ball in $\R^d$, then the equilibrium measure density is given by
    \begin{align*}
      v(z) = \dfdx{\mu_{I_Z}}{z} = \frac{c}{\sqrt{1 - \left\|z\right\|_2^2}}, 
    \end{align*}
    where $c$ is a normalization constant. Note that sampling from this density is relatively straightforward: Let $W$ be a standard normal random variable in $\R^d$, and let $B$ be a univariate Beta random variable on $[0,1]$ with shape parameters $\alpha = \frac{d}{2}$ and $\beta = \frac{1}{2}$. Then $Z = \frac{W}{\|W\|_2} \sqrt{B}$ has the desired distribution on the unit ball in $\R^d$.\footnote{To see this, note that $R = |Z|$ needs to have density proportional to $R^{d-1} (1 - R^2)^{-1/2}$. Some manipulation on this distribution shows that $R^2$ has the distribution of $B$. The $\frac{W}{\|W\|_2}$ factor is a directional factor, sampling uniformly on the boundary of the unit ball.}
  \item \cite{xu_asymptotics_1999}: if $I_Z = \left\{z \in \R^d\;\; |\;\; \sum_{j=1}^d z_j \leq 1 \textrm{ and } z_j \geq 0, \; j=1, \ldots, d\right\}$ is the unit simplex in $\R^d$, then $\dfdx{\mu_{I_Z}}{z} = c \left[ \sqrt{1 - \sum_{j=1}^d z_j} \prod_{j=1}^d z_j \right]^{-1/2}$, with $c$ a normalizing constant. This density may also be sampled: this is, in fact, the density of a $(d+1)$-dimensional Dirichlet distribution with all $d+1$ parameters equal to $\frac{1}{2}$ and the $(d+1)$'st coordinate equal to $1 - \sum_{j=1}^d z_j$. Therefore, let $D$ be the previously mentioned $(d+1)$-dimensional Dirichlet random vector. Form $Z$ by truncating the last entry in this vector; then $Z$ is a $d$-dimensional random variable with density $\dfdx{\mu_{I_Z}}{z}$.
\end{itemize}
To the best of our knowledge, these cases are essentially a complete description of the current state of knowledge about the unweighted equilibrium measure on domains in $\R^d$.

\subsubsection{The domain $\R^d$ with Gaussian density}\label{sec:csa-hermite-multivariate-sampling}
Let $I_Z = \R^d$ with Gaussian weight $w(z) = c \exp\left( -\left\|z\right\|_2^2 \right)$ with $c$ a normalizing constant. In \cite{Narayan_JJ_MC}, we prescribe sampling according to the density 
\begin{align*}
  v(\rv) = C \left( 2 - \left\|\rv\right\|^2 \right)^{d/2},
\end{align*}
with $C$ a normalization constant. We conjectured in \cite{Narayan_JJ_MC} that this is the weighted equilibrium measure associated to this choice of $\pbwt$. Like the univariate unbounded cases with $\alpha = 2$, we expand the samples by the square root of the maximum polynomial degree in the index set $\Lambda$. The following is a concrete way to sample from this expanded density: 
\begin{enumerate}
  \item Compute $n$, the maximum polynomial degree in the index set $\Lambda$. This is equal to $\max_{\bsi \in \Lambda} \left|\bsi\right|$.
  \item Generate a vector $Y=(y_1,\ldots,y_d)$ of $d$ independent normally distributed
random variables. 
  \item Draw a scalar sample $u$ from the Beta distribution on $[0,1]$, with distribution parameters $\alpha=d/2$ and $\beta=d/2+1$. 
  \item Set
    \begin{align*}
      \brv=\frac{Y}{\lVert Y\rVert_2} (2 n u)^{1/2}.
    \end{align*}
\end{enumerate}
The above procedure generates samples on the Euclidean ball of radius $\sqrt{2 n}$ in $\R^d$. We emphasize that our methodology samples from a density that is only a conjecture for the correct equilibrium measure. 

\subsubsection{The domain $[0, \infty)^d$ with exponential density}
Let $\rv$ take values on $I_Z = [0, \infty)^d$ with associated probability density $\pbwt(\rvd) = c \exp(-\|\rvd\|_1)$, where $c$ is a normalizing constant. Here we sample from the density function
\begin{align*}
  v(\rv) = C \sqrt{\frac{\left(4 - \sum_{i=1}^d \rvd_i\right)^d}{\prod_{i=1}^d \rvd_i}}
\end{align*}
As we conjectured in \cite{Narayan_JJ_MC}, this is the equilibrium measure associated to this choice of $\pbwt$. We expand the samples by the maximum polynomial degree in the index set $\Lambda$ (like the $\alpha = 1$ unbounded univariate case). The following is a concrete way to sample from this expanded density: 
\begin{enumerate}
  \item Compute $n$, the maximum polynomial degree in the index set $\Lambda$. This is equal to $\max_{\bsi \in \Lambda} \left|\bsi\right|$.
  \item Generate a $(d+1)$-variate Dirichlet random variable $W$ with parameters $\left(\frac{1}{2}, \frac{1}{2}, \ldots, \frac{1}{2},
\frac{d}{2}+1\right)$.
  \item Truncate the last ($(d+1)$'th) entry of $W$
  \item Set $\brv = 4 W n$.
\end{enumerate}
The above procedure generates samples on the set of Euclidean points in $\R^d$ whose $\ell^1$ norm is less than or equal to $4 n$.

\section{Analysis for univariate systems}\label{sec:1d-analysis}
Our analysis in this section assumes that $\rv$ is a scalar. (That is, we consider the one-dimensional case here.) We show that a CSA-based algorithm can successfully recover sparse and compressible solutions. On bounded domains, we can accomplish this with the optimal $M \gtrsim s$ samples. For unbounded domains, we pay an additional penalty, requiring $M \gtrsim s n^{2/3}$. To establish these results, we use the analysis for bounded orthonormal systems presented in \cite{Rahut_TFNMSR_2010}. Using this analysis, the unbounded penalty of $n^{2/3}$ is sharp. The CSA method also introduces some error terms stemming from the fact that the PCE basis is only approximately orthonormal under our sampling strategy; these error terms are similar to those obtained for least-squares in \cite{Narayan_JJ_MC}.


With $\rv$ a scalar, we consider approximation with a dictionary of PCE basis elements up to degree $n \in \N$, corresponding to a total of $N = n+1$ basis elements. To be explicit, in this section we assume
\begin{align}\label{eq:univariate-dictionary}
  \Lambda &= \left\{ 0, 1, \ldots, n \right\}, & |\Lambda| &= N = n+1.
\end{align}
Recall that in the unbounded case our sampling depends on $n$. When the number of dominant coefficients is $s \ll N$, we will be concerned with recovery of the $s$ dominant coefficients using $M$ samples. We use the CSA procedure detailed at the beginning of Section \ref{sec:christoffel-sparse-approximation} for sampling, which proposes the following sampling densities for $v$:
\begin{enumerate}[label=(CSA-\alph*)]
  \item When $\rv$ has a Beta distribution with shape parameters $\beta+1, \alpha+1 \geq \frac{1}{2}$ (corresponding to Jacobi polynomial parameters $\alpha, \beta \geq - \frac{1}{2}$), we sample from the Chebyshev density \eqref{eq:chebyshev-density}. This corresponds to $n$-independent sampling according to $v_n \equiv v$ defined in \eqref{eq:chebyshev-density} and $S_n \equiv [-1,1]$.
  \item When $\rv$ is a two-sided exponential random variable ($\rv \in \R$) with density $w(\rvd) = \exp(-|\rvd|^\alpha)$, for $\alpha > 1$, we sample from the expanded equilibrium measure whose density is $v_n(\rvd)$ with support $S_n$, defined in \eqref{eq:whole-line-sampling-density} and \eqref{eq:whole-sn}, respectively.
  \item When $\rv$ is a one-sided exponential random variable ($\rv \in [0, \infty)$) with density $w(\rvd) = \exp(-|\rvd|^\alpha)$, for $\alpha > \frac{1}{2}$, we sample from the expanded equilibrium measure whose density is $v_n(\rvd)$ with support $S_n$, defined in \eqref{eq:half-line-sampling-density} and \eqref{eq:half-sn}, respectively.
\end{enumerate}
In any of the three cases above, our convergence statement requires definition of an additional matrix, which is a Gramian matrix corresponding to the Christoffel function weighted PCE basis elements $\phi_k$ 
\begin{align}\label{eq:R-definition}
  R_{k,\ell} &= \int_{S_n} \phi_{k-1}(\rvd) \phi_{\ell-1}(\rvd) \left( N \lambda_n(\rvd)\right) v_n(\rvd) \dx{\rvd}, & 1 &\leq k, \ell \leq N = n+1.
\end{align}
Note that the matrix $\V{R}$ defined above is positive-definite, and any fixed entry of this matrix converges to the corresponding entry in the identity matrix as $n \rightarrow \infty$ \cite{Narayan_JJ_MC}. Our result below uses the induced $\ell^1$ norm for matrices, $\|\bs{A}\|_1$, which is the maximum $\ell^1$ vector norm of columns of $\bs{A}$. For the symmetric positive-definite matrix $\bR$, we use $\bR^{1/2}$ to denote its unique symmetric positive definite square root. 
\begin{theorem}\label{thm:csa-convergence}
  Suppose that $M$ sampling points $(\brv^{(1)} ,\ldots , \brv^{(M)})$ are drawn iid according to the equilibrium measure density $v_n$ associated with the probability measure $\pbwt$, and consider the $M\times N$ matrix $\vand$ with entries $\vand_{ij} = \phi_j(\brvsamp{i}),\;i\in[M],\; j\in[N]$ and the diagonal matrix $\bW$ with entries given by \eqref{eq:christoffel-function}.
Assume that the number of samples satisfies
\begin{equation}\label{eq:sample-count}
  M \ge L(n) \left\| \V{R}^{-1/2} \right\|^2_{1} s \log^3(s)\log(N),
\end{equation}
where $\V{R}$ is defined in \eqref{eq:R-definition} and $L(n)$ has the following behavior:
\begin{enumerate}[label=(CSA-\alph*)]
  \item \label{jacobi-bound} There is a constant $C = C(\alpha, \beta)$ such that uniformly in $n \geq 1$,
    \begin{align*}
      L(n) = C.
    \end{align*}
  \item \label{whole-exp-bound} There is a constant $C = C(\alpha)$ such that uniformly in $n \geq 1$,
    \begin{align*}
      L(n) = C n^{\max\left\{1/\alpha, 2/3\right\}} = \left\{ \begin{array}{cc} C n^{2/3}, & \alpha \geq \frac{3}{2} \\
                     C n^{1/\alpha}, & 1 < \alpha < \frac{3}{2} \end{array}\right.
    \end{align*}
  \item \label{half-exp-bound} There is a constant $C = C(\alpha)$ such that uniformly in $n \geq 1$,
    \begin{align*}
      L(n) = C n^{\max\left\{1/2\alpha, 2/3\right\}} = \left\{ \begin{array}{cc} C n^{2/3}, & \alpha \geq \frac{3}{4} \\
    C n^{1/\alpha}, & \frac{1}{2} < \alpha < \frac{3}{4} \end{array}\right.
    \end{align*}
\end{enumerate}
Then with probability exceeding $1-N^{-\gamma\log^3(s)}$ the following holds for all 
polynomials $p(x)=\sum_{j=1}\alpha_j\phi_j(x)$.
Suppose that noisy sample values $\boldf = (p(\brvsamp{1})+\eta_1, \ldots,
p(\brvsamp{M})+\eta_M)^T = \vand\coef + \bseta$ are observed,
with $\norm{\bW\bseta}{\infty} \le\stoptol$. Then the coefficient vector 
$\bR^{1/2} \coef$ is recoverable to within
a factor of its best $s$-term approximation error and to a factor of the noise
level by solving the inequality-constrained $\ell^1$-minimization problem
\begin{equation}\label{eq:csa-weighted-minimization}
  \bR^{1/2} \coef^\star  = \argmin_{\coef}\; \|\bR^{1/2} \coef\|_1\quad \text{such that}\quad
  \|\sqrt{\bW}\vand\coef - \sqrt{\bW}\boldf\|_2 \le \stoptol
\end{equation}
The error between $\coef$ and the recovered solution $\coef^\star$ satisfies
\begin{subequations}\label{eq:csa-convergence}
  \begin{align}\label{eq:csa-convergence-1}
    \norm{\coef-\coef^\star}{2} &\le \frac{C_1\sigma_s\left(\bR^{1/2} \coef\right)_1}{\sqrt{s \lambda_{\mathrm{min}}\left(\bR\right)}}+C_2\frac{\stoptol}{\sqrt{\lambda_{\mathrm{min}}\left(\bR\right)}} \\
\label{eq:csa-convergence-2}
\norm{\coef-\coef^\star}{1} &\le D_1\sigma_s\left(\bR^{1/2} \coef \right)_1 \left\|\bR^{-1/2} \right\|_1 + D_2\sqrt{s} \left\|\bR^{-1/2}\right\|_1 \stoptol
\end{align}
\end{subequations}
where $\sigma_s(z)_p=\inf_{\{\by:\norm{\by}{0}\le s\}}\norm{\by-\bz}{p}$ is the 
best $s$-term approximation of a vector $\bz\in\mathbb{R}^N$ in $\ell_p$ norm,
 and $C_1,C_2$ are universal constants .
\end{theorem}
The proof is essentially just a convergence result for sparse solutions from a basis pursuit algorithm using bounds on an orthonormal system as established in \cite{Rahut_TFNMSR_2010}, with the required bounds on Christoffel-weighted polynomials provided in Theorems \ref{thm:bound-csa} later. The full proof of the above Theorem is presented in Appendix \ref{sec:appendix-csa-convergence}.

\begin{remark}
  We expect that the case corresponding to the one-dimensional Beta distribution (Jacobi polynomials) can actually be extended to almost any kind of bounded weight function on a compact interval. (The essential results are in, e.g., \cite{erdelyi_generalized_1992}.) 
\end{remark}
\begin{remark}
  For the CSA-b case, the behavior of $L$ with respect to $n$ is sharp; i.e., when $\alpha \geq \frac{3}{2}$ the $n^{2/3}$ behavior is sharp, and when $\alpha < \frac{3}{2}$ the $n^{1/\alpha}$ behavior is sharp. See the comments following Theorem \ref{thm:bound-csa-b}. A similar statement holds for the CSA-c case.
\end{remark}
\begin{remark}
  The above results generalize to tensor-product domains and weights, if one uses tensor-product sampling from the respective univariate densities. Then the sample count criterion in \eqref{eq:sample-count} is the same, but the factor $L(n)$ is a product of $d$ such factors, each corresponding to the respective univariate bound.
\end{remark}

Note that the result is stated in terms of recovery of $\bs{R}^{1/2} \coef$, and not $\coef$. This is done because the PCE basis is only orthonormal under $N\lambda_N(\rvd) v_n(\rvd)$ when transformed according to the Gramian $\bs{R}$. Note that the actual CSA algorithm described in Section \ref{sec:christoffel-sparse-approximation} performs recovery of $\coef$, which is not the statement of the Theorem above. 

Nevertheless, we empirically observe that $\bs{R}^{1/2}$ is not only close to the identity \cite{Narayan_JJ_MC}, but also that has an approximately sparse representation. Thus, our empirical observation is that minimizing with the objective $\left\|\bs{R}^{1/2} \coef\right\|_1$ is similar to minimizing with objective $\left\|\coef\right\|_1$. When $\bs{R}$ is sparse and close to the identity, then $\sigma_s\left(\coef\right)_1$ and $\sigma_s\left(\bs{R}^{1/2} \coef\right)_1$ should likewise be similar. 

The penalty factor $\left\| \bs{R}^{-1/2}\right\|_1$ appearing in both the sampling condition \eqref{eq:sample-count} and the estimate \eqref{eq:csa-convergence-2} is likewise small. We show various values for this penalty in Figure \ref{fig:Rnorms}, which we observe to be $\mathcal{O}(1)$ for various densities $w$ on both bounded and unbounded domains. The term $\lambda_{\mathrm{min}}\left(\bs{R}\right)$ appearing in \eqref{eq:csa-convergence-1} is computed empirically in \cite{Narayan_JJ_MC} and observed to be $\mathcal{O}(1)$ for essentially all values of $n$ and all relevant shape parameters.

\begin{figure}
  \begin{center}
    \includegraphics[width=\textwidth]{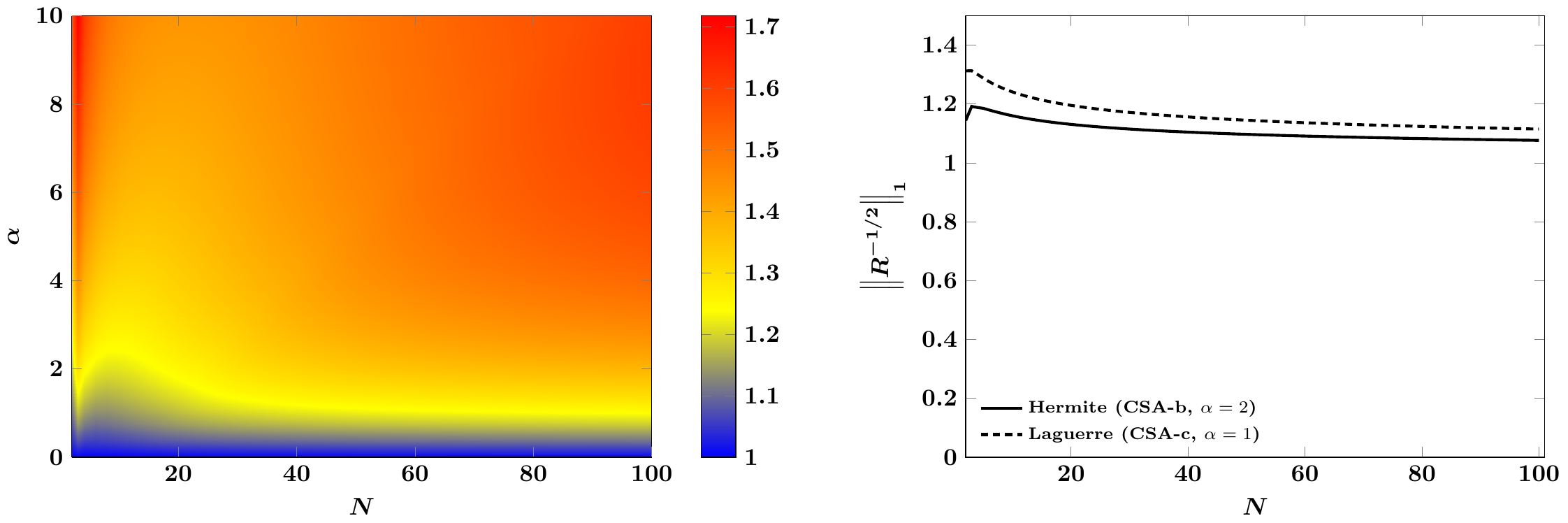}
  \end{center}
  \caption{Left: $\left\|\bR^{-1/2}\right\|_1$ for Jacobi polynomials with symmetric parameters $\alpha = \beta$. Right: $\left\|\bR^{-1/2}\right\|_1$ for two exponential-type densities: Hermite polynomials with $\pbwt = \exp(-x^2)$ on $\R$, and Laguerre polynomials with $\pbwt = \exp(-x)$ on $[0, \infty)$.}\label{fig:Rnorms}
\end{figure}

To motivate our observation that $\bs{R}$ is in fact quite close to the identity, we have the following optimal result for $Z$ a bounded uniform random variable, corresponding to a PCE basis of Legendre polynomials.
\begin{corollary}\label{cor:legendre-csa-convergence}
  Assume the setup of Thereom \ref{thm:csa-convergence} in the CSA-a (bounded) case, with parameters $\alpha = \beta = 0$. This corresponds to $Z$ a uniform random variable with a Legendre PCE basis. Then assuming
  \begin{align*}
    M \ge C s \log^3(s)\log(N),
  \end{align*}
  the solution $\coef^\star$ to 
  \begin{align*}
      \coef^\star  = \argmin_{\coef}\; \|\coef\|_1\quad \text{such that}\quad
  \|\sqrt{\bW}\vand\coef - \sqrt{\bW}\boldf\|_2 \le \stoptol
  \end{align*}
  satisfies
  \begin{align*}
    \norm{\coef-\coef^\star}{2} &\le \frac{C_1\sigma_s\left(\coef\right)_1}{\sqrt{s}}+C_2\stoptol \\
    \norm{\coef-\coef^\star}{1} &\le D_1\sigma_s\left(\coef \right)_1 + D_2\sqrt{s} \stoptol
\end{align*}
\end{corollary}
\begin{proof}
  Under the assumption that $Z$ is a scalar uniform random variable, the main result of \cite{bos_orthogonality_2015} states that $\bs{R} = \bs{I}$ for all $n$. Application of Theorem \ref{thm:csa-convergence} with $\bs{R} = \bs{I}$ is the desired conclusion.
\end{proof}

\section{Bounds on Christoffel-weighted polynomials}
Theorem \ref{thm:csa-convergence} requires knowledge of bounds on Christoffel-weighted polynomials. We provide these bounds here, but leave the proofs to the appendix due to their technical nature. 

As with Section \ref{sec:1d-analysis}, we restrict our attention to the univariate case here, with a dictionary defined by \eqref{eq:univariate-dictionary}. Given this definition of $\Lambda$, the Christoffel function $\lambda_\Lambda$ is uniquely defined by \eqref{eq:christoffel-function}. For the univariate case, we express expicit dependence on the maximum polynomial degree $n$ by writing
\begin{align*}
  \lambda_{n+1}(\rvd) = \frac{1}{\sum_{k=0}^n \phi_k^2(\rvd)}
\end{align*}
We use the subscript $n+1$ to keep consistency with other literature. This notation is used in all of the following theorems. As before, we ignore any normalization factors that are necessary to make $w$ a probability density since their presence does not affect the results in any way.
\begin{subtheorem}{theorem}\label{thm:bound-csa}
\begin{theorem}\label{thm:bound-csa-a}
  Assume $Z$ is a Beta-distributed random variable on $[-1,1]$, with shape parameters $\alpha, \beta \geq \frac{1}{2}$, having density
  \begin{align*}
    w(z) = (1-z)^{\beta-1} (1+z)^{\alpha-1}.
  \end{align*}
  Then, uniformly in $n$, 
  \begin{align}\label{eq:bound-csa-a}
    L(n) = \max_{0 \leq k \leq n} \sup_{z \in [-1,1]} (n+1) \lambda_{n+1}(z) \phi_k^2(z) \leq C
  \end{align}
\end{theorem}
The uniform boundedness of these weighted polynomials has already been investigated for compressive sampling using essentially the same kind of weight \cite{Rauhut_W_JAT_2012}. The proof is contained in Appendix \ref{sec:appendix-csa-a}. 

We now consider the more difficult CSA-b case where $Z$ has a two-sided exponential density.
\begin{theorem}\label{thm:bound-csa-b}
  Assume $Z$ is an exponentially-distributed random variable on $\R$, having density with shape parameter $\alpha > 1$ given by 
  \begin{align*}
    w(\rvd) = \exp\left(-|\rvd|^\alpha\right).
  \end{align*}
  Then, uniformly in $n$, 
  \begin{align*}
    L(n) = \max_{0 \leq k \leq n} \sup_{z \in S^W(\sqrt{w})_n} (n+1) \lambda_{n+1}(z) \phi_k^2(z) \leq C_1 n^{p(\alpha)},
  \end{align*}
  where 
  \begin{align}\label{eq:p-definition}
  \arraycolsep=1.4pt\def\arraystretch{1.4}
  p(\alpha) = \max\left\{ \frac{2}{3}, \frac{1}{\alpha} \right\} = \left\{ \begin{array}{cc} \frac{2}{3}, & \alpha \geq \frac{3}{2} \\ \frac{1}{\alpha}, & \alpha < \frac{3}{2} \end{array}\right.
  \end{align}
\end{theorem}
Some remarks on this are in order: first, in the important $\alpha = 2$ corresponding to a PCE basis of Hermite polynomials, we have the bound $L(n) \sim n^{2/3}$, and thus this weighting does not produce a bound that is uniform in $n$. The bound above is sharp with respect to the prescription of $p(\alpha)$. For example with $k=n$, and using formulas \eqref{eq:christoffel-whole-sim}, \eqref{eq:whole-varphi-bound}, and \eqref{eq:poly-whole-sup} from Appendix \ref{sec:appendix-csa-b}, we have
\begin{align*}
  L(n) \geq \sup_{z \in S^W_n} \phi_n^2(z) (n+1) \lambda_{n+1}(z) \sim n^{2/3},
\end{align*}
So that the exponent $\frac{2}{3}$ is a lower bound. However for the special case $k=0$ at $z=0$, we use \eqref{eq:poly-whole-sup} and \eqref{eq:christoffel-whole-sim} to conclude
\begin{align*}
  L(n) \geq (n+1) \lambda_{n+1}(0) \phi_0^2(0) \sim a^W_n w(0) \phi_0^2(0) \sim n^{1/\alpha},
\end{align*}
and thus we have that $L(n)$ must be larger than at least $n^{1/\alpha}$ and $n^{2/3}$. It turns out that these are in fact the dominating behaviors and so $L(n)$ behaves exactly according to the maximum of the two bounds above. The proof of this is given in Appendix \ref{sec:appendix-csa-b}. The one-sided CSA-c case is similar.
\begin{theorem}\label{thm:bound-csa-c}
  Assume $Z$ is an exponentially-distributed random variable on $[0, \infty)$, having density with shape parameter $\alpha > \frac{1}{2}$ given by 
  \begin{align*}
    w(\rvd) = \exp\left(-|\rvd|^\alpha\right).
  \end{align*}
  Then, uniformly in $n$, 
  \begin{align*}
    L(n) = \max_{0 \leq k \leq n} \sup_{z \in S^W(\sqrt{w})_n} (n+1) \lambda_{n+1}(z) \phi_k^2(z) \leq C_1 n^{p(2 \alpha)},
  \end{align*}
  where $p(\cdot)$ is defined by \eqref{eq:p-definition} and $p(2\alpha)$ satisfies
  \begin{align}\label{eq:p2-definition}
  \arraycolsep=1.4pt\def\arraystretch{1.4}
  p(2\alpha) = \max\left\{ \frac{2}{3}, \frac{1}{2\alpha} \right\} = \left\{ \begin{array}{cc} \frac{2}{3}, & \alpha \geq \frac{3}{4} \\ \frac{1}{2\alpha}, & \alpha < \frac{3}{4} \end{array}\right.
  \end{align}
\end{theorem}
The proof is given in Appendix \ref{sec:appendix-csa-c}. Just as with Theorem \ref{thm:bound-csa-b} the dependence on the exponent $p(2\alpha)$ is sharp, as can be observed by manipulating the cases $k=n$, and $k=0$ at $z=0$, by using \eqref{eq:christoffel-half-sim} and Lemmas \ref{lemma:poly-half-bounds} and \ref{lemma:half-varphi-estimate}. The result above holds for the more general case considering a family of polynomials orthonormal under the weight $w(\rvd) = \rvd^{\mu} \exp(-|\rvd|^\alpha)$ with $\alpha > \frac{1}{2}$ and $\mu \geq -\frac{1}{2}$. (That is, the exponent on $n$ is still $p(2\alpha)$, independent of $\mu$.) To establish this, one need only augment the proof in Appendix \ref{sec:appendix-csa-c} to include the $\rvd^\mu$ factor in the weight; the necessary results generalizing the cited Lemmas in \ref{sec:appendix-csa-c} are in \cite{kasuga_orthonormal_2003,levin_orthogonal_2005,levin_orthogonal_2006}.
\end{subtheorem}

%

%% file: tikz/csa-algorithm.tex
\begin{algorithm}[H]
\SetKwInOut{Input}{input}\SetKwInOut{Output}{output}

\Input{Weight/density function $w$ with associated orthonormal polynomial family $\{\phi_i\}_{i\in\Lambda}$, index set $\Lambda$, function $f$}
\Output{Expansion coefficients $\coef$ such that $f\approx\hat{f}=\sum_{i\in\Lambda}\alpha_{i}\phi_{i}(\brv)$}

\BlankLine
Generate $M$ iid samples $\left\{\brv^{(m)}\right\}$ from equilibrium density $v=\frac{d\eqmeas}{d\brv}$\;
Assemble $\bsf$ with entries $f_m = f(\brv^{(m)})$\;
Form $M \times N(\Lambda)$ matrix $\vand$ with entries $\Phi_{m,i} = \phi_{i}(\brv^{(m)})$\;
Compute weights $\bW$ with entries $(W)_{m,m} = N/{\lambda_\Lambda(\brv^{(m)})}$ \;
Compute $\coef^\star  = \argmin_{\coef}\; \|\coef\|_1\quad \text{such that}\quad \|\bW\vand\coef - \bW\bsf\|_2 \le \stoptol$\;

\caption{Christoffel Sparse Approximation (CSA)}\label{alg:csa-algorithm}
\end{algorithm}

%% file: results.tex
\section{Results}\label{sec:results}
In the following we will numerically compare the performance of CSA with other popular sampling strategies for compressive sampling of PCE. Recall that the random variable $Z$ has probability density $w$, and we attempt to recover a sparse expansion from a multi-index dictionary $\Lambda$. We let $n$ denote the maximum univariate degree in $\Lambda$. I.e., 
\begin{align*}
  n = \max_{\lambda \in \Lambda} \max_{j=1, \ldots, d} \lambda_j
\end{align*}
We use the following terms to describe the recovery procedures tested in this section. 
\begin{itemize}
  \item CSA -- This is the algorithm proposed in this paper. We sample $Z_i$ as iid realizations of the weighted pluripotential equilibrium measure. I.e., we sample from the measure with associated density $v(z)$ as described in Sections \ref{sec:univariate-sampling} and \ref{sec:multivariate-sampling}. We define weights $k_i$ as evaluations of the Christoffel function associated to $\Lambda$, and are defined as in \eqref{eq:christoffel-function}.
  \item MC -- This is the ``\naive{}" approach where we sample $Z_m$ iid from the orthogonality density $w$, define the weights $k(z) \equiv 1$, and subsequently solve \eqref{eq:pre-condition-bpdn}.
  \item Asymptotic -- This is a sampling procedure designed to approximate the asymptotic envelope of the PCE basis. Essentially, this method prescribes Chebyshev sampling on tensor-product bounded domains and uniform sampling inside a hypersphere of certain radius when $Z$ is a Gaussian random variable. We detail these asymptotic cases below in Sections \ref{sec:chebyshev-sampling} and \ref{sec:hermite-sampling}; these methods were proposed in \cite{Hampton_D_JCP_2015} building on ideas from \cite{Rauhut_W_JAT_2012,Yan_GX_IJUQ_2012}.
\end{itemize}

The following sections describe the ``asymptotic" sampling strategy mentioned above.

\subsection{Asymptotic sampling for Beta random variables: Chebyshev sampling}\label{sec:chebyshev-sampling}
In the pioneering work in~\cite{Rauhut_W_JAT_2012} it was shown that for
polynomials which are orthogonal to a certain class of weight functions $\pbwt$, 
defined on bounded domains, that random samples drawn from the Chebyshev measure can be used with preconditioned
$\ell^1$-minimization to produce Vandermonde matrices with small restricted isometry properties. The major idea is that the weighted polynomials corresponding to the preconditioned have envelopes that are constant, and are thus bounded. Using results in \cite{Rahut_TFNMSR_2010}, this boundedness can be used to show RIP properties.
Since Jacobi polynomials (orthogonal with respect to the Beta distribution on $[-1,1]$ with weight function 
$\pbwt(\brv)=(1-\rv)^{\alpha}(1+\rv)^{\beta}$ for $\alpha,\beta\geq -\frac{1}{2}$) have an envelope that behaves in an analogous fashion, then the Chebyshev sampling and corresponding weighting likewise produces a sampling matrix with good RIP properties.

Let $V=(v_1,\ldots,v_d)$ be a vector of $d$ independent uniform random
variables on $[0,1]$ then the Chebyshev sampling method generates samples according to
$$Z=\cos(\pi V)$$
These samples are then used with the preconditioning weights
$$k(\brv) = \prod_{i=1}^d(1-\rv_i^2)^{1/2}\pbwt(\brv)$$
to solve the preconditioned $\ell^1$-minimization problem~\eqref{eq:pre-condition-bpdn}. This is the ``asymptotic" sampling strategy prescribed in \cite{Hampton_D_JCP_2015}.

\subsection{Asymptotic sampling for Gaussian random variables}\label{sec:hermite-sampling}
The Chebyshev sampling method is not applicable to unbounded random variables. 
In~\cite{Hampton_D_JCP_2015} an asymptotic sampling scheme was proposed for sparse 
Hermite polynomial approximation of functions parameterized by Gaussian random variables.
Asymptotic sampling draws random samples from an envelope that behaves like the
asymptotic (in order) envelope for the Hermite polynomials as the polynomial degree $n$ goes to infinity. 

Let $Y=(y_1,\ldots,y_d)$ be a vector of $d$ independent normally distributed
random variables and let $u$ be a independent uniformly distributed random variable on $[0,1]$ then we generate samples according to
\begin{align*}
\brv=\frac{Y}{\lVert Y\rVert_2}ru^{1/d}
\end{align*}
where $r=\sqrt{2n+1}$ and $n$ is the order of the total degree polynomial space.  These asymptotic Gaussian samples are uniformly distributed on the $d$-dimensional ball of radius $r$ and are used with precondition weights generated from
$$k(\brv) = \exp(-\lVert \brv \rVert_2^2/2 ).$$ Note the radius $r \approx \sqrt{2 n}$ here is essentially the same as the CSA radius prescription in Section \ref{sec:csa-hermite-multivariate-sampling}. (In contrast to Gaussian asymptotic sampling, the CSA algorithm does not sample from the uniform density.)

\subsection{Manufactured sparse solutions}
In this section we investigate the performance of the CSA method when used to recover randomly 
generated $s$-sparse vectors $\coef$ from noiseless data, such that $\vand\coef=\boldf$. Specifically we 
construct a $s$-sparse vector by selecting $s$
non-zero entries uniformly at random and set the magnitude of the non-zero 
elements to be draws from the standard normal 
distribution. Samples $\{f_m=p(\brvsamp{m})\}_{m\in[M]}$ are then taken from the polynomial 
$p(\brv)=\sum_{i=1}^N\alpha_i\phi_i(\brv)$ and Basis Pursuit is used 
to recover the coefficients $\coef$, by solving 
\begin{equation*}
\coef^\star  = \argmin_{\coef}\; \|\coef\|_1\quad \text{such that}\quad \wtmat\vand\coef=\wtmat\boldf.
\end{equation*}
Note that this minimization procedure is not the same one as \eqref{eq:csa-weighted-minimization} considered in Theorem \ref{thm:csa-convergence}. Nevertheless, we observe very good performance, which we can attribute to the fact that the matrix $\bR$ is close to the identity.

To compare the performance of the CSA method we measure the probability of `successfully'
recovering $s$-sparse vectors of varying sparsity $s$, dictionary sizes $N = |\Lambda|$, and number of samples $M$. 
We use 100 random trials and deem recovery to be successful when $\norm{\coef-\coef^\star}{2}/\norm{\coef}{2}\le 0.01$.
Here and throughout the remainder of the paper we use least angle regression (LAR) with the LASSO modification~\cite{Efron_HJT_AS_2004}
to solve the basis pursuit and basis pursuit denoising problems.

Figures~\ref{fig:uniform-transition-plots}-\ref{fig:exponential-transition-plots} demonstrate the dependence of successful recovery on the sparsity $s$, the number of samples $M$, and the number of basis terms $N$ for four types of PCE expansions. 
Specifically we consider $d$-dimensional random vectors $\rv$ whose components are iid Uniform, $\text{Beta}(\alpha=2,\beta=5)$, 
standard normal, and exponential random variables. These correspond to a PCE basis of (tensor-product) Legendre, Jacobi, 
Hermite, or Laguerre polynomials. For both bounded variable types
(Legendre, Jacobi) the random variables take values $\brv\in[-1,1]^d$.
We plot the probability of recovery as a function of the number of samples normalized
by the number of basis terms, i.e. $M/N$ (x-axis) and as a function of the sparsity normalized by the number of samples, i.e. $s/M$ (y-axis). As is standard in the compressive sampling literature we restrict attention to situations when the ratios $M/N$ and $s/M$ are less than 1. Such situations are also typical of uncertainty quantification studies when the computational expense of a simulation model limits the number of samples that can be taken, often resulting in $M\ll N$.

In each of the plots there is a sharp phase transition between the successful and unsuccessful recovery, which is a well-known effect. Effective sampling and pre-conditioning strategies can be judged based upon the location of this phase transition. (For a fixed $N/M$ the transition is ideally located at large values of $s/M$, and likewise for fixed $s/M$ the transition is ideally located at large values of $N/M$.)

\begin{figure}[ht]
\begin{center}
\includegraphics[width=0.325\textwidth]{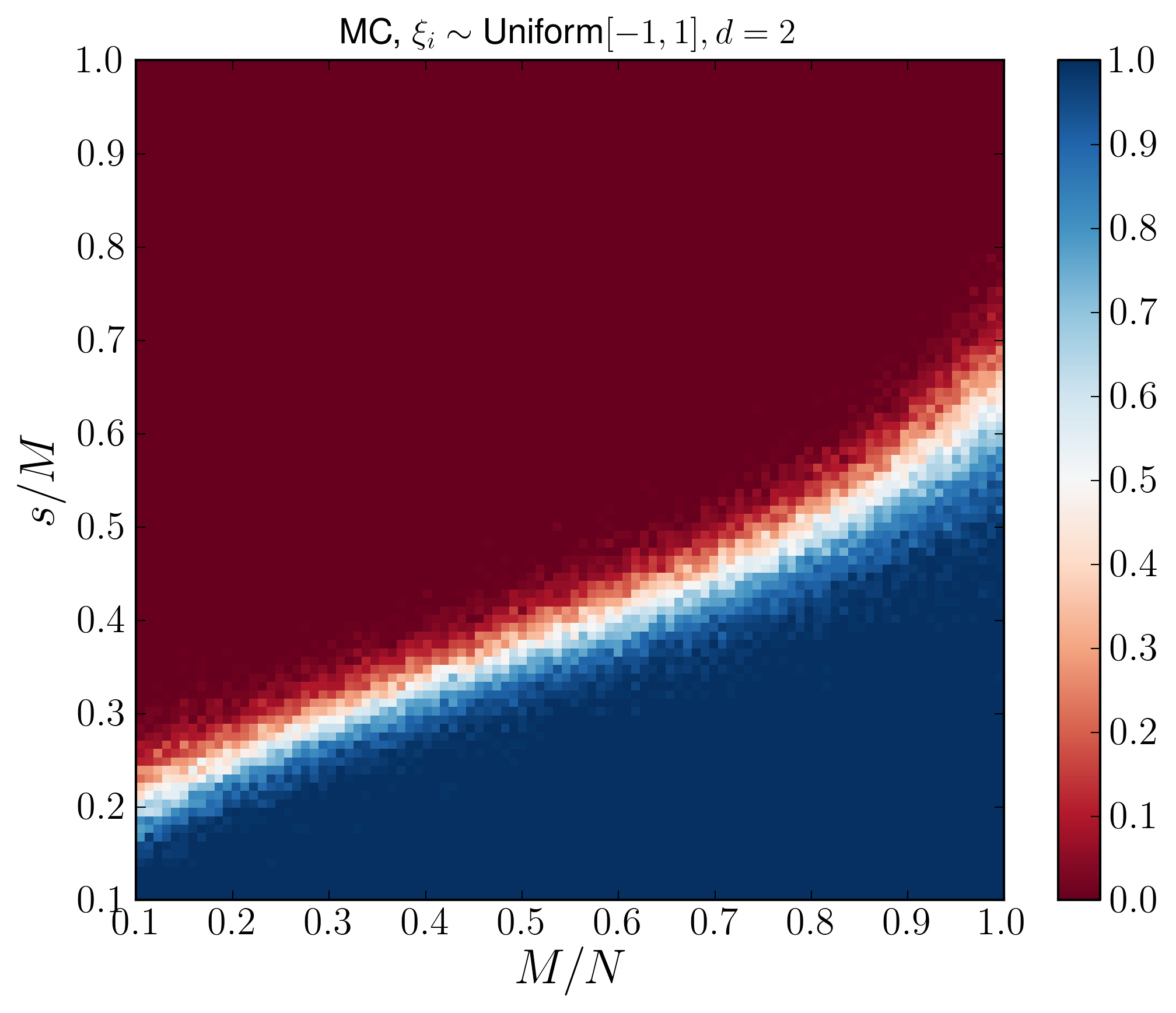}
\includegraphics[width=0.325\textwidth]{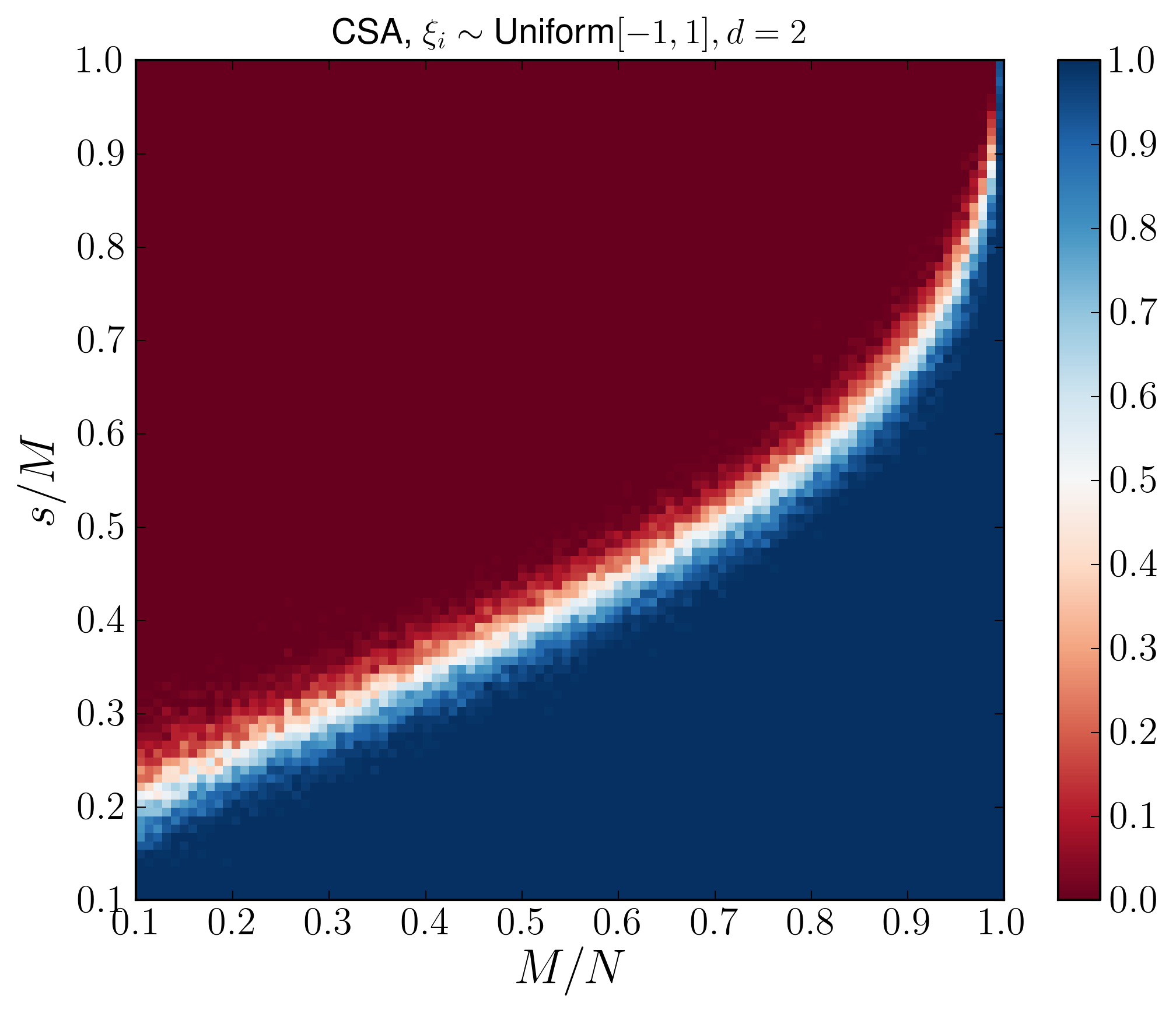}
\includegraphics[width=0.325\textwidth]{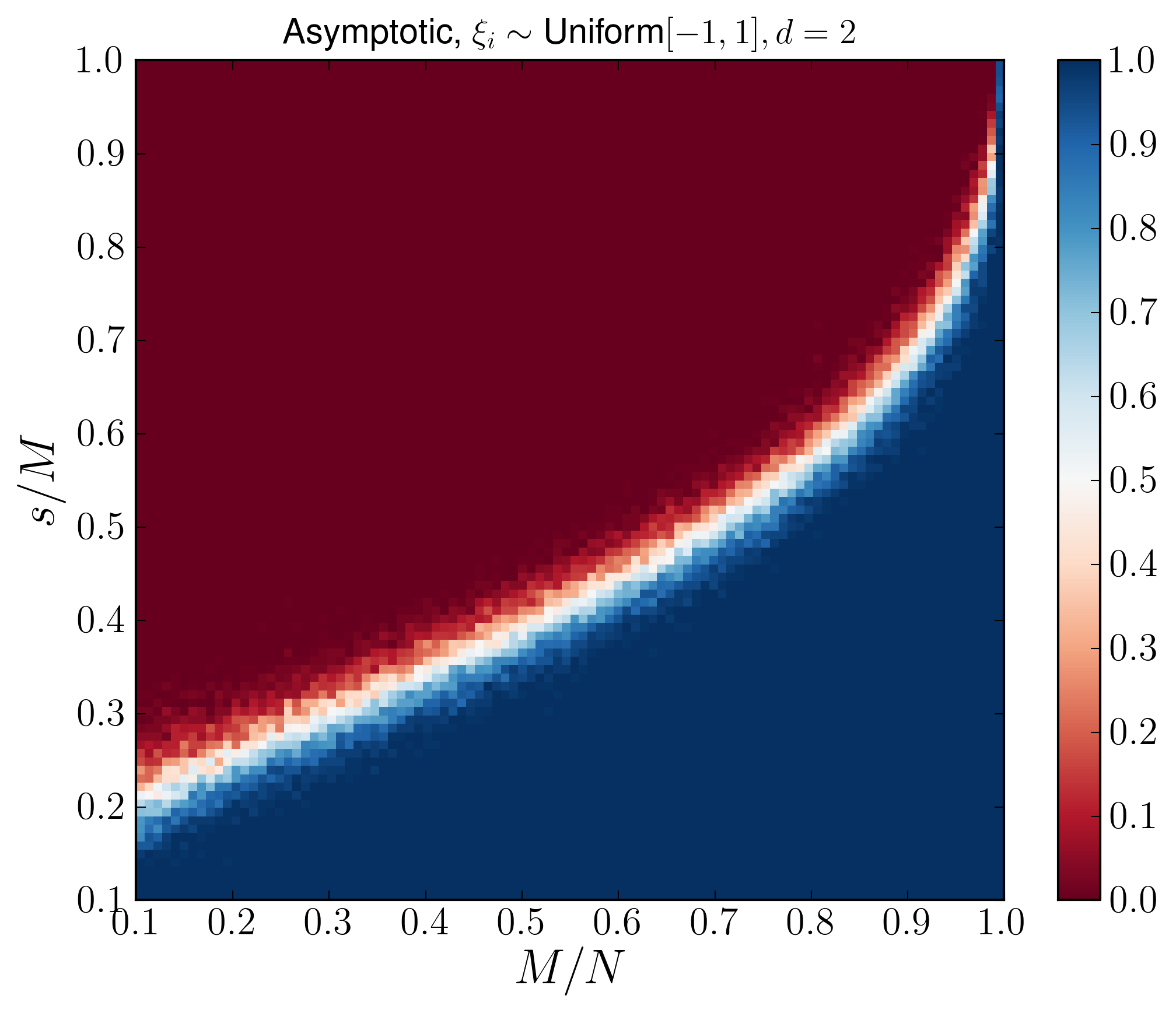}\\
\includegraphics[width=0.325\textwidth]{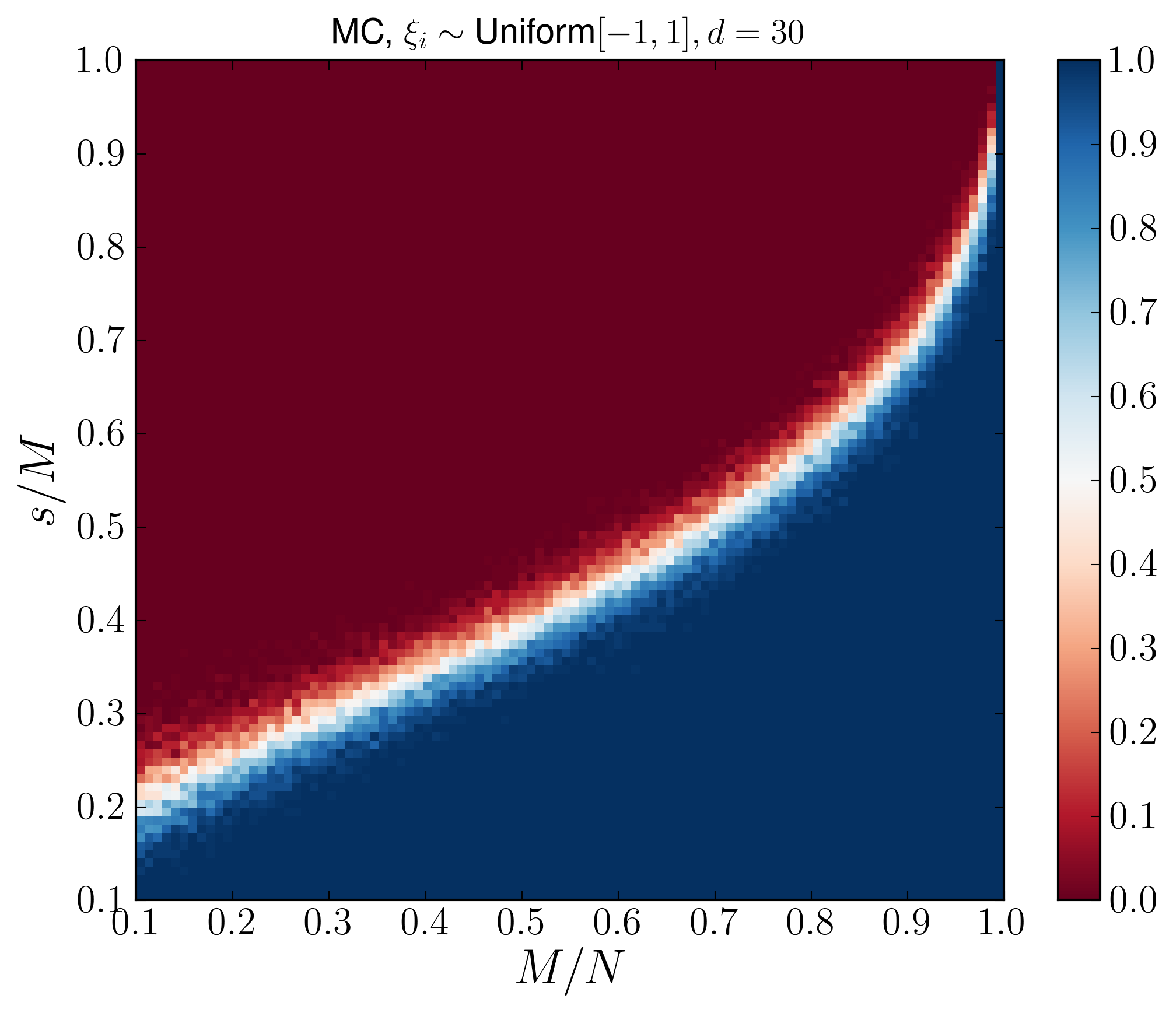}
\includegraphics[width=0.325\textwidth]{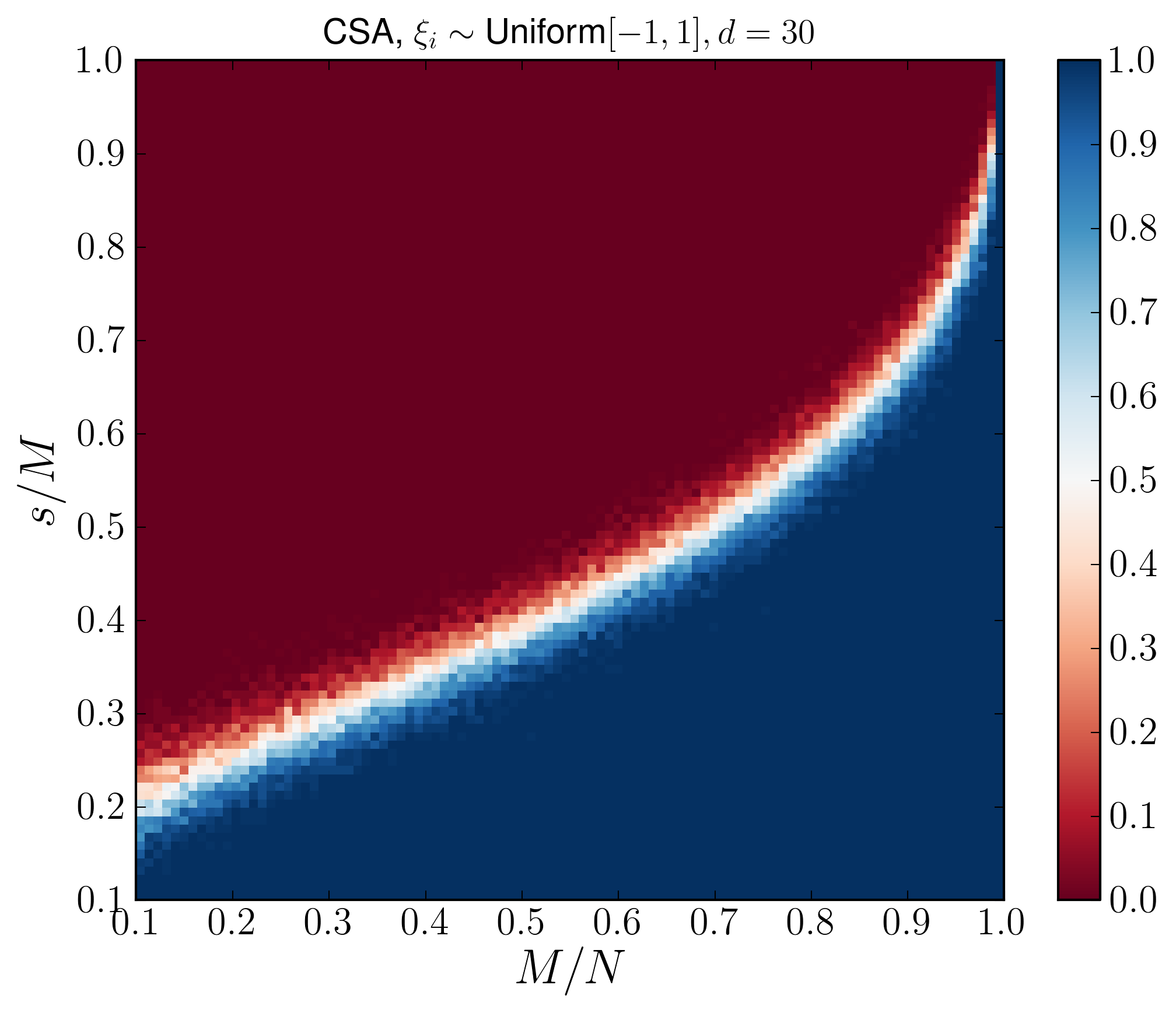}
\includegraphics[width=0.325\textwidth]{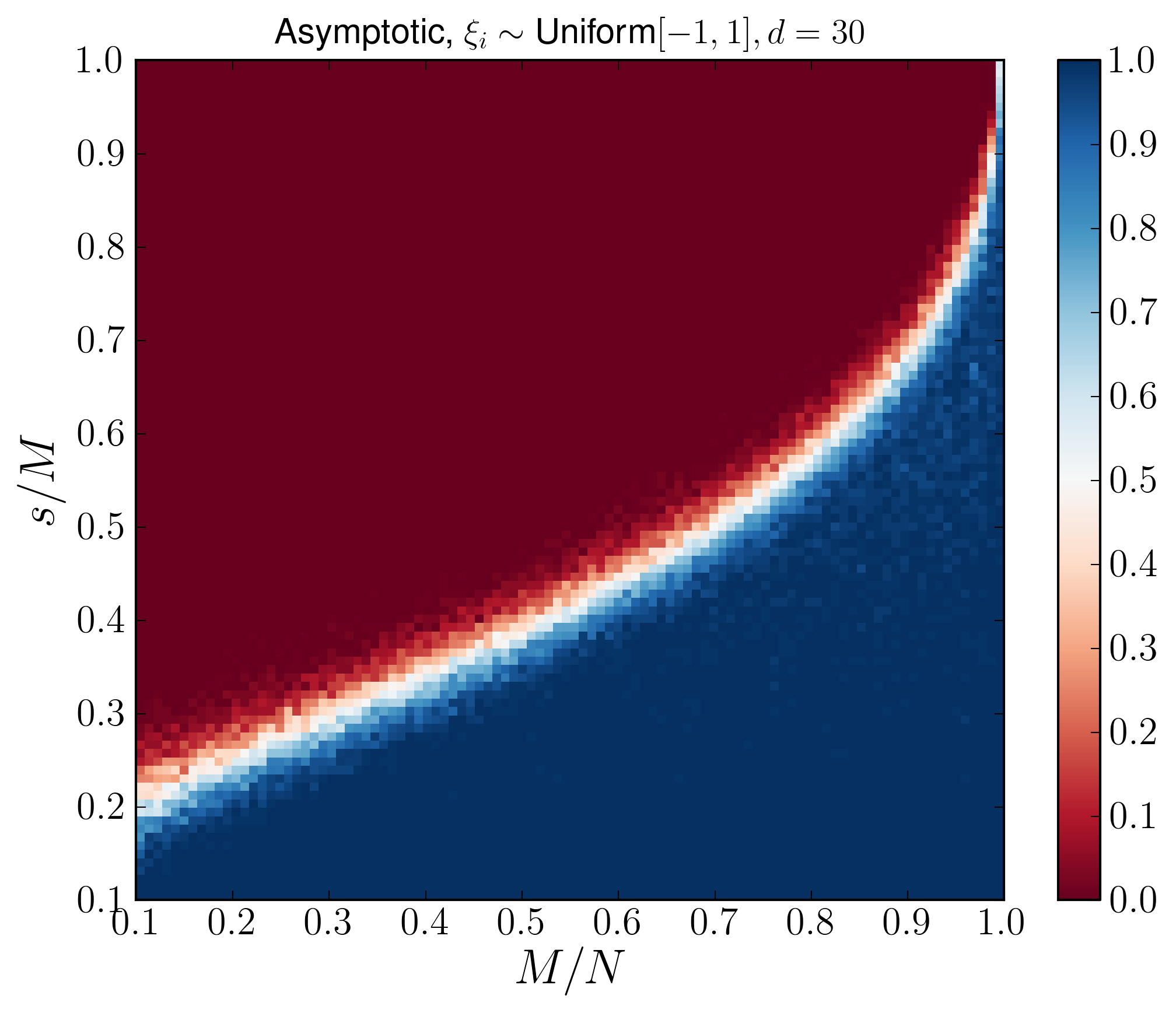}
\end{center}
\caption{Transition plots for uniform random variables for $d=2$ (top rows) and $d=30$ (bottom row).
The left column corresponds to sampling from the random variable density $\pbwt$, the middle column 
the CSA method and the right column asymptotic sampling.}
\label{fig:uniform-transition-plots}
\end{figure}
For all low-dimensional and high-degree situations considered, CSA has a high rate of recovery and performs significantly better than than probabilistic sampling according to the density $\pbwt(\rvd)$, which exhibits almost no recovery. CSA also exhibits similar rates
of recovery to standard Monte Carlo sampling (MC) when applied to high-dimensional, low-degree polynomials associated with uniform and Gaussian variables (Figures~\ref{fig:uniform-transition-plots} and \ref{fig:gaussian-transition-plots}, respectively). However recovery of CSA is slightly worse than MC when recovering high-dimensional Laguerre polynomials (Figure~\ref{fig:exponential-transition-plots}), and significantly worse than MC when recovering Jacobi polynomials (middle column of Figure~\ref{fig:beta-transition-plots}). We attribute the degradation in the performance of CSA for high-dimensions to the low-degree polynomials used in this setting. The Christoffel function induces only approximate orthogonality for finite polynomial degree. In the univariate case, as the polynomial degree is reduced the magnitude of the error induced by the the approximate orthogonality (i.e., the deviation of $\bR$ from $\bs{I}$) increases as predicted by Theorem~\ref{thm:csa-convergence}. 

The approximate orthogonality of the Christoffel function for low-degree polynomials seems to have little effect on recovery of the Hermite and Legendre polynomials. A possible explanation for the success of this procedure for Legendre polynomials is Corollary \ref{cor:legendre-csa-convergence}: The univariate Legendre polynomials actually remain an orthogonal family when weighted by the Christoffel function under the equilibrium measure \cite{bos_orthogonality_2015}. 

Similar to CSA, the bounded and Gaussian asymptotic sampling methods achieve higher-rates of recovery than probabilistic sampling when approximating low-dimensional, high-degree polynomials. However unlike CSA there is no one scheme that can be applied to Uniform, Beta, Normal and Exponential random variables. Indeed the authors are unaware of any preconditioning scheme for Exponential variables, Moreover the error in the approximation recovered by the asymptotic bounded sampling method for Beta variables increases with dimension. When $d=30$ the asymptotic bounded sampling method fails to recover any polynomials regardless of the sparsity or the number of samples used.

It is worth noting that case of Legendre polynomials sampled by Chebyshev distribution we have a complete independence of the order of approximation, which agrees with previous results in~\cite{Rauhut_W_JAT_2012}. However there are numerical results in \cite{Hampton_D_JCP_2015,Yan_GX_IJUQ_2012} showing almost no recovery when using the Chebyshev sampling method in high-dimensions.

With the help of the authors of~\cite{Hampton_D_JCP_2015} we have verified that the poor performance exhibited in the aforementioned papers is a result of numerical issues associated with the authors use of the $\ell^1$-minimization solver in SparseLab~\cite{Donoho_ST_manual}. Specifically, the authors of \cite{Hampton_D_JCP_2015} were using more lenient optimization tolerances, and when these tolerances were made tighter to match our optimization tolerance, the authors of \cite{Hampton_D_JCP_2015} obtained results consistent with Figure~\ref{fig:uniform-transition-plots}.

\begin{figure}[ht]
\begin{center}
\includegraphics[width=0.325\textwidth]{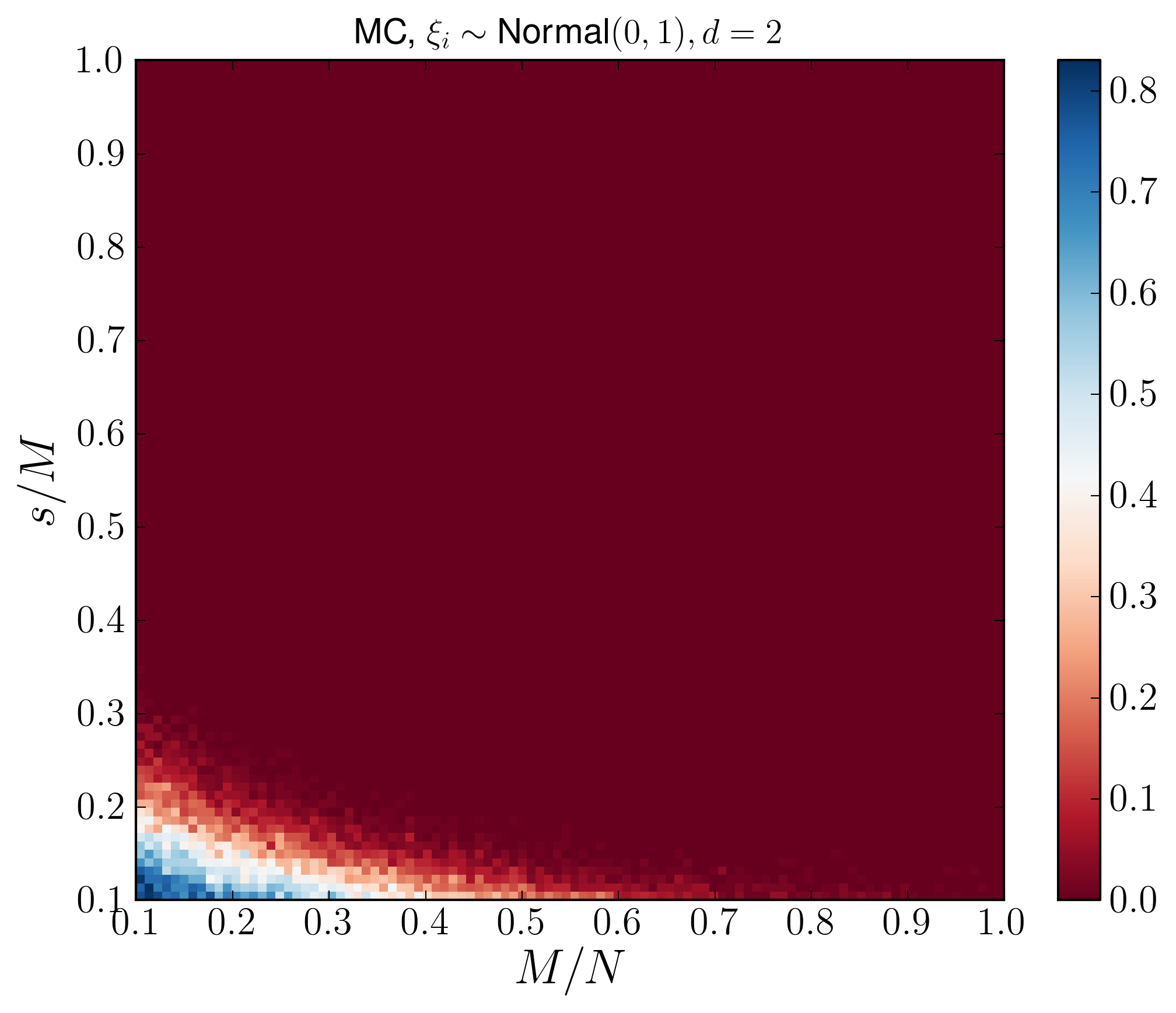}
\includegraphics[width=0.325\textwidth]{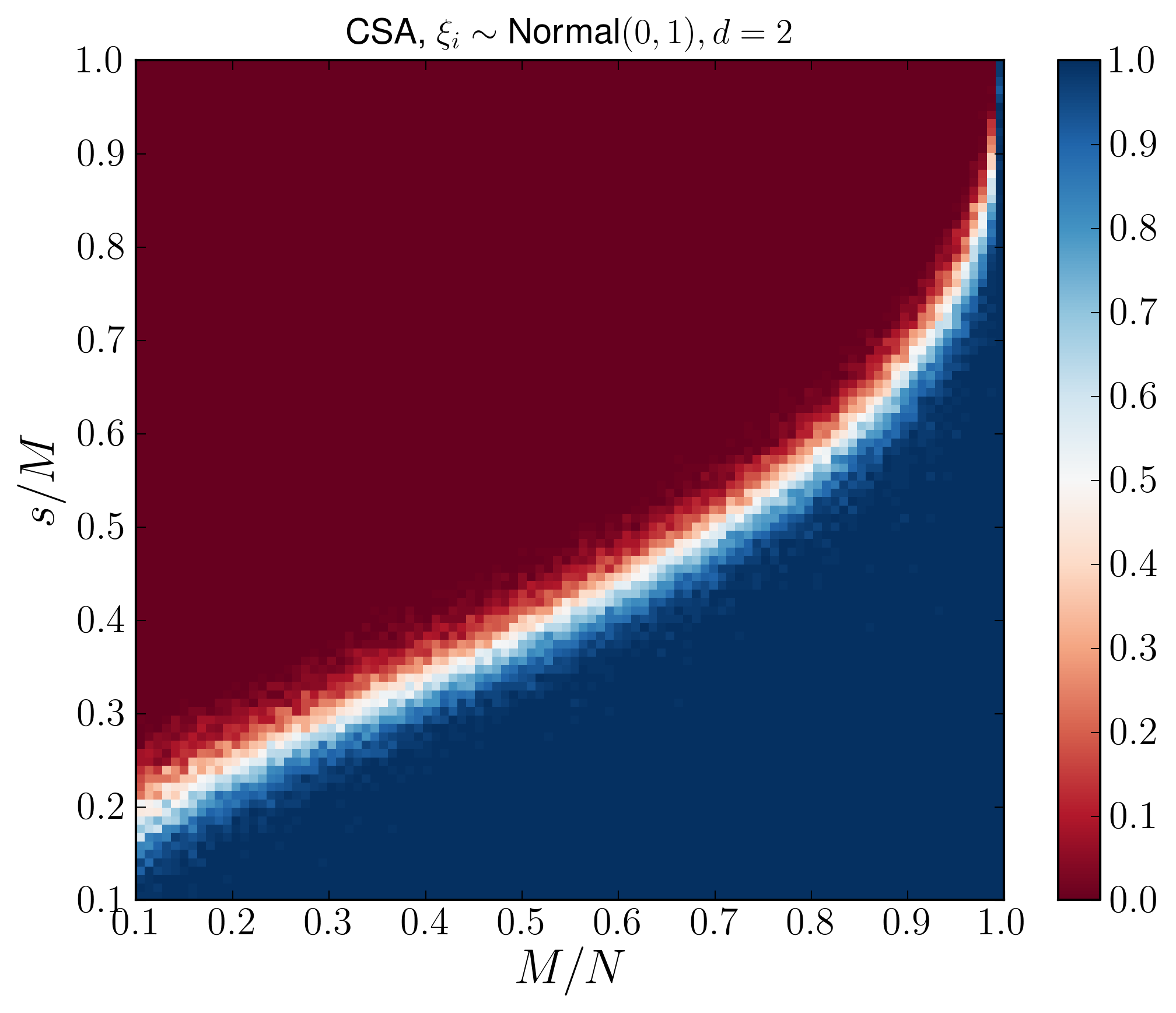}
\includegraphics[width=0.325\textwidth]{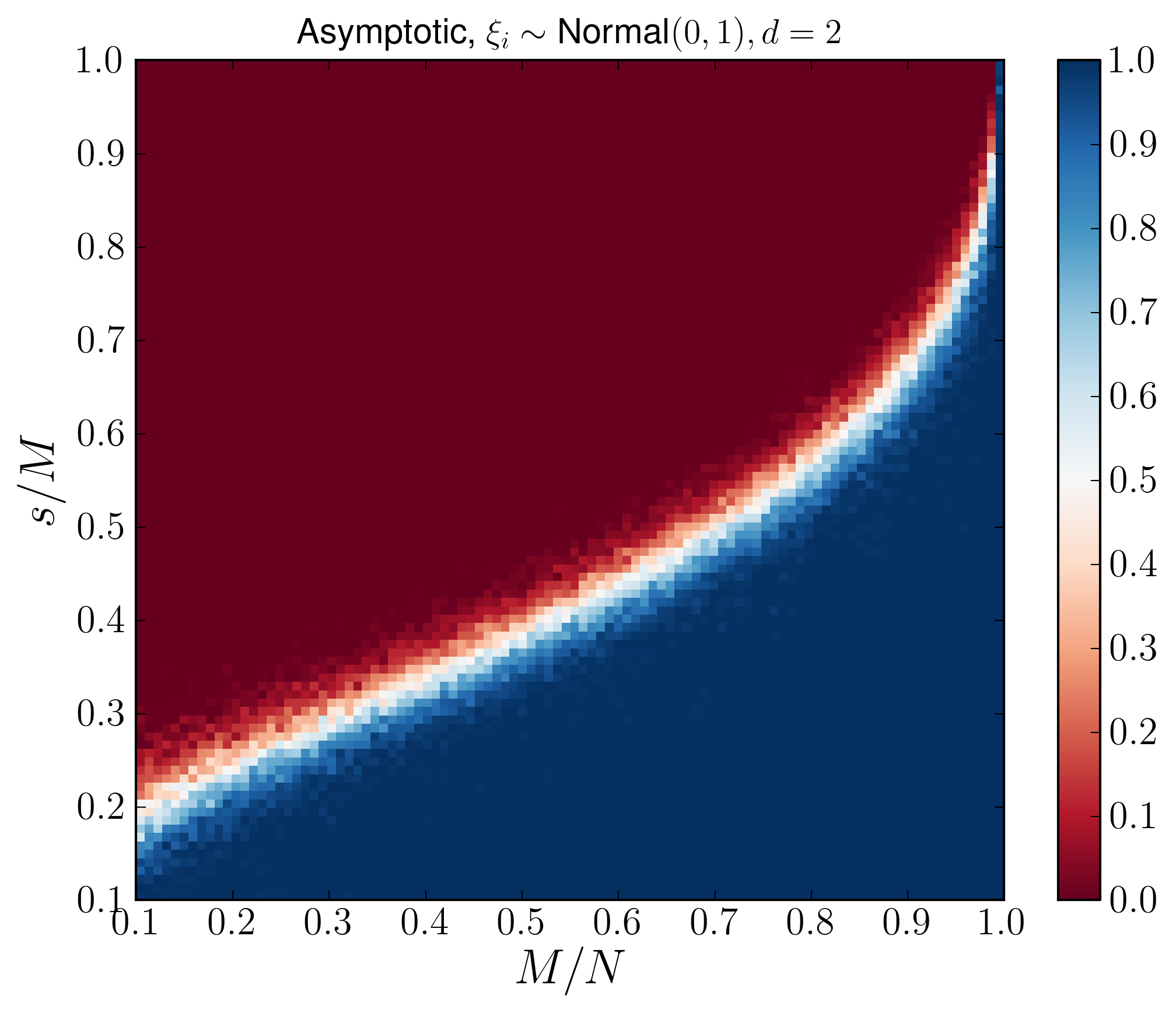}\\
\includegraphics[width=0.325\textwidth]{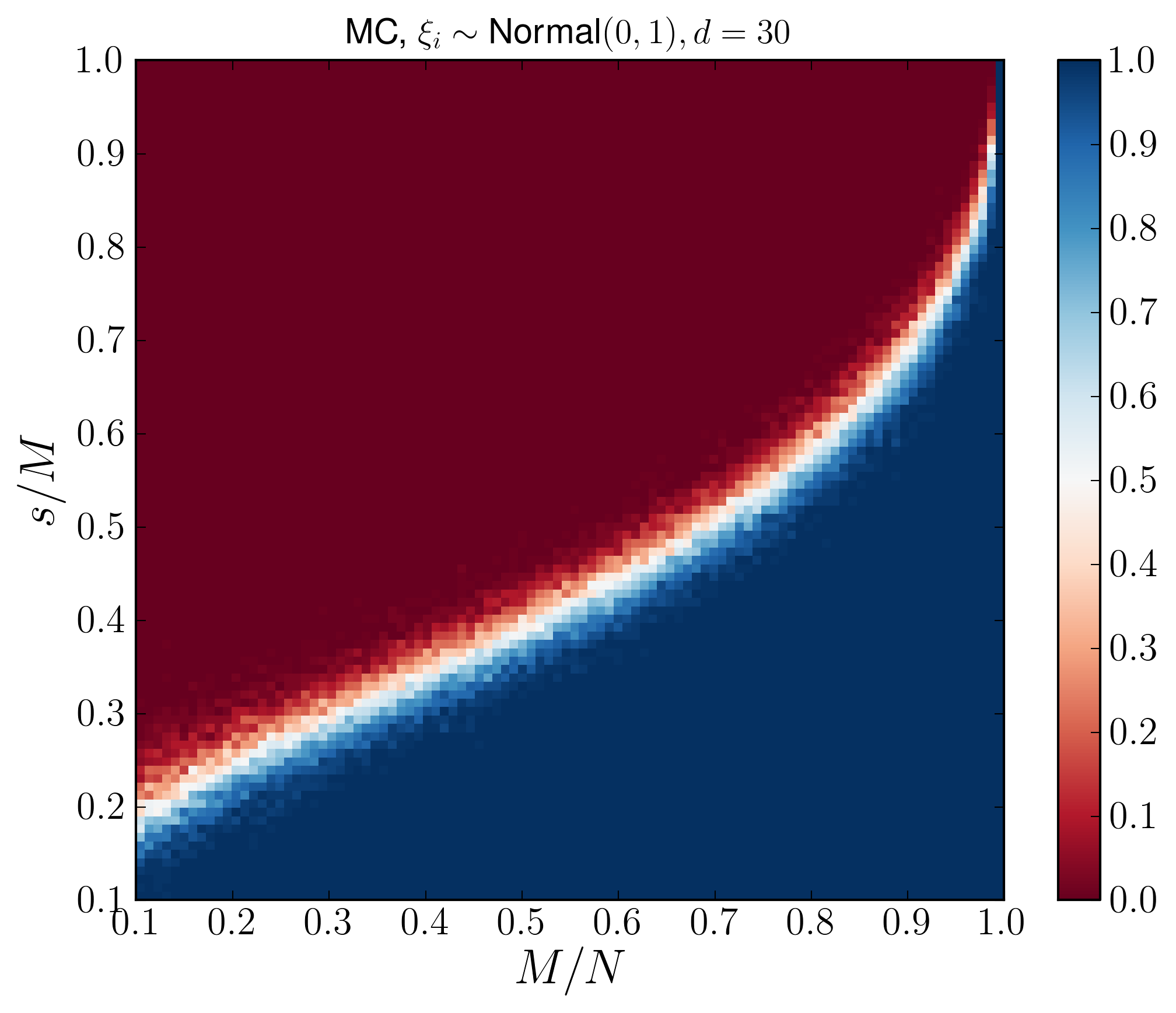}
\includegraphics[width=0.325\textwidth]{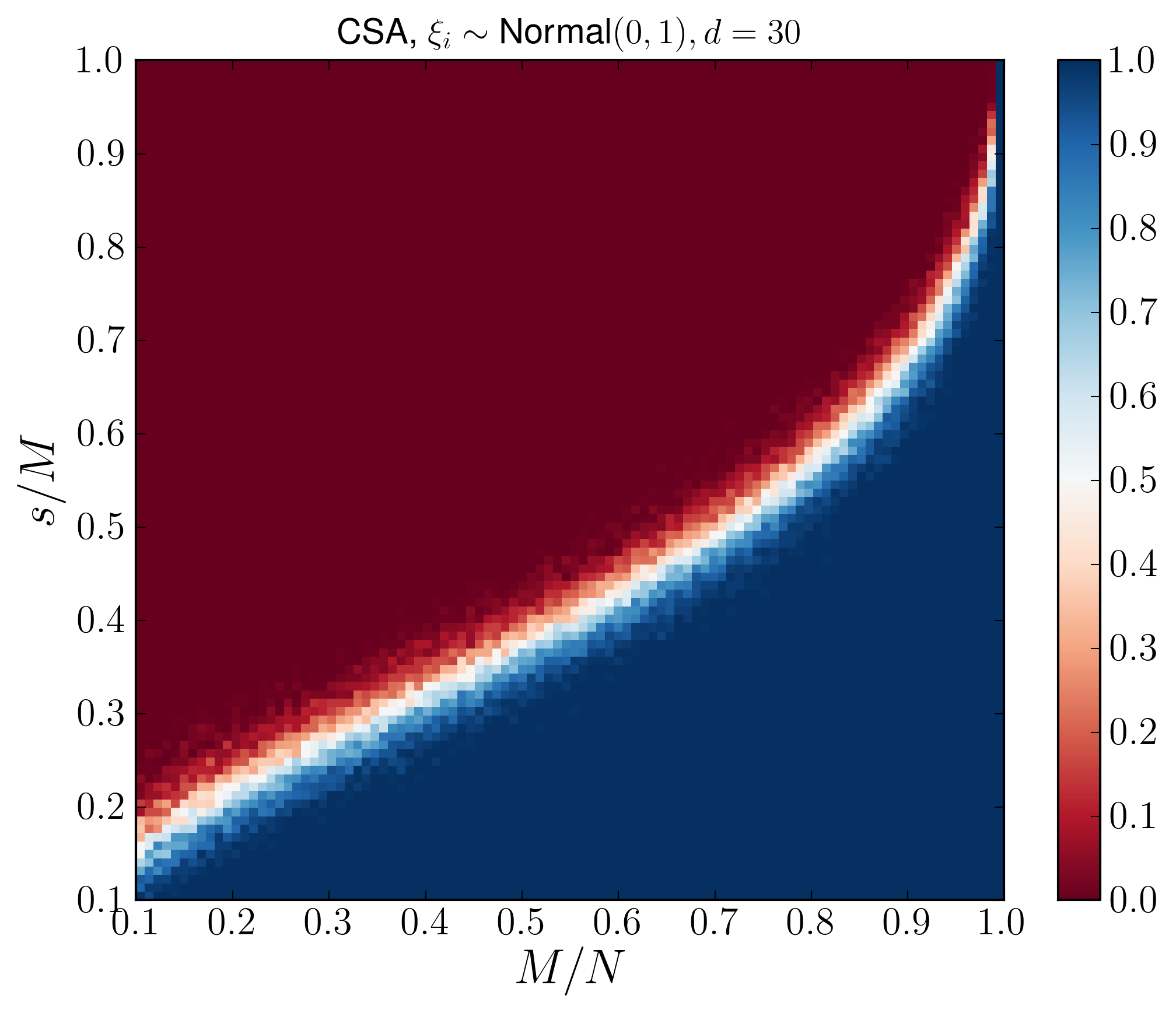}
\includegraphics[width=0.325\textwidth]{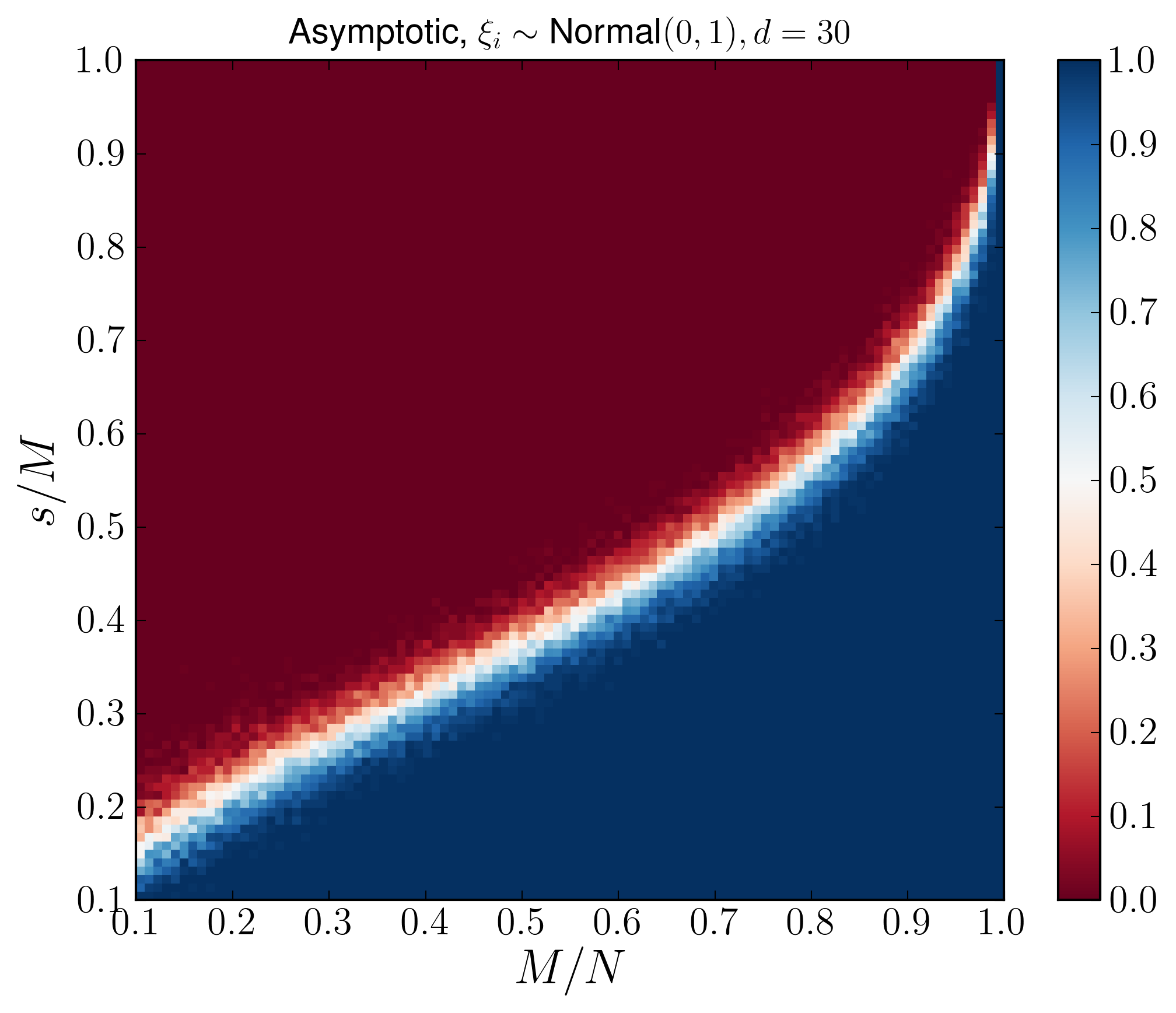}
\end{center}
\caption{Transition plots for normally distributed random variables for $d=2$ (top rows) and $d=30$ (bottom row). The left column corresponds to sampling from the random variable density $\pbwt$, the middle column the CSA method and the right column asymptotic sampling.}
\label{fig:gaussian-transition-plots}
\end{figure}  
 
\begin{figure}[ht]
\begin{center}
\includegraphics[width=0.325\textwidth]{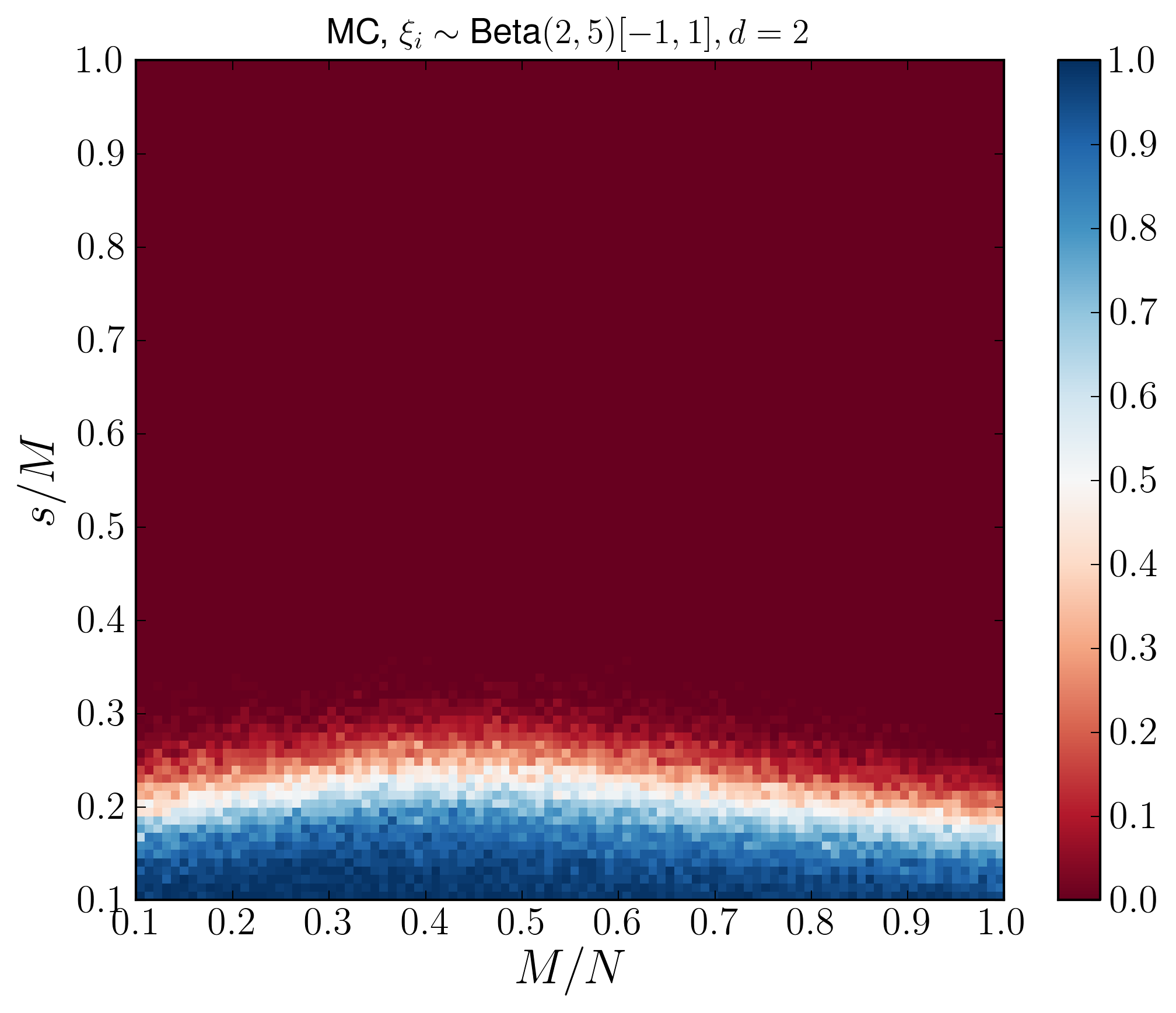}
\includegraphics[width=0.325\textwidth]{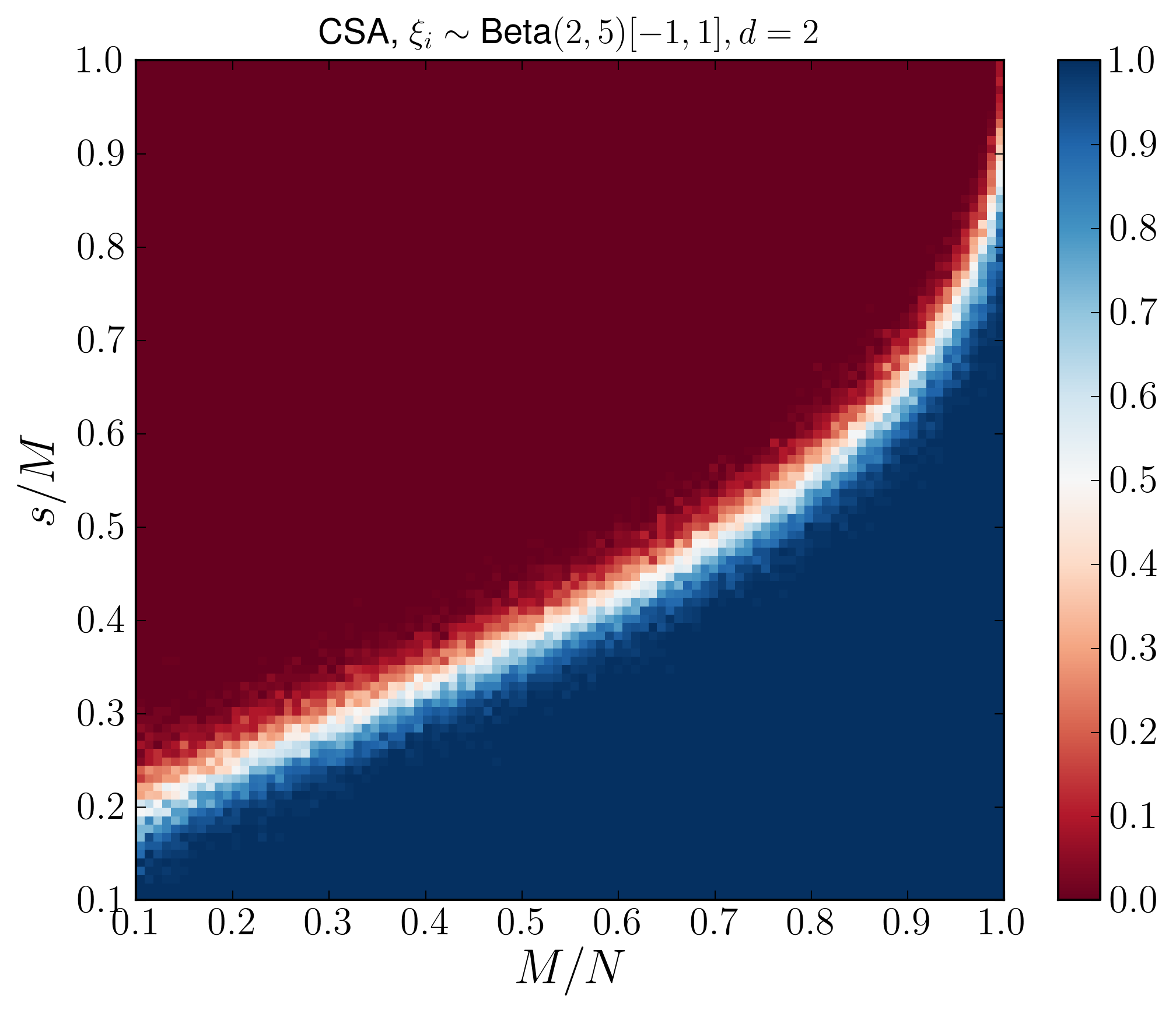}
\includegraphics[width=0.325\textwidth]{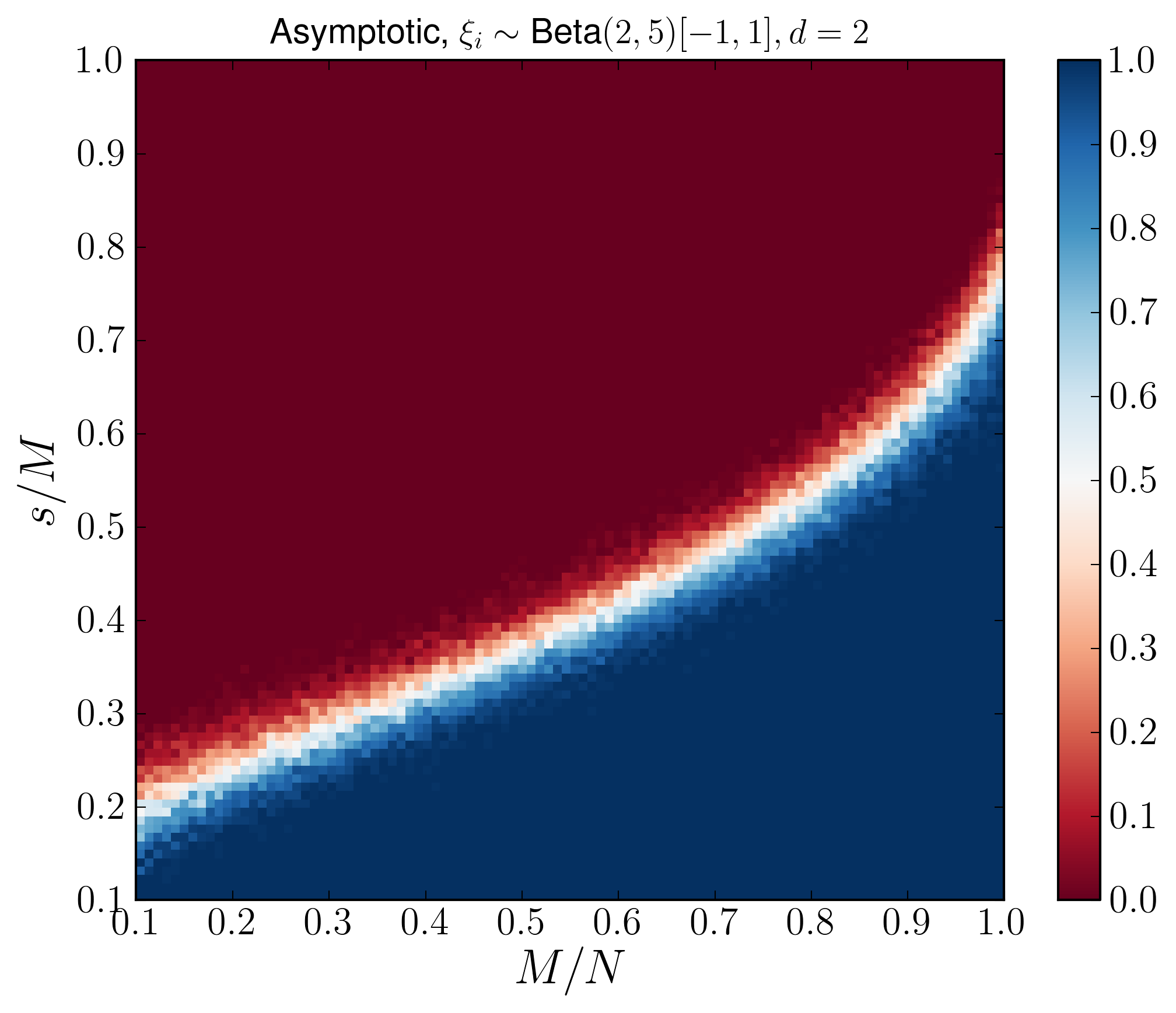}\\
\end{center}
\begin{flushleft}
\includegraphics[width=0.325\textwidth]{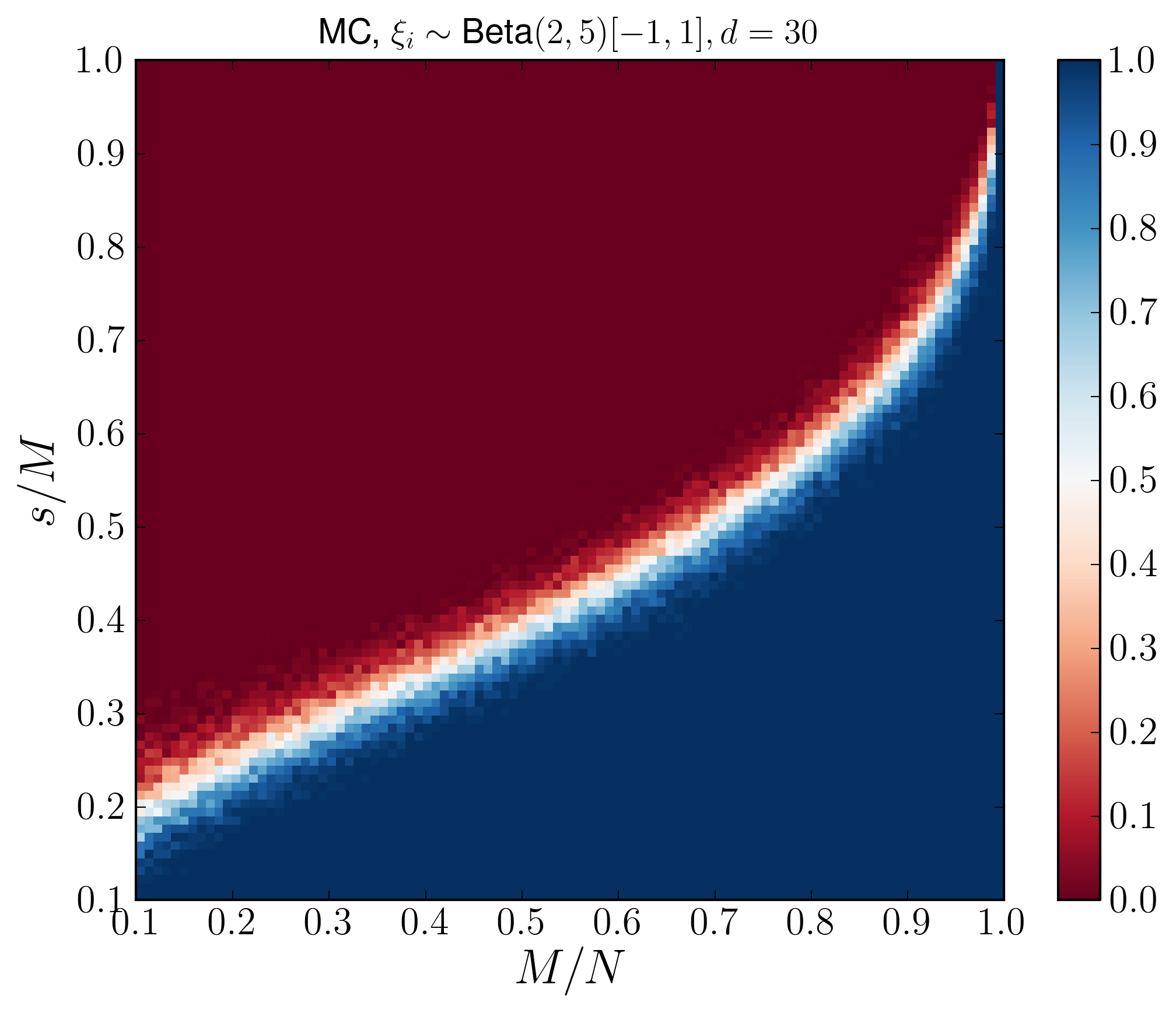}
\includegraphics[width=0.325\textwidth]{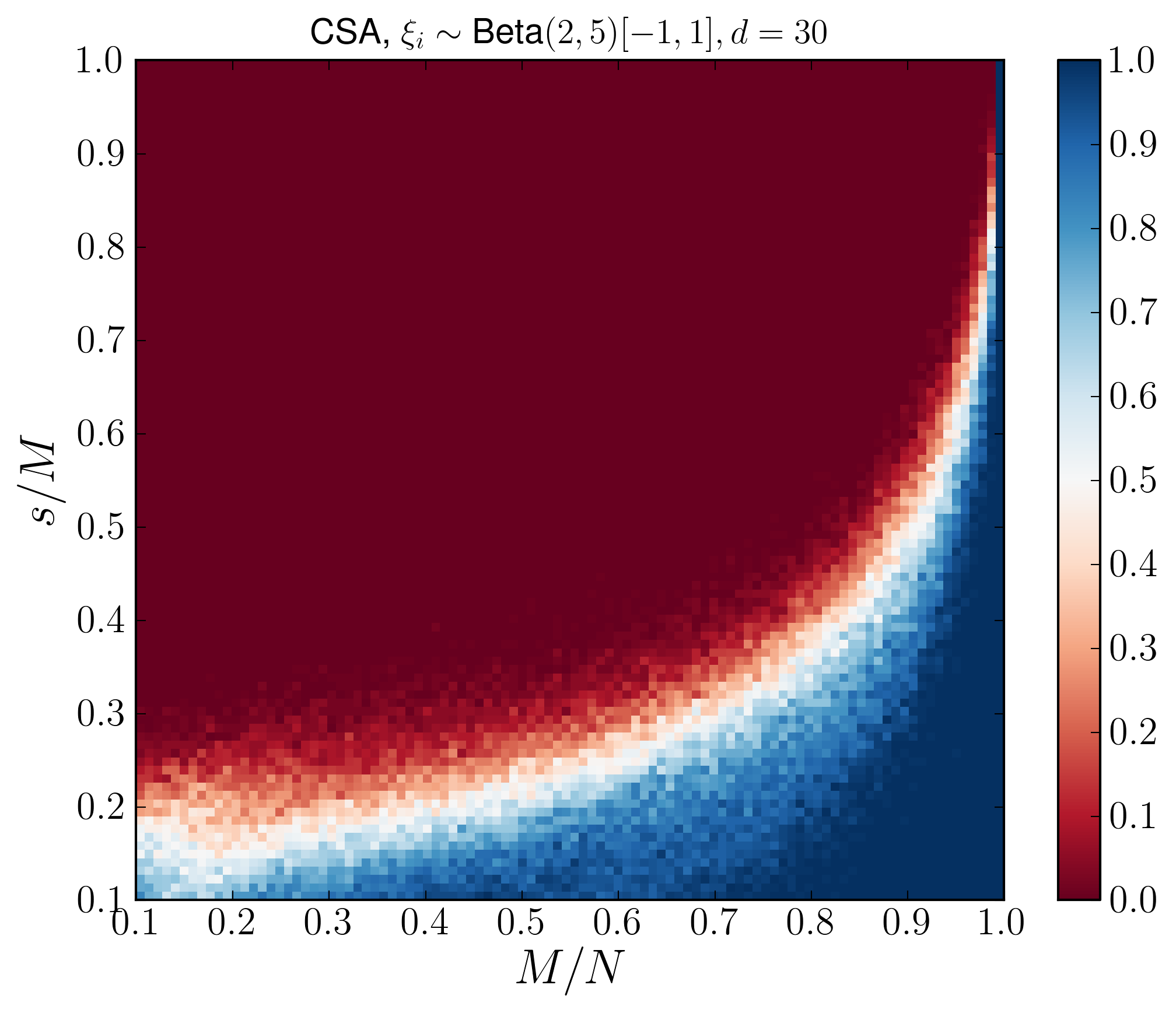}
\end{flushleft}
\caption{Transition plots for Beta(2,5) distributed random variables for $d=2$ (top rows) and $d=30$ (bottom row).  The left column corresponds to sampling from the random variable density $\pbwt$, the middle column the CSA method and the right column asymptotic sampling. No plot is shown for asymptotic sampling for $d=30$ because no recovery is obtained.}
\label{fig:beta-transition-plots}
\end{figure}  

\begin{figure}[ht]
\begin{center}
\includegraphics[width=0.325\textwidth]{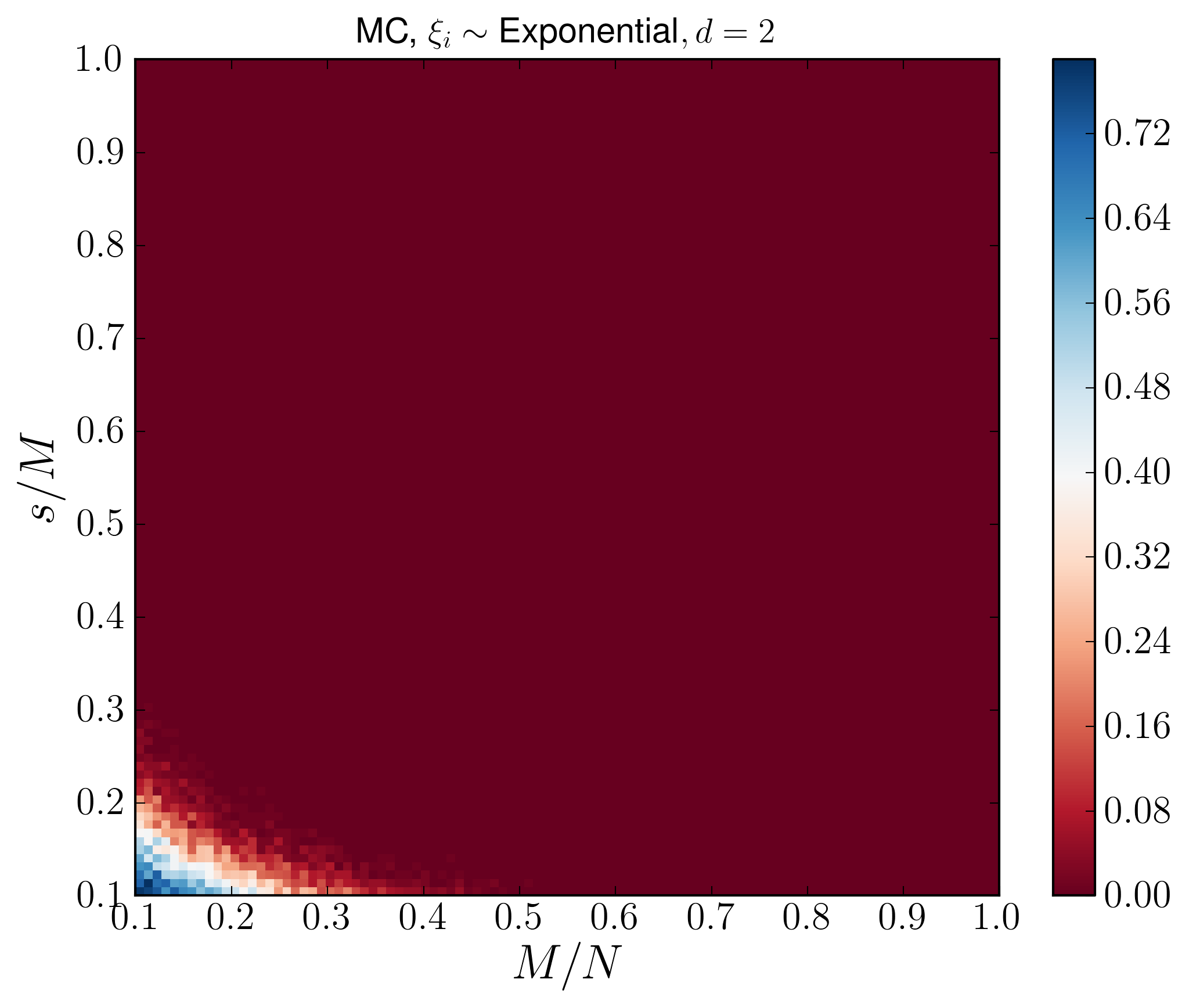}
\includegraphics[width=0.325\textwidth]{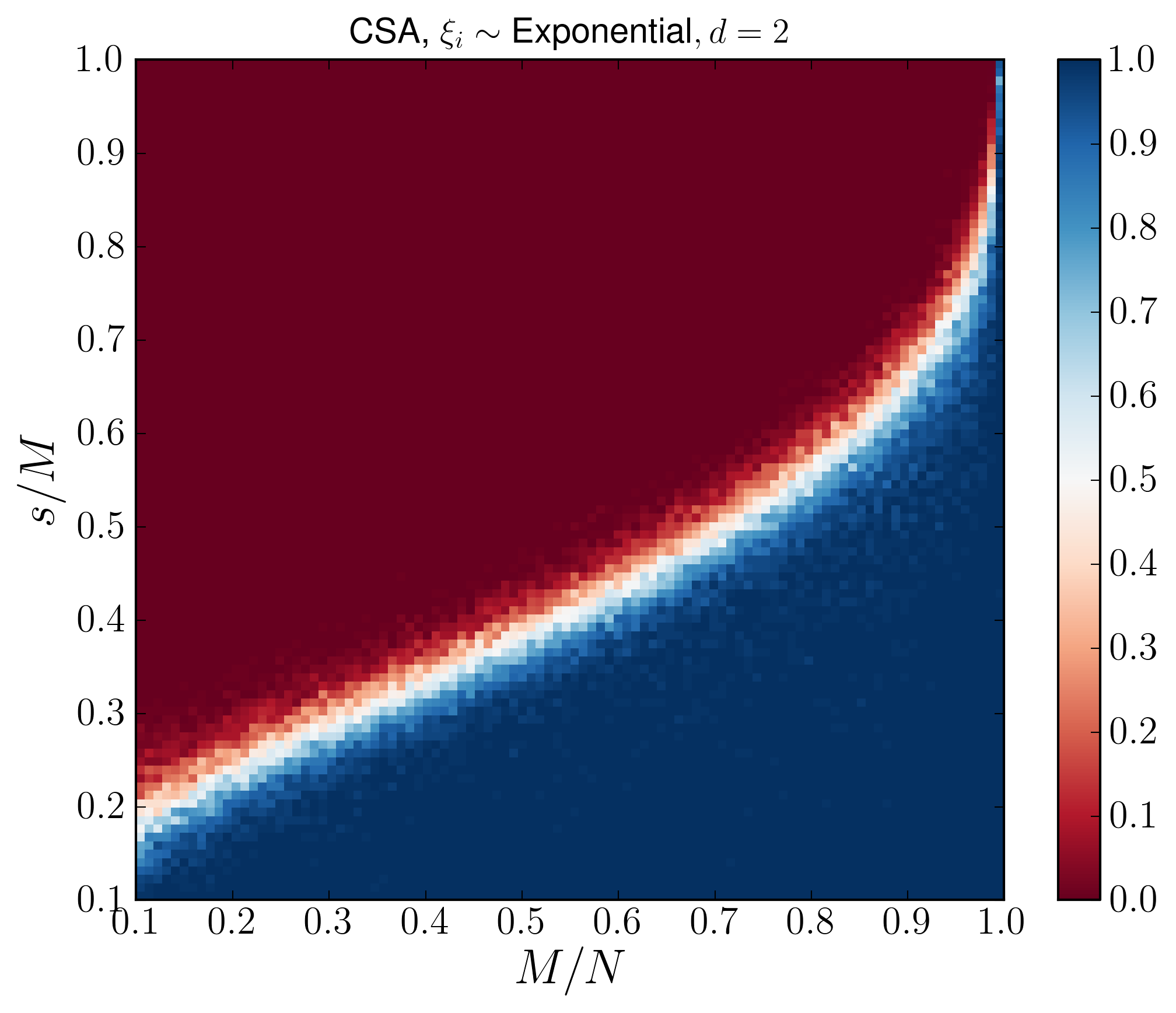}\\
\includegraphics[width=0.325\textwidth]{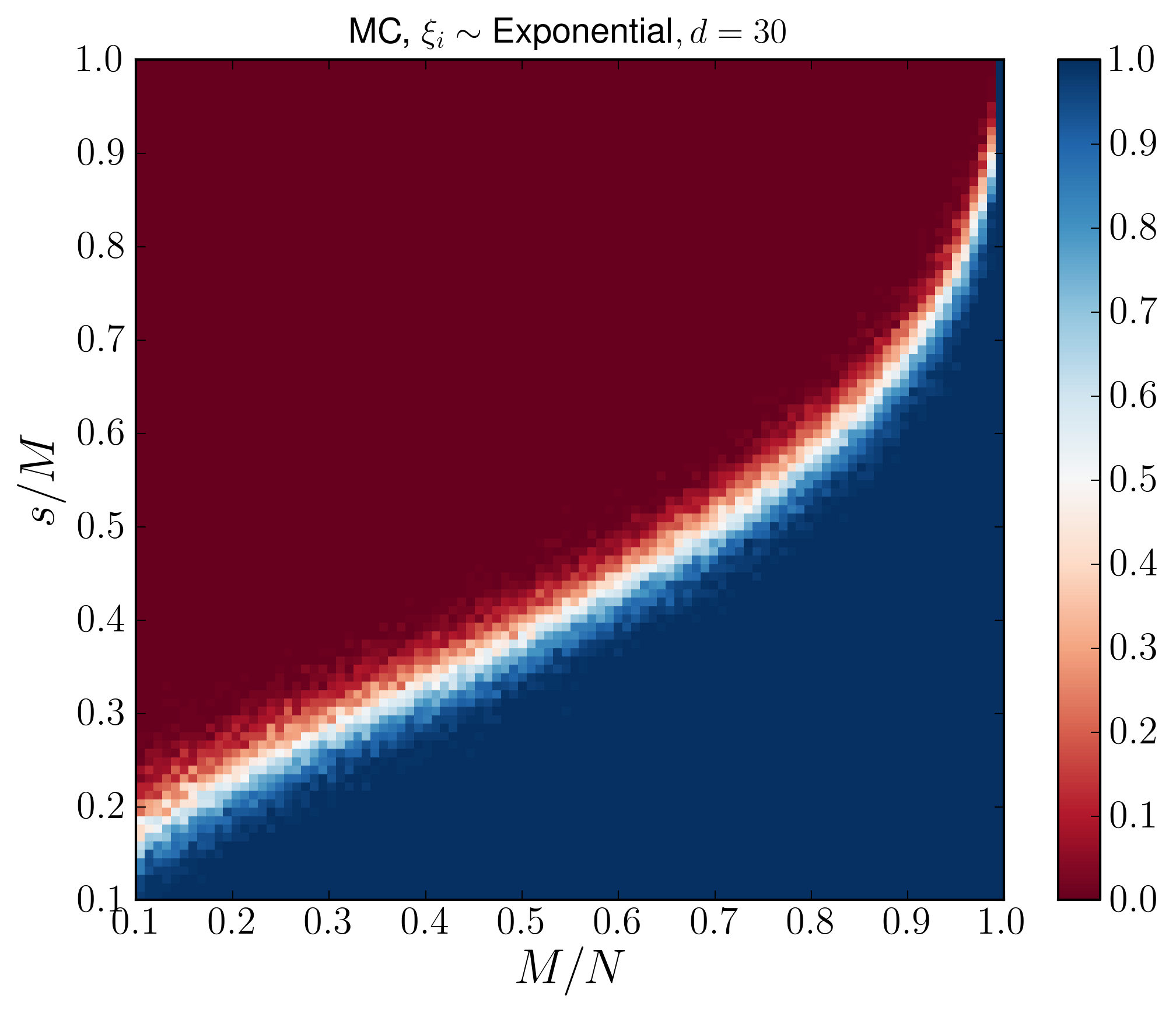}
\includegraphics[width=0.325\textwidth]{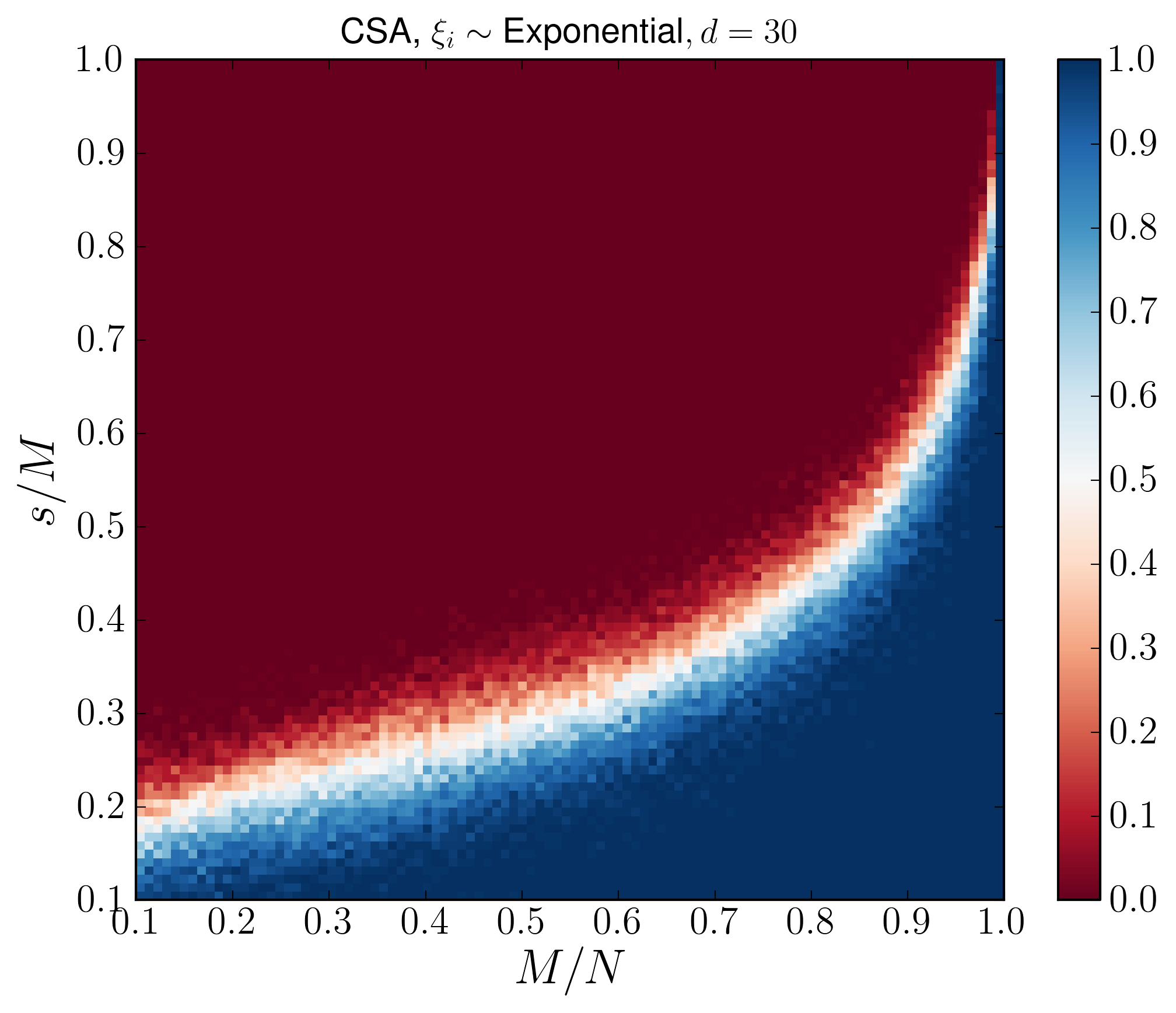}
\end{center}
\caption{Transition plots for exponentially distributed random variables for $d=2$ (top rows) and $d=30$ (bottom row).  The left column corresponds to sampling from the random variable density $\pbwt$, the middle column the CSA method and the right column asymptotic sampling.}
\label{fig:exponential-transition-plots}
\end{figure}  

%

\subsection{Elliptic PDE with random inputs}
In this section, we consider the polynomial approximation of a functional of the solution of the heterogeneous diffusion equation subject to uncertainty in the diffusivity coefficient. This problem has been used as a benchmark in other works~\cite{Hampton_D_JCP_2015,Yang_K_JCP_2013}.

Attention is restricted to one-dimensional physical space to avoid unnecessary complexity, but the procedure described here can easily be extended to higher physical dimensions.
Consider the following problem with $d \ge 1$ random parameters:
\begin{equation}\label{eq:hetrogeneous-diffusion}
-\frac{d}{dx}\left[a(x,\brv)\frac{du}{dx}(x,\brv)\right] = 1,\quad 
(x,\brv)\in(0,1)\times I_\brv
\end{equation}
subject to the physical boundary conditions
\begin{equation*}
u(0,\brv)=0,\quad u(1,\brv)=0.
\end{equation*}
Furthermore, assume that the random log-diffusivity satisfies
\begin{equation*}\label{eq:diffusivityZ}
\log(a(x,\brv))=\bar{a}+\sigma_a\sum_{k=1}^d\sqrt{\gamma_k}\varphi_k(x)\rv_k,
\end{equation*} 
where $\{\gamma_k\}_{k=1}^d$ and $\{\varphi_k(x)\}_{k=1}^d$ are, respectively, 
the eigenvalues and eigenfunctions of the squared exponential covariance kernel 
\[
 C_a(x_1,x_2) = \exp\left[-\frac{(x_1-x_2)^2}{l_c^2}\right].
\]
The variability of the diffusivity field~\eqref{eq:diffusivityZ} is controlled by $\sigma_a$ and the correlation length $l_c$ which determines the decay of the eigenvalues $\gamma_k$. 

In the following we use the CSA method to approximate the a quantity of interest $q$ defined by $q(\brv)=u(1/2,\brv)$ for varying dimension $d$ and random variables $\brv$.  We set $\bar{a}=0.1$, $l_c=1/10$ and vary $\sigma$ with dimension, specifically when $d=2$ we set $\sigma_a=1$ and when $d=20$ we set $\sigma_a=0.017$. The spatial solver for the model~\eqref{eq:hetrogeneous-diffusion} uses spectral Chebyshev collocation with a high enough spatial resolution to neglect discretization errors in our analysis. 

To measure the performance of an approximation, we will use the $\pbwt$-weighted $\ell_2$ error.  Specifically given a set of $Q=10,000$ random samples $\{\brv^{(j)}\}_{j=1}^Q\in\dom$ drawn from the density $\pbwt$ we evaluate the
true function $f(\brv^{(j)})$ and the PCE approximation $\hat{f}(\brv^{(j)})$ and compute
\[
 \varepsilon_{\ell_2(\pbwt)} = \left(\frac{1}{Q}\sum_{j=1}^Q \lvert\hat{f}(\brv^{(j)})-f(\brv^{(j)})\rvert^2\right)^{1/2}
\]

Figures \ref{fig:uniform-dimension-comparison}--\ref{fig:beta-dimension-comparison} compare CSA with standard Monte Carlo probablistic sampling (labeled MC in the plots) when approximating $q$ in low and high-dimensions. In these and all subsequent plots, we show the median and 1st and 3rd quantiles of the $\varepsilon_{\ell_2(w)}$ of 100 random trials. For all the low-dimension ($d=2$) examples provided CSA significantly out-performs probabilistic sampling. A much smaller number of samples are needed to achieve a PCE with a given $\ell_2$ accuracy. In high-dimensions ($d=20$) CSA is competitive with probabilistic sampling when using Legendre polynomials, however MC is more efficient for the other variable types shown.  Again this is due to the fact in our high-dimensional simulations we can only use a low-degree polynomials, which inhibits the effectiveness of the CSA method. These findings are consistent with the behavior observed in the transition plots, Figures~\ref{fig:uniform-transition-plots}-\ref{fig:exponential-transition-plots}.

\begin{figure}[ht]
\begin{center}
\includegraphics[width=\textwidth]{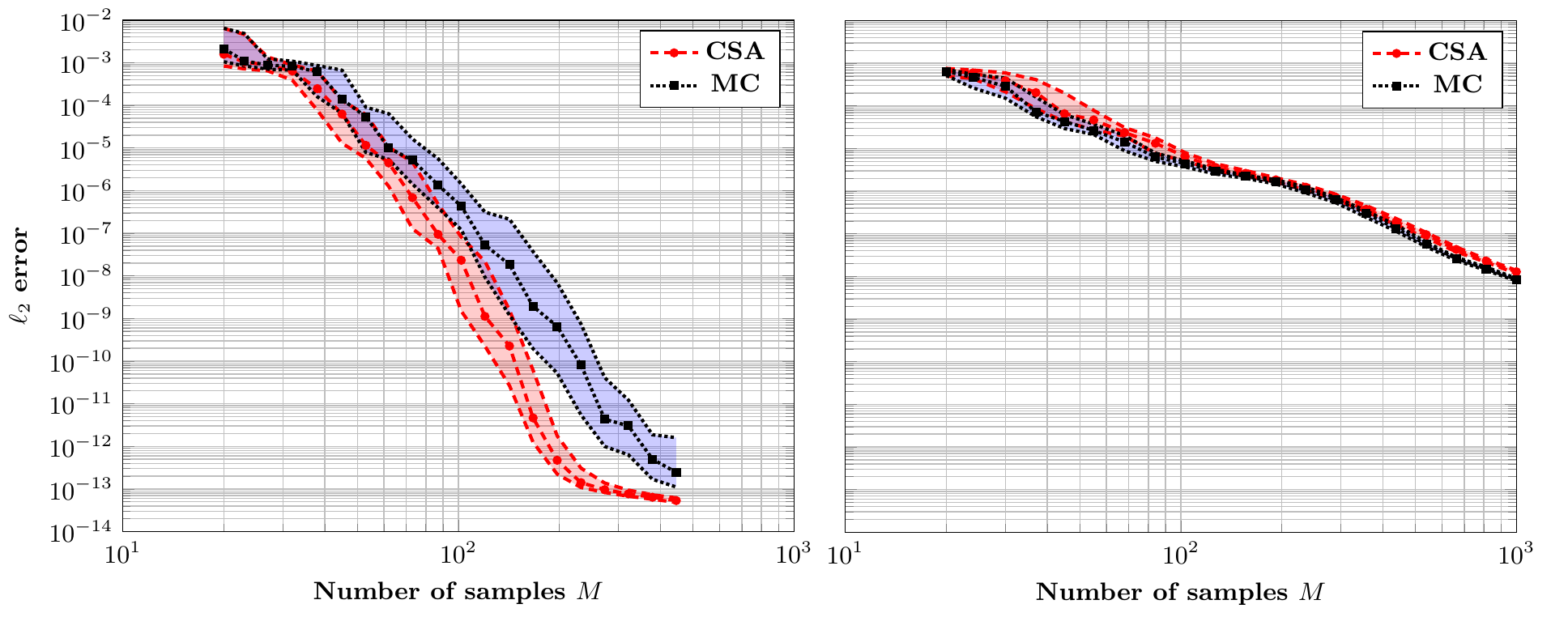}
\end{center}
\caption{The effect of dimension on the convergence of the CSA Legendre-PCE approximation of the diffusion equation~\eqref{eq:hetrogeneous-diffusion}. (Left) 30th degree polynomial in 2 dimensions.  (Right) 4th degree polynomial in 20 dimensions }
\label{fig:uniform-dimension-comparison}
\end{figure}  

\begin{figure}[ht]
\begin{center}
\includegraphics[width=\textwidth]{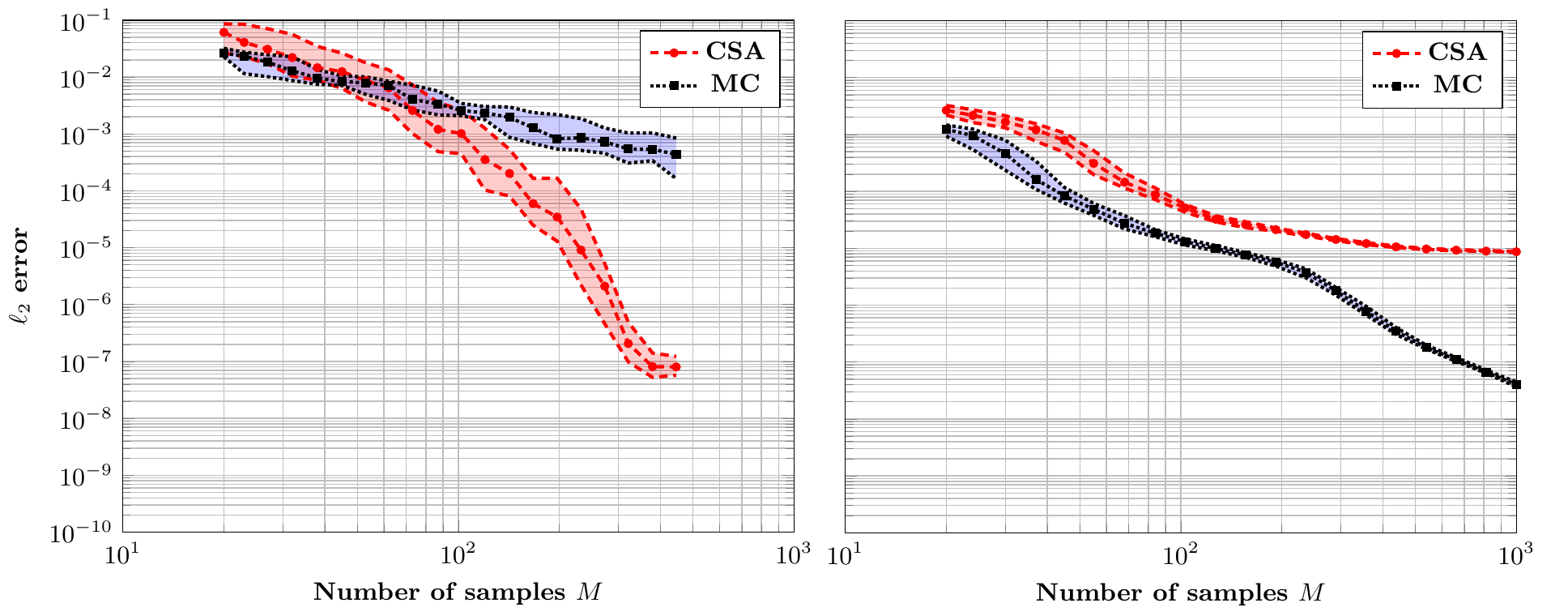}
\end{center}
\caption{The effect of dimension on the convergence of the CSA Hermite-PCE approximation of the diffusion equation~\eqref{eq:hetrogeneous-diffusion}. (Left) 30th degree polynomial in 2 dimensions.  (Right) 4th degree polynomial in 20 dimensions }
\label{fig:gaussian-dimension-comparison}
\end{figure}  

\begin{figure}[ht]
\begin{center}
\includegraphics[width=\textwidth]{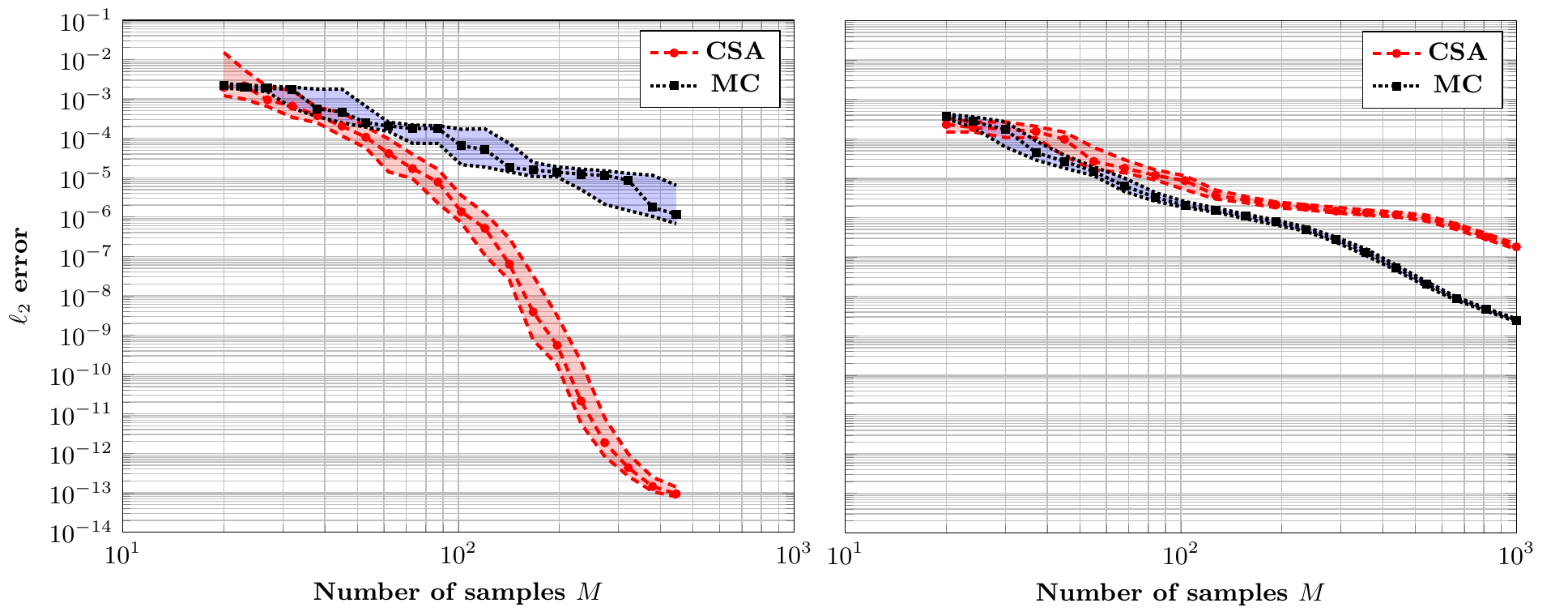}
\end{center}
\caption{The effect of dimension on the convergence of the CSA Jacobi-PCE approximation of the diffusion equation~\eqref{eq:hetrogeneous-diffusion}. (Left) 30th degree polynomial in 2 dimensions.  (Right) 4th degree polynomial in 20 dimensions }
\label{fig:beta-dimension-comparison}
\end{figure}  
Figures \ref{fig:gaussian-method-comparison}--\ref{fig:uniform-beta-method-comparison} compare CSA with the Gaussian asymptotic sampling and Chebyshev sampling methods respectively. In most cases CSA is more accurate than the alternative for a fixed sample size. For the diffusion problem tested the asymptotic sampling method is more accurate than CSA when using Hermite polynomial approximation in twenty dimensions.
\begin{figure}[ht]
\begin{center}
\includegraphics[width=\textwidth]{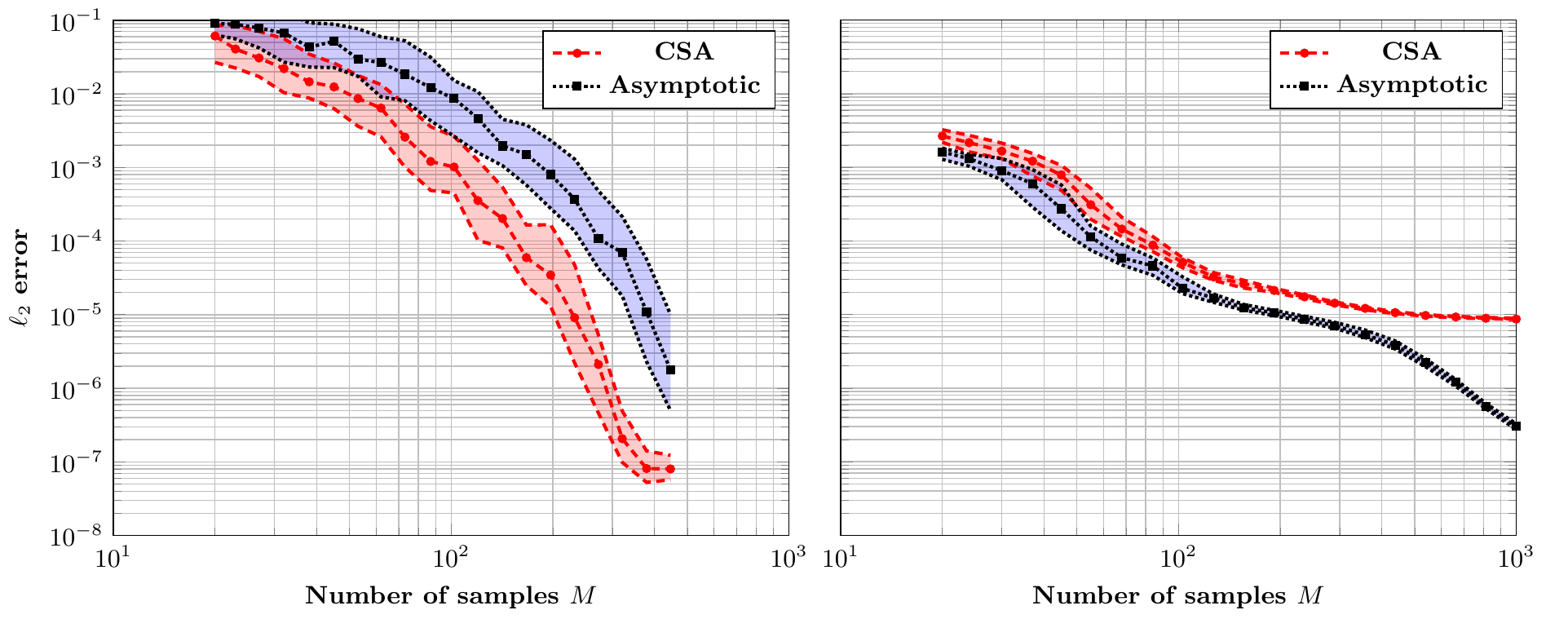}
\end{center}
\caption{Comparison of CSA with asymptotic method for (left) $d=2$ and (right) $d=20$ 
Hermite approximation.}
\label{fig:gaussian-method-comparison}
\end{figure}

\begin{figure}[ht]
\begin{center}
\includegraphics[width=\textwidth]{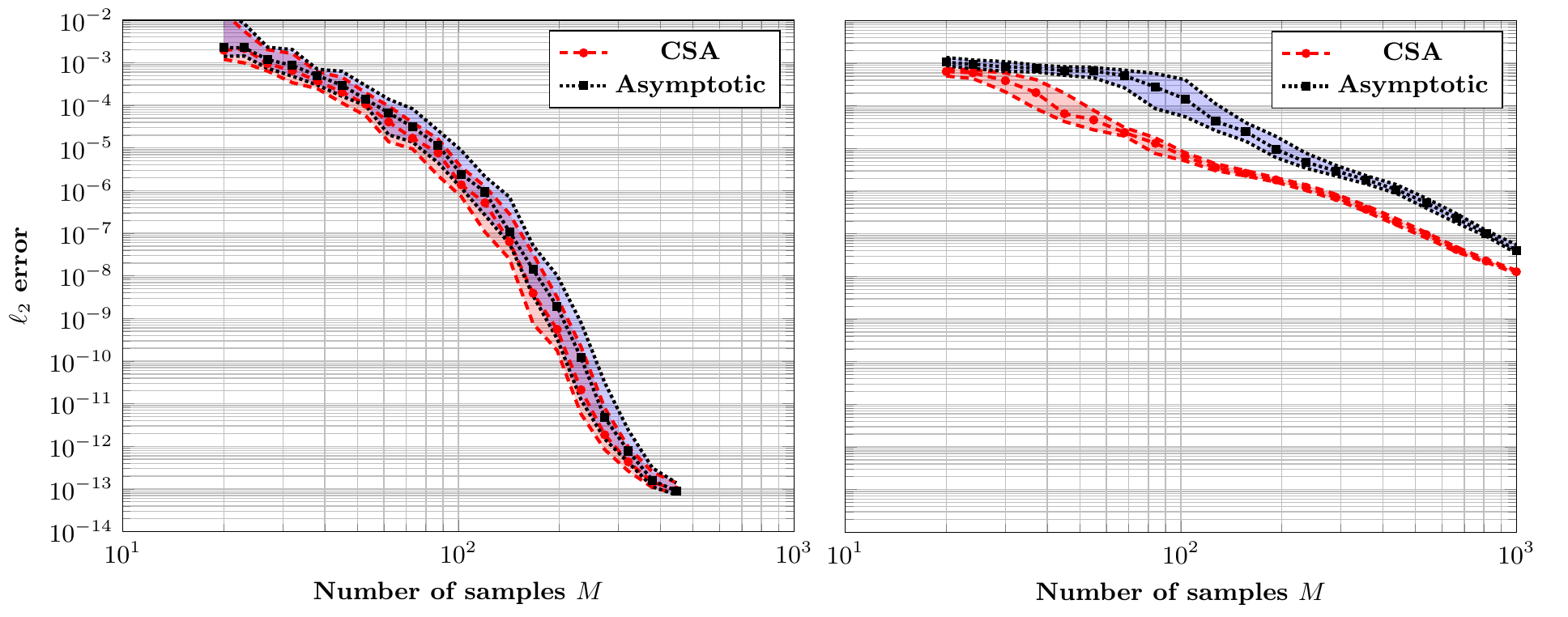}
\end{center}
\caption{Comparison of CSA with asymptotic method for (left) $d=2$ Jacobi approximation and (right)
$d=20$ Legendre approximation.}
\label{fig:uniform-beta-method-comparison}
\end{figure}

%
%

%% file: conclusions.tex
\section{conclusions}\label{sec:conclusions}

Building on the method proposed in \cite{Narayan_JJ_MC}, we propose using
equilibrium-measure-based sampling with Christoffel-function preconditioning to
recover sparse and compressible PCE representations. Unlike most existing algorithms,
the proposed CSA algorithm
can be applied to functions parameterized by random variables that have essentially
any type of standard distribution. Our thoretical and numerical results indicate that the CSA algorithm
is very efficient, that is requires a small number of samples to recover a $s$-sparse 
signal, when the maximum polynomial degree in the dictionary is
large. However convergence does deteriorate in some cases for small polynomial
degrees. 


%% file: proofs.tex
\appendix


\section{Proof of Theorem \ref{thm:bound-csa-a}}\label{sec:appendix-csa-a}

Throughout this section, we assume the CSA-a conditions: We let $\rv$ be a Beta-distributed random variable whose shape parameters $\alpha, \beta \geq -\frac{1}{2}$, with density 
\begin{align*}
  w(z) = (1-z)^{\alpha} (1+z)^{\beta}
\end{align*}
so that the PCE basis elements $\phi_k$ are Jacobi polynomials. We recall that the CSA sampling density $v(z)$ in this case is the Chebyshev density \eqref{eq:chebyshev-density}. We need the Christoffel function associated with this PCE family, whose definition is
\begin{align}\label{eq:christoffel-appendix}
  \lambda_n(\rvd) = \frac{1}{\sum_{j=0}^{n-1} \phi^2_j(x)}.
\end{align}
Associated to a degree-$n$ dictionary, we seek to establish the bound in \eqref{eq:bound-csa-a}, given by
\begin{align*}
  L(n) = \max_{k = 0, \ldots, n} \sup_{z \in [-1,1]} (n+1) \lambda_{n+1} \phi^2_{k} (z) \leq C,
\end{align*}
uniformly in $n$.

Our proof essentially chains together some well-known orthogonal polynomial bounds. In order to prove the result, we require two lemmas concerning the behavior of Jacobi polynomials.

\begin{lemma}[\cite{nevai_generalized_1994}]
  For all $\alpha, \beta \geq -\frac{1}{2}$ then uniformly in $n$ and $\rv \in [-1,1]$,
  \begin{align}\label{eq:nevai-jacobi-bound}
    \sup_{\rvd \in [-1,1]} n \lambda_n(z) \phi^2_{n-1}(\rvd)\leq C,
  \end{align}
  where $C = C(\alpha,\beta) \sim \sqrt{\alpha^2 + \beta^2}$.
\end{lemma}

The following result states that $n \lambda_n(\rv) \approx \frac{w}{v}$ for large $n$.  
\begin{lemma}[\cite{nevai_orthogonal_1980}, Thm 6.3.28]
  Define a regularized version of $w/v$:
  \begin{align}\label{eq:regularized-rho}
    \rho_n(\rvd) = \left(\sqrt{1 - z} + \frac{1}{n} \right)^{2 \alpha+1} \left(\sqrt{1 + z} + \frac{1}{n} \right)^{2\beta + 1},
  \end{align}
  for $n \in \N$. Then there are constants $c_1$ and $c_2$ such that uniformly in $n \in \N$ and $z \in [-1,1]$,
  \begin{align}\label{eq:christoffel-rho}
    c_1 n \lambda_n(\rvd) \leq \rho_n(\rvd) \leq c_2 n \lambda_n(\rvd).
  \end{align}
\end{lemma}

We can now complete the proof of Theorem \ref{thm:bound-csa-a}. If $1 \leq k \leq n$, then from \eqref{eq:regularized-rho},
\begin{align}\label{eq:rho-domination}
  \rho_n(\rvd) \leq \rho_k(\rvd),
\end{align}
since $\alpha + \frac{1}{2} \geq 0$ and $\beta + \frac{1}{2} \geq 0$. Then we have
\begin{align*}
  n \lambda_n(\rvd) &\stackrel{\eqref{eq:christoffel-rho}}{\leq} c_3 \rho_n(\rvd) \\
                        &\stackrel{\eqref{eq:rho-domination}}{\leq} c_4 \rho_k(\rvd) \\
                        &\stackrel{\eqref{eq:christoffel-rho}}{\leq} c_5 k \lambda_k(\rvd).
\end{align*}
Thus, for $1 \leq k \leq n$,
\begin{align*}
  n \lambda_n(z) \left( \phi_{k-1}(\rvd)\right)^2 \leq c_5 k \lambda_k(\rvd) \left( \phi_{k-1}(\rvd)\right)^2 \stackrel{\eqref{eq:nevai-jacobi-bound}}{\leq} c_6,
\end{align*}
for all $z \in [-1,1]$. This proves \eqref{eq:bound-csa-a}.

\section{Proof of Theorem \ref{thm:bound-csa-b}}\label{sec:appendix-csa-b}
For the unbounded case, we use the notation of Section \ref{sec:univariate-sampling-whole}. Let $w(\rvd) = \exp(-|z|^\alpha)$ for $\rvd \in \R$ with $\alpha > 1$, defining $a^W_n(\sqrt{w})$ and $S^W_n(\sqrt{w})$ as in \eqref{eq:MRS-whole} and \eqref{eq:whole-sn}. To remove some notational clutter, we'll omit the `$W$' superscripts, i.e., in this section we write
\begin{align*}
  a^W_n &= a_n, & S^W_n &= S_n.
\end{align*}
The CSA sampling density $v_n(x)$ is defined in \eqref{eq:whole-line-sampling-density}. The $n$th Christoffel function is given by the formula \eqref{eq:christoffel-appendix}, with $\phi_k$ the orthonormal PCE basis associated to $w$. We define a slightly extended version of $S_n$, which depends on specification of some $L \geq 0$:
\begin{align}\label{eq:whole-sn-extended}
  S^\ast_n = \left[ -a_n\left(1 + L \eta_n\right), a_n \left(1 + L \eta_n\right)\right],
\end{align}
with $\eta_n = (\alpha n)^{-2/3}$. Many of the statements we make below present a constant $L \geq 0$, which defines $S^\ast_n$ through \eqref{eq:whole-sn-extended}.

The first result we cite quantifies how fast weighted polynomials outside of the interval $S_n$ decay.
\begin{lemma}[\cite{levin_orthogonal_2001}]\label{lemma:whole-restricted-range}
  With $w(\rvd) = \exp(-|\rvd|^\alpha)$ and $\alpha > 1$, let $r > 1$ be fixed. Then for any polynomial $p$ of degree $n$ or less,
  \begin{align}\label{eq:whole-restricted-range}
    \sup_{x \in \R} p^2(x) w(x) = \sup_{x \in S_n} p^2(x) w(x).
  \end{align}
  Furthermore, there exist constants $c_1, c_2 > 0$ such that, for $\eta > 1$ any real-valued number and for any polynomial $p$ of degree $\lfloor\eta\rfloor$ or less, 
  \begin{align}\label{eq:whole-restricted-range-decay}
    \sup_{\rvd \in \R \backslash S_{r \eta}} \left| p^2(\rvd) w(z) \right| \leq c_1 \exp(-c_2 \eta) \sup_{\rvd \in S_n} \left| p^2(z) w(z) \right|,
  \end{align}
  where $n = \deg p$. The constants $c_1$ and $c_2$ do not depend on $\eta$ or $n$.
\end{lemma}

To proceed further we will need the auxiliary function $\varphi_n$, which is a regularized version of $\left[ n v_n(x) \right]^{-1}$:
\begin{align}\label{eq:varphi-def}
  \varphi_n(x) = \left\{ \begin{array}{rcl}
      \frac{\left(a_{2n}\right)^2 - x^2}{n \sqrt{ \left( \left|x - a_{n}^W\right| + a_{n} \eta_n \right) \left( \left|x + a_{n}^W\right| + a_{n} \eta_n \right)}}, & & x \in S_{n}, \\
      \varphi_n\left(a_{n}\right) , & & x \not\in S_{n},
    \end{array} \right.
\end{align}
The auxiliary function $\varphi_n$ is distinct from the degree-$n$ orthonormal polynomial $\phi_n$. The ``boundary" value of $\varphi_n$ satisfies
\begin{align*}
  \varphi_n\left(a_{n}\right)  = n^{1/\alpha - 2/3} \left[ \left(2^{2/\alpha - 1}\right) a_1 \alpha^{1/3} \left( 2 + \eta_n\right)^{-1/2} \right] \leq c(\alpha) a_n n^{-2/3}
\end{align*}
Note that 
\begin{align}\label{eq:whole-varphi-bound}
  \sup_{x \in \R} \varphi_n(x) \leq c a_n n^{-2/3}
\end{align}
We can now state estimates for Christoffel functions.
\begin{lemma}[Corollary 9.4 of \cite{levin_orthogonal_2001}]\label{lemma:christoffel-whole-bounds}
  Let $w = \exp(-|x|^\alpha)$ on $\R$ with $\alpha > 1$. Let $L > 0$, which defines $S_n^\ast$ through \eqref{eq:whole-sn-extended}. Then 
  \begin{enumerate}[label=\alph*)]
  \begin{subequations}
    \item There are constants $c_1$, $c_2$, such that uniformly in $n$ and $x \in S^\ast_n$,
      \begin{align}\label{eq:christoffel-whole-sim}
        c_1 w(x) \varphi_n(x) \leq \lambda_n(x) \leq c_2 w(x) \varphi_n(x).
      \end{align}
    \item There is a constant $c_3$ such that, uniformly in $n$ and $x \in \R$,
      \begin{align}\label{eq:christoffel-whole-lower}
        c_3 w(x) \varphi_n(x) \leq \lambda_n(x) 
      \end{align}
  \end{subequations}
  \end{enumerate}
\end{lemma}
One can use these estimates to show bounds on orthogonal polynomials.
\begin{lemma}[Theorem 13.2 of \cite{levin_orthogonal_2001}]\label{lemma:poly-whole-bounds}
  Let $w = \exp(-|x|^\alpha)$ on $\R$ with $\alpha > 1$. Then uniformly in $n \geq 1$,
  \begin{align}\label{eq:poly-whole-sup}
    \sup_{x \in \R} \phi_n^2(x) w(x) \leq C \frac{n^{1/3}}{a_n}.
  \end{align}
\end{lemma}
Using the Lemmas above, we can show:
\begin{lemma}\label{lemma:whole-exponential-bound}
  Let $w = \exp(-|z|^\alpha)$ on $\rvd \in \R$ with $\alpha > 1$. Let $L \geq 0$ be fixed, defining $S_n^\ast$.
  \begin{subequations}
  \begin{enumerate}[label=(\alph*)]
  \item Uniformly in $n$ and $0 < k < n$,
    \begin{align}\label{eq:whole-exponential-bound-1}
      \sup_{\rvd \in S_n^\ast} n \lambda_n(z) \phi^2_{k}(z) &\leq C (k n)^{1/3} \left(\frac{n}{k}\right)^{1/\alpha}.
    \end{align}
    When $k=0$, uniformly in $n$,
    \begin{align}\label{eq:whole-exponential-bound-1-k-0}
      \sup_{\rvd \in S_n^\ast} n \lambda_n(z) \phi^2_{0}(z) &\leq C n^{1/3+1/\alpha}.
    \end{align}
  \item For any fixed $0 \leq \delta < 1$, then uniformly in $n$
    \begin{align}\label{eq:whole-exponential-bound-3}
      \max_{\delta n \leq k < n} \sup_{\rvd \in S_n^\ast} n \lambda_n(z) \phi^2_{k}(z) &\leq C n^{2/3}.
    \end{align}
  \item When $\alpha \geq 3$, then uniformly in $n$
    \begin{align}\label{eq:whole-exponential-bound-2}
      \max_{0 \leq k < n} \sup_{\rvd \in S_n^\ast} n \lambda_n(z) \phi^2_{k}(z) &\leq C n^{2/3}.
    \end{align}
  \end{enumerate}
  \end{subequations}
  In all the cases above, $C = C(\alpha,L,\delta)$.
\end{lemma}
\begin{proof}
  Proving \eqref{eq:whole-exponential-bound-1} is a simple chaining of results from Lemmas \ref{lemma:christoffel-whole-bounds} and \ref{lemma:poly-whole-bounds}. For $x \in S^\ast_n$ and $k > 0$,
  \begin{align*}
      n \lambda_n(x) \phi_k^2(x) &\stackrel{\eqref{eq:christoffel-whole-sim}}{\leq} c_1 n \varphi_n(x) w(x) \phi_k^2(x) \\
                                 &\leq c_2 a_n n^{1/3} w(x) \phi_k^2(x) \\
                                 &\stackrel{\eqref{eq:poly-whole-sup}}{\leq} c_3 \left(\frac{a_n}{a_k}\right) n^{1/3} k^{1/3} \\
                                 &\leq c_4 \left(\frac{n}{k}\right)^{1/\alpha} (n k)^{1/3}.
  \end{align*}
  The $k=0$ bound \eqref{eq:whole-exponential-bound-1-k-0} is obtained by repeating the above procedure for the special case with $\phi_0^2(\rvd) \equiv c_5$. The results (b) and (c) are shown by manipulation of the right-hand side of \eqref{eq:whole-exponential-bound-1}. To show (b), we assume $\delta n \leq k < n$, and so 
  \begin{align*}
    \left(\frac{n}{k}\right)^{1/\alpha} (n k)^{1/3} \leq \delta^{-1/\alpha} (n k)^{1/3} \leq c_5 n^{2/3}.
  \end{align*}
  Under the assumption that $\alpha \geq 3$, then (c) follows:
  \begin{align*}
    \left(\frac{n}{k}\right)^{1/\alpha} (n k)^{1/3} \leq \left(\frac{n}{k}\right)^{1/3} (n k)^{1/3} = n^{2/3}.
  \end{align*}
\end{proof}
Note that \eqref{eq:whole-exponential-bound-2} is the conclusion of the theorem, but requires $\alpha \geq 3$. To extend the result to $1 \leq \alpha < 3$, we need a sharper analysis. 
\begin{lemma}\label{lemma:whole-low-degree-bound}
  With $w(z) = \exp(-|z|^\alpha)$ and $\alpha > 1$, let $0 < \delta < 1$ and $L \geq 0$ be fixed. The choice of $L$ defines $S^\ast_n$ through \eqref{eq:whole-sn-extended}. Then uniformly for $0 < k < \delta n$ and $n \geq 1$,
  \begin{subequations}
  \begin{align}\label{eq:whole-low-degree-bound}
    \sup_{\rvd \in S^\ast_n} n \lambda_n(\rvd) \phi_k^2(\rvd) \leq c_1 \left(\frac{n}{k}\right)^{1/\alpha} k^{1/3} 
  \end{align}
  When $k=0$, we have uniformly for $n \geq 1$,
  \begin{align}\label{eq:whole-low-degree-bound-k-0}
    \sup_{\rvd \in S^\ast_n} n \lambda_n(\rvd) \phi_0^2(\rvd) \leq c_1 n^{1/\alpha} 
  \end{align}
  \end{subequations}
\end{lemma}
\begin{proof}
  With $0 < k < \delta n$, choose a $\tau \in (\delta, 1)$, and let $r = \tau/\delta > 1$. With this $\tau$, then uniformly in $n$,
  \begin{align*}
    \sup_{\rvd \in S_{\tau n}} n \varphi_n(\rvd) \leq c_1 a_n.
  \end{align*}
  Thus, uniformly in $n$ and $0 < k < \delta n$,
  \begin{align}
    \nonumber \sup_{\rvd \in S_{\tau n}} n \lambda_n(\rvd) \phi_k^2(\rvd) &\stackrel{\eqref{eq:christoffel-whole-sim}}{\leq} c_1 \sup_{\rvd \in S_{\tau n}} n\varphi_n(\rvd) w(\rvd) \phi_k^2(\rvd) \\
    \nonumber &\leq c_2 a_n \sup_{\rvd \in S_{\tau n}} w(\rvd) \phi_k^2(\rvd) \\
    \label{eq:temp-1}&\stackrel{\eqref{eq:poly-whole-sup}}{\leq} c_2 a_n k^{1/3-1/\alpha}
  \end{align}
  Let $R_n = S^\ast_n \backslash S_{\tau n}$ be the complement of $S_{\tau n}$ in $S^\ast_n$.  When $z \in R_n$, this implies that $|z| > a_{r (\delta n)}$. Then for $z \in R_n$, \eqref{eq:whole-restricted-range-decay} implies, uniformly in $n$,
  \begin{align}
    \nonumber \sup_{\rvd \in R_n} n \lambda_n(\rvd) \phi_k^2(\rvd) &\leq c_1 \sup_{\rvd \in R_n} n\varphi_n(\rvd) w(\rvd) \phi_k^2(\rvd) \\
                                                          \nonumber &\stackrel{\eqref{eq:whole-restricted-range-decay}}{\leq} c_2 \exp(-c_3 \delta n) \left[\sup_{\rvd \in R_n} n \varphi_n(\rvd) \right] \left[\sup_{\rvd \in S_k} w(\rvd) \phi_k^2(\rvd)\right] \\
    \nonumber &\leq c_4 \exp(-\delta n) \frac{a_n}{a_k} n^{1/3}  k^{1/3} \\
  \label{eq:temp-2}&\leq c_5 \exp(-\delta n) n^{2/3+1/\alpha} \leq c_6,
  \end{align}
  Then \eqref{eq:temp-1} and \eqref{eq:temp-2} imply the conclusion \eqref{eq:whole-low-degree-bound}. The bound \eqref{eq:whole-low-degree-bound-k-0} is obtained by repeating the same procedures as above for the specialized case $\phi_0^2(\rvd) \equiv c_7$.
\end{proof}

We can now finish the proof of Theorem \ref{thm:bound-csa-b}.
\begin{proof}[Proof of Theorem \ref{thm:bound-csa-b}]
  Let $L \geq 0$ be as in the assumptions of the theorem. Note that \eqref{eq:whole-exponential-bound-2} is the desired conclusion of Theorem \ref{thm:bound-csa-b} when $\alpha \geq 3$. Therefore, we need only consider $1 < \alpha < 3$. 
  
  Choose some $\delta \in (0,1)$. Consider first $k > 0$. Formula \eqref{eq:whole-low-degree-bound} from Lemma \ref{lemma:whole-low-degree-bound} implies that 
  \begin{align*}
    \max_{0 < k < \delta n} n \lambda_n(\rvd) \phi_k^2(\rvd) \leq c_1 k^{1/3} \left(\frac{n}{k}\right)^{1/\alpha}.
  \end{align*}
  A computation shows that, when $\alpha < 3$,
  \begin{align*}
    k^{1/3} \left(\frac{n}{k}\right)^{1/\alpha} \leq n^{1/\alpha}.
  \end{align*}
  This and \eqref{eq:whole-low-degree-bound-k-0} imply
  \begin{align*}
    \max_{0 \leq k < \delta n} n \lambda_n(\rvd) \phi_k^2(\rvd) \leq c_1 n^{1/\alpha}
  \end{align*}
  When $\delta n \leq k < n$, \eqref{eq:whole-exponential-bound-3} implies
  \begin{align*}
    \max_{\delta n \leq k < n} n \lambda_n(\rvd) \phi_k^2(\rvd) \leq c_1 n^{2/3}
  \end{align*}
  Therefore, 
  \begin{align*}
    \max_{0 \leq k < n} n \lambda_n(\rvd) \phi_k^2(\rvd) \leq c_1 n^{p(\alpha)},
  \end{align*}
  where $p(\alpha) = \max\left\{ \frac{2}{3}, \frac{1}{\alpha}\right\}$ as defined in \eqref{eq:p-definition}. This bound is essentially the conclusion of Theorem \ref{thm:bound-csa-b}, except that it applies to $n \lambda_n \phi_k^2$ instead of $(n+1) \lambda_{n+1} \phi_k^2$. However, since $(n+1)/n$ is uniformly bounded for all $n \geq 1$, the conclusion of the Theorem follows.
\end{proof}

\section{Proof of Theorem \ref{thm:bound-csa-c}}\label{sec:appendix-csa-c}

We now prove essentially the same result for one-sided exponential weights. The strategy is essentially the same as the for two-sided weights in Section \ref{sec:appendix-csa-b}, with some methods reminiscent of those in Section \ref{sec:appendix-csa-a}. However, the requisite results are slightly different. Since there are no new ideas, we simply present the analogues of Lemmas \ref{lemma:whole-restricted-range} through \ref{lemma:whole-low-degree-bound} without proof. 

We use the notation of Section \ref{sec:univariate-sampling-half}. Let $w(z) = \exp(-|z|^\alpha)$ for $z \in [0, \infty)$ with $\alpha > \frac{1}{2}$, defining $a^H_n(\sqrt{w})$ and $S^H_n(\sqrt{w})$ as in \eqref{eq:MRS-half} and \eqref{eq:half-sn}. To remove some notational clutter, we'll omit the `$H$' superscripts, i.e., in this section we write
\begin{align*}
  a^H_n &= a_n, & S^H_n &= S_n.
\end{align*}
The CSA sampling density $v_n(x)$ is defined in \eqref{eq:half-line-sampling-density}. The $n$th Christoffel function is given by the formula \eqref{eq:christoffel-appendix}, with $\phi_k$ the orthonormal PCE basis associated to $w$. We define a slightly extended version of $S_n$, which depends on specification of some $L \geq 0$:
\begin{align}\label{eq:half-sn-extended}
  S^\ast_n = \left[ 0, a_n \left(1 + L \eta_n\right)\right],
\end{align}
with $\eta_n = (\alpha n)^{-2/3}$. Many of the statements we make below present a constant $L \geq 0$, which defines $S^\ast_n$ through \eqref{eq:half-sn-extended}.

The first three Lemmas we reproduce below are cited from \cite{levin_orthogonal_2005} because the notation there is similar to ours. However, these results are essentially known from the earlier work \cite{kasuga_orthonormal_2003}.

We quantify how fast weighted polynomials outside of the interval $S_n$ decay.
\begin{lemma}[\cite{levin_orthogonal_2005}]\label{lemma:half-restricted-range}
  With $w(\rvd) = \exp(-|\rvd|^\alpha)$ and $\alpha > \frac{1}{2}$, let $r > 1$ be fixed. Then for any polynomial $p$ of degree $n$ or less,
  \begin{align}\label{eq:half-restricted-range}
    \sup_{x \in [0, \infty)} p^2(x) w(x) = \sup_{x \in S_n} p^2(x) w(x).
  \end{align}
  Furthermore, there exist constants $c_1, c_2, \nu > 0$ such that, for $\sigma > 1$ any real-valued number and for any polynomial $p$ of degree $\lfloor\sigma\rfloor$ or less, 
  \begin{align}\label{eq:half-restricted-range-decay}
    \sup_{\rvd \in [0, \infty) \backslash S_{r \sigma}} \left| p^2(\rvd) w(z) \right| \leq c_1 \exp(-c_2 \sigma^\nu) \sup_{\rvd \in S_n} \left| p^2(z) w(z) \right|,
  \end{align}
  where $n = \deg p$. The constants $c_1$, $c_2$, and $\nu$ do not depend on $\sigma$ or $n$.
\end{lemma}

To proceed further we will need the auxiliary function $\varphi_n$, which is a regularized version of $\left[ n v_n(\rvd) \right]^{-1}$:
\begin{align}\label{eq:half-varphi-def}
  \varphi_n(\rvd) = \left\{ \begin{array}{rcl}
      \frac{\left(a_{2n} - \rvd\right)\sqrt{\rvd + a_n n^{-2}}}{n \sqrt{ a_n -\rvd + a_n \eta_n}}, & & \rvd \in [0, a_n], \\
      \varphi_n\left(a_{n}\right) , & & \rvd > a_n, \\
      \varphi_n\left(0\right) , & & \rvd < 0, \\
    \end{array} \right.
\end{align}
The auxiliary function notation $\varphi_n$ is distinct from the degree-$n$ orthonormal polynomial $\phi_n$. The ``boundary" values of $\varphi_n$ satisfies
\begin{align*}
  \varphi_n\left(a_{n}\right) &\leq c_1(\alpha) a_n n^{-2/3}, & \varphi_n\left(0\right) &\leq c_2(\alpha) a_n n^{-2}
\end{align*}
Note that $\sup_{x \in [0, \infty)} \varphi_n(x) \leq c a_n n^{-2/3}$. We can now state estimates for Christoffel functions.

\begin{lemma}[\cite{levin_orthogonal_2005}]\label{lemma:christoffel-half-bounds}
Let $w = \exp(-|x|^\alpha)$ on $[0, \infty)$ with $\alpha > \frac{1}{2}$. Let $L > 0$, which defines $S_n^\ast$ through \eqref{eq:half-sn-extended}. Then 
\begin{enumerate}[label=\alph*)]
\begin{subequations}
  \item There are constants $c_1$, $c_2$, such that uniformly in $n$ and $\rvd \in S^\ast_n$,
    \begin{align}\label{eq:christoffel-half-sim}
      c_1 w(\rvd) \varphi_n(\rvd) \leq \lambda_n(\rvd) \leq c_2 w(\rvd) \varphi_n(\rvd).
    \end{align}
  \item There is a constant $c_3$ such that, uniformly in $n$ and $\rvd \in \R$,
    \begin{align}\label{eq:christoffel-half-lower}
      c_3 w(\rvd) \varphi_n(\rvd) \leq \lambda_n(\rvd) 
    \end{align}
\end{subequations}
\end{enumerate}
\end{lemma}

One can use these estimates to show bounds on orthogonal polynomials.
\begin{lemma}[\cite{levin_orthogonal_2005}]\label{lemma:poly-half-bounds}
  Let $w = \exp(-|x|^\alpha)$ on $[0, \infty)$ with $\alpha > \frac{1}{2}$. Let $\beta > 0$ be given and fixed. Then uniformly in $n \geq 1$,
  \begin{subequations}
  \begin{align}\label{eq:poly-half-sup}
    \sup_{\rvd \in [0, \infty)} \phi_n^2(\rvd) w(\rvd) &\leq C \frac{n}{a_n}, \\
    \label{eq:poly-half-sup-upper} \sup_{\rvd \in [a_{\beta n}, \infty)} \phi_n^2(\rvd) w(\rvd) &\leq C \frac{n^{1/3}}{a_n} \\
    \label{eq:poly-half-sup-lower} \sup_{\rvd \in [0, a_{\beta n})} \phi_n^2(\rvd) w(\rvd) \sqrt{\rvd + a_n n^{-2}} &\leq C \frac{1}{a_n}
  \end{align}
  \end{subequations}
\end{lemma}

We make a quantitative estimate on $\varphi_n$ that will be useful later.
\begin{lemma}\label{lemma:half-varphi-estimate}
  Let $\beta \in (0, 1)$ be given and fixed. Then uniformly in $n$,
  \begin{subequations}
  \begin{align}\label{eq:half-varphi-estimate-upper}
    \sup_{\rvd \in [a_{\beta n}, \infty)} \varphi_n(\rvd) \leq c_1 a_n n^{-2/3}.
  \end{align}
  Also, uniformly in $n$ and $\rvd \in S_{\beta n}$,
  \begin{align}\label{eq:half-varphi-estimate-lower}
    \varphi_n(\rvd) \leq c_2 \frac{\sqrt{a_n}}{n} \sqrt{\rvd + a_n n^{-2}}
  \end{align}
  \end{subequations}
\end{lemma}

We can now show the result analogous to Lemma \ref{lemma:whole-exponential-bound}.
\begin{lemma}\label{lemma:half-exponential-bound}
Let $w = \exp(-|z|^\alpha)$ on $\rvd \in [0, \infty)$ with $\alpha > \frac{1}{2}$. Let $L \geq 0$ be fixed, defining $S_n^\ast$.
\begin{subequations}
\begin{enumerate}[label=(\alph*)]
\item Uniformly in $n$ and $0 < k < n$,
  \begin{align}\label{eq:half-exponential-bound-1}
    \sup_{\rvd \in S_n^\ast} n \lambda_n(z) \phi^2_{k}(z) &\leq C (k n)^{1/3} \left(\frac{n}{k}\right)^{1/\alpha}.
  \end{align}
  When $k=0$, uniformly in $n$,
  \begin{align}\label{eq:half-exponential-bound-1-k-0}
    \sup_{\rvd \in S_n^\ast} n \lambda_n(z) \phi^2_{0}(z) &\leq C n^{1/3+1/\alpha}.
  \end{align}
\item For any fixed $0 \leq \delta < 1$, then uniformly in $n$
  \begin{align}\label{eq:half-exponential-bound-3}
    \max_{\delta n \leq k < n} \sup_{\rvd \in S_n^\ast} n \lambda_n(z) \phi^2_{k}(z) &\leq C n^{2/3}.
  \end{align}
\item When $\alpha \geq 3$, then uniformly in $n$
  \begin{align}\label{eq:half-exponential-bound-2}
    \max_{0 \leq k < n} \sup_{\rvd \in S_n^\ast} n \lambda_n(z) \phi^2_{k}(z) &\leq C n^{2/3}.
  \end{align}
\end{enumerate}
\end{subequations}
In all the cases above, $C = C(\alpha,L,\delta)$.
\end{lemma}
\begin{proof}
  To show \eqref{eq:half-exponential-bound-1}, we first have uniformly in $n \geq 1$ and $\rvd \in S^\ast_n$,
  \begin{align}\label{eq:temp-10}
    n \lambda_n(\rvd) \phi_k^2(\rvd) \leq c_1 n \varphi_n(\rvd) w(\rvd) \phi_k^2(\rvd).
  \end{align}
  Now let $\beta \in (0, 1)$ be fixed. For $\rvd \in S_{\beta n}$, we have
  \begin{align}
    \nonumber \sup_{\rvd \in S_{\beta n}} n \varphi_n(\rvd) w(\rvd) \phi_k^2(\rvd) &\stackrel{\eqref{eq:half-varphi-estimate-lower}}{\leq} c_1 \sqrt{a_n} \sup_{\rvd \in S_{\beta n}} \sqrt{\rvd + a_n n^{-2}} w(\rvd) \phi_k^2(\rvd) \\
                                                                         \nonumber &\stackrel{\eqref{eq:poly-half-sup-lower}}{\leq} c_2 \frac{1}{\sqrt{a_n}}  \\
                                                                         \label{eq:temp-11}&\leq c_3 n^{-1/\alpha} \leq c_4
  \end{align}
  With $z \in R_n = S^\ast_n\backslash S_{\beta n}$, we have
  \begin{align}
    \nonumber \sup_{\rvd \in R_n} n \varphi_n(\rvd) w(\rvd) \phi_k^2(\rvd) &\stackrel{\eqref{eq:half-varphi-estimate-upper}}{\leq} c_5 a_n n^{1/3} \sup_{\rvd \in R_n} w(\rvd) \phi_k^2(\rvd) \\
                                                                 \nonumber &\stackrel{\eqref{eq:poly-half-sup-upper}}{\leq} c_6 a_n n^{1/3} \frac{k^{1/3}}{a_k} \\
                                                                 \label{eq:temp-12}&\leq c_7 (k n)^{1/3} \left(\frac{n}{k}\right)^{1/\alpha}
  \end{align}
  Combining \eqref{eq:temp-10}, \eqref{eq:temp-11}, and \eqref{eq:temp-12} proves \eqref{eq:half-exponential-bound-1}. To prove the bound in \eqref{eq:half-exponential-bound-1-k-0}, all the above arguments can be specialized to $\phi^2_0(\rvd) \equiv c_8$. Finally, the right-hand side of \eqref{eq:half-exponential-bound-1} can be manipulated to show the results \eqref{eq:half-exponential-bound-3} and \eqref{eq:half-exponential-bound-2}. We omit these latter proofs as the mechanics are identical to the proofs of \eqref{eq:whole-exponential-bound-3} and \eqref{eq:whole-exponential-bound-2}.
\end{proof}
To extend our proof to $\alpha < 3$, we provide the following analogue of Lemma \ref{lemma:whole-low-degree-bound}.
\begin{lemma}\label{lemma:half-low-degree-bound}
  With $w(z) = \exp(-|z|^\alpha)$ on $z \in [0, \infty)$ and $\alpha > \frac{1}{2}$, let $0 < \delta < 1$ and $L \geq 0$ be fixed. The choice of $L$ defines $S^\ast_n$ through \eqref{eq:half-sn-extended}. Then uniformly for $0 < k < \delta n$ and $n \geq 1$,
  \begin{subequations}
  \begin{align}\label{eq:half-low-degree-bound}
    \sup_{\rvd \in S^\ast_n} n \lambda_n(\rvd) \phi_k^2(\rvd) \leq c_1 n^{1/2\alpha} k^{1/3-1/\alpha} 
  \end{align}
  When $k=0$, we have uniformly for $n \geq 1$,
  \begin{align}\label{eq:half-low-degree-bound-k-0}
    \sup_{\rvd \in S^\ast_n} n \lambda_n(\rvd) \phi_0^2(\rvd) \leq c_1 n^{1/2\alpha} 
  \end{align}
  \end{subequations}
\end{lemma}
\begin{proof}
  With $0 < k < \delta n$, choose a $\beta \in (\delta, 1)$, and let $r = \beta/\delta > 1$. With this $\beta$, then uniformly in $n$ and $\rvd \in S_{\beta n}$,
  \begin{align*}
    n \varphi_n(\rvd) \stackrel{\eqref{eq:half-varphi-estimate-lower}}{\leq} c_1 \sqrt{a_n} \sqrt{z + a_n n^{-2}}
  \end{align*}
  Thus, uniformly in $n$ and $0 < k < \delta n$,
  \begin{align}
    \nonumber \sup_{\rvd \in S_{\beta n}} n \lambda_n(\rvd) \phi_k^2(\rvd) &\stackrel{\eqref{eq:christoffel-half-sim}}{\leq} c_1 \sup_{\rvd \in S_{\beta n}} n\varphi_n(\rvd) w(\rvd) \phi_k^2(\rvd) \\
    \nonumber &\stackrel{\eqref{eq:half-varphi-estimate-lower}}{\leq} c_2 \sqrt{a_n} \sup_{\rvd \in S_{\beta n}} w(\rvd) \sqrt{z + a_n n^{-2}} \phi_k^2(\rvd) \\
    \nonumber &\leq c_2 \sqrt{a_n} \sup_{\rvd \in S_{\beta n}} w(\rvd) \sqrt{z + a_k k^{-2}} \phi_k^2(\rvd) \\
    \label{eq:temp-20}&\stackrel{\eqref{eq:poly-half-sup-lower}}{\leq} c_2 \sqrt{a_n} k^{1/3-1/\alpha} \leq c_3 n^{1/2\alpha} k^{1/3-1/\alpha}
  \end{align}
  Let $R_n = S^\ast_n \backslash S_{\beta n}$ be the complement of $S_{\beta n}$ in $S^\ast_n$.  When $z \in R_n$, this implies that $z > a_{r (\delta n)}$. Then for $z \in R_n$, \eqref{eq:half-restricted-range-decay} implies, uniformly in $n$,
  \begin{align}
    \nonumber \sup_{\rvd \in R_n} n \lambda_n(\rvd) \phi_k^2(\rvd) &\stackrel{\eqref{eq:christoffel-half-sim}}{\leq} c_1 \sup_{\rvd \in R_n} n\varphi_n(\rvd) w(\rvd) \phi_k^2(\rvd) \\
                                                          \nonumber &\stackrel{\eqref{eq:half-restricted-range-decay}}{\leq} c_2 \exp(-c_3 \delta n) \left[\sup_{\rvd \in R_n} n \varphi_n(\rvd) \right] \left[\sup_{\rvd \in S_k} w(\rvd) \phi_k^2(\rvd)\right] \\
    \nonumber &\stackrel{\eqref{eq:half-varphi-estimate-upper},\eqref{eq:poly-half-sup}}{\leq} c_4 \exp(-\delta n) \frac{a_n}{a_k} n^{1/3} k \\
    \label{eq:temp-21}&\leq c_5 \exp(-\delta n) n^{4/3+1/\alpha} \leq c_6,
  \end{align}
  Then \eqref{eq:temp-20} and \eqref{eq:temp-21} imply the conclusion \eqref{eq:half-low-degree-bound}. The bound \eqref{eq:half-low-degree-bound-k-0} is obtained by repeating the same procedures as above for the specialized case $\phi_0^2(\rvd) \equiv c_7$.
\end{proof}

We can now finish the proof of Theorem \ref{thm:bound-csa-c}.
\begin{proof}[Proof of Theorem \ref{thm:bound-csa-c}]
  Let $L \geq 0$ be as in the assumptions of the theorem. Note that \eqref{eq:half-exponential-bound-2} is the desired conclusion of Theorem \ref{thm:bound-csa-c} when $\alpha \geq 3$. Therefore, we need only consider $\frac{1}{2} < \alpha < 3$. 
  
  Choose some $\delta \in (0,1)$. Consider first $k > 0$. Formula \eqref{eq:half-low-degree-bound} from Lemma \ref{lemma:half-low-degree-bound} implies that 
  \begin{align*}
    \max_{0 < k < \delta n} n \lambda_n(\rvd) \phi_k^2(\rvd) \leq c_1 n^{1/2\alpha} k^{1/3-1/\alpha}
  \end{align*}
  When $\alpha < 3$,
  \begin{align*}
    n^{1/2\alpha} k^{1/3-1/\alpha} < n^{1/2\alpha}
  \end{align*}
  This and \eqref{eq:half-low-degree-bound-k-0} imply
  \begin{align*}
    \max_{0 \leq k < \delta n} n \lambda_n(\rvd) \phi_k^2(\rvd) \leq c_1 n^{1/2\alpha}
  \end{align*}
  When $\delta n \leq k < n$, \eqref{eq:half-exponential-bound-3} implies
  \begin{align*}
    \max_{\delta n \leq k < n} n \lambda_n(\rvd) \phi_k^2(\rvd) \leq c_1 n^{2/3}
  \end{align*}
  Therefore, 
  \begin{align*}
    \max_{0 \leq k < n} n \lambda_n(\rvd) \phi_k^2(\rvd) \leq c_1 n^{p(2\alpha)},
  \end{align*}
  where $p(2\alpha) = \max\left\{ \frac{2}{3}, \frac{1}{2\alpha}\right\}$ as defined in \eqref{eq:p2-definition}.
\end{proof}

\section{Proof of Theorem \ref{thm:csa-convergence}}\label{sec:appendix-csa-convergence}

With $\phi_k(\rvd)$ the PCE basis corresponding to the distribution of $Z$, let $p(\rvd) = \sum_{k=0}^n \alpha_{k+1} \phi_k(\rvd)$ be an arbitrary degree-$n$ polynomial defined by the expansion coefficients $\coef$. 
The iid samples $Z_1, \ldots, Z_M$ are distributed according to the density $v_n$. The data vector $\boldf$ contains evaluations of $p$ polluted by noise $\bseta$, $f_m = p(Z_m) + \eta_m$, with $\bseta$ satisfying $\left\| \wtmat \bseta \right\|_\infty \leq \stoptol$.

The deviation of the Gramian $\bR$ from the identity indicates how much the original PCE basis is non-orthogonal with respect to the augmented weight $(N \lambda_N) v_n$. We can define a new set of basis elements, $\psi_k$ for $k=0, \ldots, N-1$, which are orthonormal under this augmented weight:
\begin{align*}
  \psi_k(\rvd) &= \sum_{j=1}^N (S)_{k,j} \phi_k(\rvd), & \bS = \bR^{-1/2}.
\end{align*}
I.e., we have
\begin{align*}
  \int_{S_n} \psi_k(\rvd) \psi_\ell(\rvd) (N \lambda_N(\rvd)) v_n(\rvd) \dx{\rvd} = \delta_{k,\ell}.
\end{align*}
By this construction, the basis $\left\{ \psi_k(\rvd) \sqrt{N \lambda_N(\rvd)} \right\}_k$ is orthonormal with respect to the density $v_n$. We can recast the original minimization problem \eqref{eq:csa-weighted-minimization} as one operating on this orthonormal basis, and so we require a bound for these functions. We have, uniformly in $\rvd \in S_n$, $n \in \N$, and $0 \leq k \leq n \geq 1$,
\begin{align*}
  \left\| \sqrt{N \lambda_N(\rvd)} \psi_k(\rvd) \right\|^2 &= \left\| \sum_{j=1}^N S_{k,j} \sqrt{N \lambda_N(\rvd)} \phi_k(\rvd) \right\|^2 \\
                                                           &\leq \left\|\bS\right\|_{\infty}^2 \left\| \sqrt{N \lambda_N(\rvd)} \phi_k(\rvd) \right\|^2 \\
                                                           &\leq \left\|\bR^{-1/2} \right\|_{\infty}^2 L(n),
\end{align*}
where $L(n)$ is the bound on Christoffel-weighted polynomials given in Theorems \ref{thm:bound-csa}.

The arbitrary polynomial $p$ can be expressed as an expansion in the $\psi_k$ via the coefficients 
\begin{align*}
  \bsbeta &= \V{R}^{1/2} \coef
\end{align*}
The Vandermonde-like matrix corresponding to the evaluations of the basis $\psi_k$ at the samples $Z_m$ is given by $\bsPsi \triangleq \vand \bS$. 

By Corollary 12.34 in \cite{rauhut_mathematical_2013} (or Theorem 4.4 in \cite{Rahut_TFNMSR_2010}), then under the assumed sampling condition \eqref{eq:sample-count}, the solution $\bsbeta^\star$ to 
\begin{align*}
  \bsbeta^\star &= \argmin_{\bsc}\; \left\| \bsc\right\|_{1} \quad \text{such that}\quad
  \|\sqrt{\bW}\bsPsi \bsc - \sqrt{\bW}\boldf\|_2 \le \stoptol
\end{align*}
satisfies the estimates
\begin{align*}
  \left\| \bsbeta - \bsbeta^\ast \right\|_2 &\leq \frac{C_1 \sigma_{s,1}(\bsbeta)}{\sqrt{s}} + C_2 \stoptol \\
  \left\| \bsbeta - \bsbeta^\ast \right\|_1 &\leq \frac{C_1 \sigma_{s,1}(\bsbeta)}{\sqrt{s}} + C_2 \sqrt{s} \stoptol.
\end{align*}
The above actually just our convergence result in disguise: we make the identification $\bsbeta = \bR^{1/2} \coef$ and $\bsPsi = \vand \bS = \vand \bR^{-1/2}$ to obtain 
\begin{align*}
  \bR^{1/2} \bsalpha^\star &= \argmin_{\bsc}\; \left\| \bR^{1/2} \bsc \right\|_{1} \quad \text{such that}\quad
  \|\sqrt{\bW}\vand \bsc - \sqrt{\bW}\boldf\|_2 \le \stoptol,
\end{align*}
which is the same as \eqref{eq:csa-weighted-minimization}. Then the discrepancy between $\coef$ and $\coef^\star$ satisfies
\begin{align*}
  \lambda_{\mathrm{min}} \left(\bR^{1/2}\right) \left\| \coef - \coef^\star \right\|_2 &\leq \left\| \bR^{1/2} \coef - \bR^{1/2} \coef^\star \right\|_2 \leq \frac{C_1 \sigma_{s,1}(\bsbeta)}{\sqrt{s}} + C_2 \stoptol, \\
  \frac{1}{\left\|\bR^{-1/2}\right\|_\infty} \left\| \coef - \coef^\star \right\|_1 &\leq \left\| \bR^{1/2} \coef - \bR^{1/2} \coef^\star \right\|_1 \leq C_1 \sigma_{s,1}(\bsbeta) + C_2 \sqrt{s} \stoptol.
\end{align*}
The second inequality chain above uses the facts (i) that for a non-singular matrix $\bs{A}$ and any vector $\bs{x}$, then $\left\| \bs{A} \bs{x} \right\| \geq \left\|\bs{A}^{-1}\right\|^{-1} \left\|\bs{x}\right\|$ for $\left\|\bs{x}\right\|$ any vector norm and $\left\|\bs{A}\right\|$ any corresponding sub-multiplicative norm, and (ii) that $\left\|\bs{A}\right\|_1 = \left\|\bs{A}\right\|_\infty$ for any symmetric matrix $\bs{A}$. The two statements above are exactly \eqref{eq:csa-convergence}.
